\documentclass{amsart}

\RequirePackage{fix-cm}
\usepackage{graphicx}
\usepackage[all]{xy}
\usepackage{latexsym}
\usepackage{amsmath}
\usepackage{amsthm}
\usepackage[T1]{fontenc}
\usepackage{amssymb,amscd}
\usepackage{setspace}
\usepackage{mathdots}
\usepackage[mathscr]{eucal}
\usepackage{bbm}
\usepackage{stmaryrd}
\usepackage{color}
\usepackage{mathrsfs}
\usepackage{hyperref}
\usepackage{cite}
\makeatletter
\def\@thmcountersep{.}
\makeatother

\newcommand{\alg}{\mathrm{alg}}
\newcommand{\ps}{\mathrm{ps}}

\newcommand{\ver}{\mathrm{ver}}
\newcommand{\nn}{\mathfrak n}
\newcommand{\fa}{\mathfrak{a}}
\newcommand{\p}{\mathfrak{p}}
\newcommand{\m}{\mathfrak{m}}
\newcommand{\n}{\mathfrak{n}}
\newcommand{\md}{\mathrm{m}}
\newcommand{\cO}{\mathcal{O}}

\newcommand{\F}{\mathbb F}

\newcommand{\Q}{\mathbb Q}

\newcommand{\Z}{\mathbb Z}

\newcommand{\A}{\mathbb A}

\newcommand{\w}{\mathbf{w}}
\newcommand{\cont}{\mathrm{cont}}
\newcommand{\OK}{\cO\llbracket K\rrbracket}
\newcommand{\br}[1]{\llbracket #1\rrbracket}
\newcommand{\CO}{\mathfrak{C}(\cO)}

\newcommand{\VV}{\mathbf{V}} 
\newcommand{\qq}{\mathfrak q}
\newcommand{\inv}{\mathrm{inv}}

\newcommand{\rsoc}{\mathrm{soc}}

\newcommand{\JH}{\mathrm{JH}}
\newcommand{\Ker}{\mathrm{Ker}}

\newcommand{\Fil}{\mathrm{Fil}}

\newcommand{\Sp}{\mathrm{Sp}}
\DeclareMathOperator{\Spec}{Spec}
\DeclareMathOperator{\mSpec}{m-Spec}
\newcommand{\pp}{\mathfrak p}
\newcommand{\ra}{\rightarrow}

\newcommand{\End}{\mathrm{End}}
\newcommand{\Hom}{\mathrm{Hom}}
\newcommand{\Ext}{\mathrm{Ext}}
\newcommand{\GL}{\mathrm{GL}}
\newcommand{\SL}{\mathrm{SL}}

\newcommand{\Gal}{\mathrm{Gal}}
\newcommand{\id}{\mathrm{id}}
\newcommand{\bFp}{\overline{\F}_p}
\newcommand{\bQp}{\overline{\Q}_p}

\newcommand{\Sym}{\mathrm{Sym}}

\DeclareMathOperator{\Ind}{Ind}

\DeclareMathOperator{\cInd}{c-Ind}

\DeclareMathOperator{\Mod}{Mod}
\newcommand{\rhobar}{\bar{\rho}}
\newcommand{\cyc}{\mathrm{cyc}}
\newcommand{\Eins}{\mathbf 1}
\newcommand{\wt}{\mathbf{w}}
\newcommand{\Zp}{\mathbb{Z}_p}
\newcommand{\dualcat}{\mathfrak C}
\newcommand{\cri}{\mathrm{cr}}
\newcommand{\ide}{\mathbf{1}}
\newcommand{\mm}{\mathfrak m}
\newcommand{\plim}{\varprojlim}

\newcommand{\sm}{\mathrm{sm}}
\DeclareMathOperator{\tr}{\mathrm{tr}}

\newcommand{\xto}[1][]{\xrightarrow{#1}}
\newcommand{\simto}{
\xto[\sim]} 

\newcommand{\matr}[4]{\begin{pmatrix}{#1}&{#2}\\ {#3}&{#4}\end{pmatrix}}
\newcommand{\smatr}[4]{\bigl(\begin{smallmatrix} {#1}& {#2}\\ {#3}&{#4}\end{smallmatrix}\bigl)}

\newcommand{\Qp}{\mathbb{Q}_p}
\newcommand{\univ}{\mathrm{univ}}
\newcommand{\st}{\mathrm{st}}

\newcommand{\OO}{\mathcal O}

\newcommand{\wtimes}{\widehat{\otimes}}
\DeclareMathOperator{\ad}{ad}
\DeclareMathOperator{\Frob}{Frob}
\newcommand{\cV}{\check{\mathbf V}}
\newcommand{\pro}{\mathrm{pro}}
\newcommand{\AfF}{\mathbb{A}^{\infty}_F}
 
\newcommand{\tf}{\mathrm{tf}}

\newcommand{\twomat}[4]{\begin{pmatrix} #1 & #2 \\ #3 & #4 \end{pmatrix}}

\renewcommand{\bar}{\overline}
\usepackage{mathtools}

\newcommand{\GG}{\mathbb{G}}

\newcommand{\QQ}{\mathbb{Q}}

\newcommand{\Ac}{\mathcal{A}}

\newcommand{\Cc}{\mathcal{C}}

\newcommand{\Oc}{\mathcal{O}}

\newcommand{\Sc}{\mathcal{S}}

\newcommand{\Xc}{\mathcal{X}}

\newcommand{\mf}{\mathfrak{m}}

\newcommand{\pf}{\mathfrak{p}}

\newcommand{\qf}{\mathfrak{q}}

\newcommand{\rarrow}{\rightarrow}

\newcommand{\onto}{\twoheadrightarrow}

\let\Thet=\Theta
\def\Theta{\mathrm{\Thet}}

\let\Pii=\Pi
\def\Pi{\mathrm{\Pii}}

\let\Sigm=\Sigma
\def\Sigma{\mathrm{\Sigm}}

\let\Xii=\Xi
\def\Xi{\mathrm{\Xii}}

\let\Upsilo=\Upsilon
\def\Upsilon{\mathrm{\Upsilo}}

\let\Phii=\Phi
\def\Phi{\mathrm{\Phii}}

\let\Omeg=\Omega
\def\Omega{\mathrm{\Omeg}}

\newtheorem{thm}{Theorem}[section]
\newtheorem{lem}[thm]{Lemma}
\newtheorem{cor}[thm]{Corollary}
\newtheorem{prop}[thm]{Proposition}
\newtheorem{defn}[thm]{Definition}

\newtheorem{rem}[thm]{Remark}

\newtheorem{ques}[thm]{Question}
 
\begin{document}

\title[Crystabelline deformation rings]{On  crystabelline deformation rings of $\Gal(\bQp/\Q_p)$ (with an appendix by Jack Shotton) 
}


\author{Yongquan Hu \and Vytautas Pa\v{s}k\={u}nas
}


\maketitle
\begin{abstract}
We prove that certain crystabelline deformation rings of two dimensional residual representations of $\Gal(\bQp/\Q_p)$ are Cohen--Macaulay. As a consequence, this allows to improve Kisin's $R[1/p]=\mathbb{T}[1/p]$ theorem to an $R=\mathbb{T}$ theorem.
\end{abstract}

\section{Introduction}

Let $p$ be a prime, $\Q_p$ the field of $p$-adic numbers, $L$ a finite extension of $\Q_p$ with the ring of integers $\cO$, uniformizer $\varpi$ and residue field $k$. Let $\rho:G_{\Q_p}=\Gal(\bQp/\Q_p)\ra \GL_2(k)$ be a continuous representation such that $\End_{G_{\Q_p}}(\rho)\cong k$. After Mazur \cite{maz}, there exists a universal deformation ring $R_{\rho}^{\rm un}$, together with a universal deformation $\rho^{\rm un}:G_{\Q_p}\ra \GL_2(R_{\rho}^{\rm un})$, which parametrizes all deformations of ${\rho}$ to local artinian $\cO$-algebras whose residue field is isomorphic to $k$. If $\n\in\mSpec R^{\rm un}_{\rho}[1/p]$, where $\mSpec$ denotes the set of maximal ideals, then $\kappa(\n):=R^{\rm un}_{\rho}[1/p]/\n$ is a finite extension of $L$ and we get a genuine $p$-adic representation $\rho_{\n}^{\rm un}:G_{\Q_p}\ra \GL_2(\kappa(\n))$ via specialization.
Fix a $p$-adic Hodge type $(\w,\tau,\psi)$, where $\w=(a,b)$ is a pair of integers with $b>a$, $\tau:I_{\Q_p}\ra \GL_2(L)$ is a smooth representation of the inertia subgroup, and $\psi:G_{\Q_p}\ra \cO^{\times}$ is a continuous character such that $\psi \epsilon\equiv \det \rho\ (\mathrm{mod}\  \varpi)$, where $\epsilon$ is the cyclotomic character.  We say $\rho_{\n}^{\rm un}$ is of type $(\w,\tau,\psi)$ if it is potentially semi-stable of Hodge-Tate weights $\w$, its determinant is equal to $\psi\epsilon$ and $\mathrm{WD}(\rho_{\n}^{\rm un})|_{I_{\Q_p}}\cong \tau$, where   $\mathrm{WD}(\rho_{\n}^{\rm un})$ denotes the Weil-Deligne representation associated to $\rho_{\n}^{\rm un}$ by Fontaine (see \cite{fo}). 

We will denote by $R^{\psi}_{\rho}$ the quotient of $R_{\rho}^{\rm un}$, which parameterizes deformations with determinant $\psi\epsilon$. By a general theorem of Kisin \cite{ki08}, there is a unique
reduced $\cO$-flat quotient of $R^{\psi}_{\rho}$, denoted by $R^{\psi}_{\rho}(\w,\tau)$ (resp. $R^{\psi,\rm cr}_{\rho}(\w,\tau)$), which parametrizes all deformations of $\rho$ which are potentially semi-stable (resp. potentially crystalline)   of type $(\w,\tau,\psi)$. By this, we mean that a  point $\n\in\mSpec R^{\rm un}_{\rho}[1/p]$ lies in $\mSpec R^{\psi}_{\rho}(\w,\tau)[1/p]$ (resp. $\mSpec R^{\psi,\rm cr}_{\rho}(\w,\tau)[1/p]$) if and only if   $\rho_{\n}^{\rm un}$ is  potentially semi-stable (resp. potentially crystalline) of type $(\w,\tau,\psi)$. These deformation rings have played a crucial role in proving modularity lifting theorems (see for example \cite{ki09}, \cite{kw1}, \cite{kw2}).

In a recent paper \cite{pa12}, the second author has given another construction of the rings $R^{\psi}_{\rho}(\w,\tau)$ and $R^{\psi,\rm cr}_{\rho}(\w,\tau)$, based on the $p$-adic local Langlands correspondence for $\GL_2(\Q_p)$ and his previous work \cite{pa10}.
In this paper, using results and techniques developed in \emph{loc. cit.},  we prove the  following theorem about the structure of $R^{\psi,\rm cr}_{\rho}(\w,\tau)$.
\begin{thm}\label{theorem-main-intro}
Assume that $\End_{G_{\Qp}}(\rho)=k$ and if $p=3$ assume further that $\rho\not\sim \bigl ( \begin{smallmatrix}\chi \omega & \ast \\ 0 & \chi \end{smallmatrix}\bigr )$ 
for any character $\chi: G_{\Qp}\rightarrow k^{\times}$, where $\omega$ is the mod $p$ cyclotomic character.
 Then, for any $p$-adic Hodge type $(\w,\theta_1 \oplus \theta_2,\psi)$,  where $\theta_1, \theta_2: I_{\Qp}\rightarrow L^{\times}$
are \textbf{distinct} 
characters with open kernel, which extend to $W_{\Qp}$,  the potentially crystalline deformation ring $R^{\psi,\rm cr}_{\rho}(\w, \theta_1\oplus \theta_2)$ is Cohen-Macaulay. 
\end{thm}

\begin{rem} The assumption $\theta_1\neq \theta_2$ implies that the potentially semi-stable deformation ring $R^{\psi}_{\rho}(\w,\tau)$ and the 
potentially crystalline deformation ring $R^{\psi,\rm cr}_{\rho}(\w, \tau)$ coincide. 
\end{rem} 

We will refer to the inertial types $\tau$ which are direct sum of two characters as above as principal series types. If $\rho_{\n}^{\rm un}$ is 
potentially semi-stable of type  $(\w,\tau,\psi)$ then $\tau$ is a principal series type if and only if $\rho_{\n}^{\rm un}$ becomes a crystalline representation when restricted to an abelian extension of $\Q_p$.  These representations are named \emph{crystabelline} in \cite{bb},  thus  explains the title of the paper.

Let us explain how the assumption on $\tau$ gets used in the proof of Theorem \ref{theorem-main-intro}.   An inertial type $\tau$ singles out a Bernstein component $\Omega$ in the category of smooth $\overline{L}$-representations of $\GL_2(\Qp)$. If $\tau=\theta_1 \oplus \theta_2$ 
with $\theta_1\neq \theta_2$ then all the irreducible representations in $\Omega$ are principal series. In the course of the proof we use 
a result of Berger--Breuil \cite{bb} that if $\pi$ is a smooth irreducible principal series representation and $\pi \otimes \Sym^{b-a-1} L^2\otimes \det^a$ 
admits a $G$-invariant norm then the completion is an admissible irreducible Banach space representation of $G$ and one can control 
the $G$-representation obtained by reducing the unit ball modulo $\varpi$. This fails for all the other inertial types since $\pi\otimes \Sym^{b-a-1} L^2 \otimes
\det^a$ with a unitary central character will admit infinitely many non-isomorphic completions and some of those will be non-admissible when $\pi$
is supercuspidal or a twist of Steinberg representation of $\GL_2(\Qp)$.

We also prove a version of the theorem for framed deformation rings, when $\rho$ is split and generic.  

\begin{thm}\label{gen_split_intro} Assume that $\rho=\bigl ( \begin{smallmatrix} \delta_1 & 0 \\ 0 & \delta_2\end{smallmatrix} \bigr)$ with $\delta_1\delta_2^{-1}\neq \omega^{\pm 1}, \Eins$. Let $\tau=\theta_1\oplus \theta_2$ be a principal series type with $\theta_1\neq \theta_2$ and let $\wt=(a, b)$ with $a<b$. If  $\theta_1$ is not congruent modulo $\varpi$ to any of the four characters 
$\delta_1\omega^{-a}$, $\delta_1\omega^{-b}$, $\delta_2\omega^{-a}$, $\delta_2\omega^{-b}$ and $R^{\square, \psi}_{\rho}(\wt, \tau)\neq 0$ then 
it is Cohen--Macaulay.
\end{thm}

The assumption on the congruence class of $\theta_1$ modulo $\varpi$ is equivalent to the assumption that the reducible locus is empty. In general, we have trouble controlling the intersection of the reducible and irreducible loci. 
 
\begin{rem}
When $\rho$ is an \emph{unramified} extension of $\delta$ by itself, including the case when $\rho\cong \delta\oplus \delta$,  Fabian Sander has found examples of crystalline deformation rings which are not Cohen-Macaulay, see \cite{sa}. However, the weight in those examples was small 
and so the irreducible locus was empty. So it is not clear to us how far one can relax the assumptions  in Theorem \ref{gen_split_intro}.
\end{rem}

The assumption $\theta_1\neq \theta_2$ excludes the crystalline case, which corresponds to both $\theta_1$ and $\theta_2$ being trivial. In \S \ref{section-further} we pursue a different idea to try and prove that these rings are Cohen--Macaulay: 
in \cite{pa12} a finitely generated $R_{\rho}^{\rm un}$-module $M(\Theta)$ is constructed, where $\Theta$ is a $K$-invariant  $\OO$-lattice 
in a finite dimensional irreducible representation $\sigma(\w,\tau)$ (resp. $\sigma^{\mathrm{cr}}(\w,\tau)$), which in the crystalline case is equal to $\Sym^{b-a-1} L^2\otimes \det^a$. Here $K=\GL_2(\Zp)$.
It is then shown that $M(\Theta)$ is Cohen--Macaulay and the action of $R_{\rho}^{\rm un}$ on $M(\Theta)$ factors through the 
faithful action of $R^{\psi}_{\rho}(\w,\tau)$ (resp. $R^{\psi, \mathrm{cr}}_{\rho}(\w,\tau)$). If we can choose $\Theta$ so that 
$\Hom_K(\Theta/\varpi, \kappa(\rho))$, where $\kappa(\rho)$ is the $\GL_2(\Qp)$-representation corresponding to $\rho$ under
the mod $p$ local Langlands correspondence, is one dimensional then the module $M(\Theta)$ is generated by one element, and hence 
it is free over $R^{\psi}_{\rho}(\w,\tau)$ (resp. $R^{\psi, \mathrm{cr}}_{\rho}(\w,\tau)$) of rank $1$.  In this case, the Cohen--Macaulayness of the module 
implies the same property for the ring. In general, it is  very hard to construct such lattices, since both $\Theta/\varpi$ and the restriction 
of $\kappa(\rho)$ to $K$ are not semi-simple. We manage to do so in the crystalline case, when 
the Hodge--Tate weights satisfy $1\le b-a\le 2p$ under some genericity assumption on $\rho$. We show that in this case the crystalline deformation ring is always of the form  $\OO[[x,y]]/( y^e+ a_{e-1} y^{e-1}+\cdots+ a_0)$, where $a_i$ lie in the maximal ideal of $\OO[[x]]$ and 
$e$ is the Hilbert--Samuel multiplicity of the special fibre. In particular, the rings are complete intersection. The restriction on the weights implies that $e$ is at most $3$. 

For an arbitrary inertial type $\tau$ and Hodge--Tate weight $\w$ we devise a representation theoretic criterion, when the rings 
$R^{\psi}_{\rho}(\w,\tau)$ (resp. $R^{\psi, \mathrm{cr}}_{\rho}(\w,\tau)$) are Gorenstein if we assume them to be Cohen--Macaulay, see
Proposition \ref{prop-criterion1-Gorenstein}. The criterion is "if and only if", but unfortunately it seems very hard to check in practice, when 
the Hilbert--Samuel multiplicity is large. If $R^{\psi}_{\rho}$ is formally smooth, which happens in most of the cases, we show that   
$R^{\psi}_{\rho}(\w,\tau)$ (resp. $R^{\psi, \mathrm{cr}}_{\rho}(\w,\tau)$) are Gorenstein if and only if they are complete intersection, see 
Proposition \ref{gor_ci}. The proof is pure commutative algebra  and relies on a happy coincidence that $\dim R^{\psi}_{\rho}- \dim R^{\psi}_{\rho}(\w,\tau)=2$.

\subsection{Global application} Our result combined with the results of Jack Shotton proved in \cite{shotton}, when $p>2$, and  in 
\S\ref{appendix_jack}, when $p=2$, imply that certain global potentially semi-stable deformation rings are $\OO$-torsion free. This was one of our motivations to prove the Cohen--Macaulayness of local deformation rings. We will describe the result in more detail.  

Let $F$ be a totally real field in which $p$ splits completely, let $\Sigma$ and $S$ be   finite sets  of places of $F$ containing all the places above $p$ and all the archimedean places. Assume that $\Sigma \subset S$ and if $p=2$ then $\Sigma \neq S$. We fix an algebraic closure $\overline{F}$ of $F$. Let $G_{F,S}$ be 
the Galois group of the maximal extension of $F$ in $\overline{F}$ which is unramified outside $S$. 
Let $\rhobar: G_{F, S}\rightarrow \GL_2(k)$ be a continuous irreducible representation, which we assume  to be modular. We will denote by $\rhobar_v$ the restriction of $\rhobar$ to a decomposition subgroup at $v$.

If $p=2$ then we 
assume that the image of $\rhobar$ is non-solvable and $\rhobar_v$ does not have scalar semi-simplification at any $v\mid p$. If $p>2$ then we assume that the restriction of $\rhobar$ to $G_{F(\zeta_p)}$ is irreducible,  where $\zeta_p$ is a primitive $p$-th root of unity. If $p=5$ then we further assume that the projective image of $\rhobar|_{G_{F(\zeta_p)}}$ is not isomorphic to $A_5$ and if $p=3$ then we further assume $\rhobar_v\not\sim \bigl ( \begin{smallmatrix}\chi \omega & \ast \\ 0 & \chi \end{smallmatrix}\bigr )$ for any character $\chi: G_{F_v}\rightarrow k^{\times}$ and for any $v\mid p$.

Let $\psi: G_{F, S}\rightarrow \OO^{\times}$ be a totally even character, such that $\det \rhobar \equiv \psi\epsilon \pmod{\varpi}$. Let 
$R^{\psi}_{F, S}$ be the universal deformation ring of $\rhobar$ parameterizing deformations with determinant $\psi\epsilon$. We will construct 
a quotient of $R^{\psi}_{F, S}$ by imposing local deformation conditions at places in $\Sigma$. Let $R^{\square, \psi}_v$ be the universal 
framed deformation ring of $\rhobar_v$ with determinant $\psi\epsilon$. 

If $v$ is infinite then we let $\bar{R}^{\square, \psi}_v=R^{\square, \psi}_v$. It can be checked that this ring is complete intersection.  For each finite $v\in \Sigma$ we fix a semisimple representation $\tau_v: I_{F_v} \rightarrow \GL_2(L)$, where $I_{F_v}$ is the 
inertia subgroup of $G_{F_v}$. If  $v\nmid p$  we let $\bar{R}_v^{\square, \psi}$ be the maximal reduced and $p$-torsion free quotient of $R_v^{\square, \psi}$ all of whose $\overline{L}$-points give rise to representations $\rho$ of $G_{F_v}$, such that the semisimplification 
of the restriction of $\rho$ to $I_{F_v}$ is isomorphic to $\tau_v$. Jack Shotton has proved in \cite{shotton} and in the appendix below that 
these rings are Cohen--Macaulay. 

For each $v\mid p$ we additionally fix a pair of integers $\wt_v=(a_v,b_v)$ with $a_v<b_v$.   We let $\bar{R}_v^{\square, \psi}$ be the maximal reduced and $p$-torsion free quotient of $R_v^{\square, \psi}$ all of whose $\overline{L}$-points give rise to representations $\rho$ of $G_{F_v}$, which are potentially semistable (resp. potentially crystalline) of type $(\mathbf{w}_v, \tau_v, \psi)$.

Let $\bar{R}^{\psi}_{F, S}$ be the quotient of $R^{\psi}_{F, S}$ corresponding to the local deformation conditions $\{ \bar{R}^{\square, \psi}_v\}_{v\in \Sigma}$, see \S\ref{present} for a precise definition.
\begin{thm} If  $\bar{R}^{\square, \psi}_v\neq 0$ for all $v\in \Sigma$ then $\bar{R}^{\psi}_{F, S}$ is a finitely generated $\OO$-module of rank at
least $1$ and $\bar{R}^{\psi}_{F, S}[1/p]$ is reduced. Moreover, if for all $v \mid p$, $\rhobar_v$, $\tau_v$, $\w_v$ satisfy the assumptions of either 
Theorem \ref{theorem-main-intro} or Theorem \ref{gen_split_intro} then $\bar{R}^{\psi}_{F, S}$ is $p$-torsion free. 
\end{thm} 

The first assertion is proved by  a Khare--Wintenberger argument, where we use  the work of Gee--Kisin \cite{gee_kisin}, when $p>2$,  and  \cite{blocksp2}, when $p=2$, as an input. The assumptions on $\rhobar_v$, when $p=2$ and $p=3$, appear, because we don't know the Breuil--M\'ezard conjecture in those cases.  
The fact that $\bar{R}^{\psi}_{F, S}$ is $p$-torsion free if all the local deformation 
rings are Cohen--Macaulay is a well known consequence of a Khare--Wintenberger argument, see for example \cite[\S 5]{sn} or 
Remark after Lemma 4.6 in \cite{kw2}.  The proof that the rings $\bar{R}^{\psi}_{F, S}[1/p]$ are reduced follows from a little bit of commutative 
algebra, see Lemma \ref{com_alg}, together with Kisin's approach to Taylor--Wiles method. It turned out that Kisin was aware of our proof, although the statement does not seem to appear in the literature. Jack Thorne has pointed out to us  that a similar argument occurs in the recent paper of Patrick Allen, in the proof of \cite[Thm. 3.1.3]{allen}. We have also made the effort to treat the case $p=2$, where the proof is somewhat more involved.

Typically, whenever one proves a modularity lifting theorem, one proves that the surjection $\bar{R}^{\psi}_{F, S}\twoheadrightarrow \mathbb T$
is an isomorphism if we invert $p$ and pass to the reduced rings, where $\mathbb T$ is a suitable Hecke algebra. If $\bar{R}^{\psi}_{F, S}$ is reduced and 
$p$-torsion free then we can conclude that $\bar{R}^{\psi}_{F, S}\overset{\cong}{\longrightarrow} \mathbb T$. This implies that the modular forms 
which live in characteristic $0$, already know everything about deformations of $\rhobar$ to rings which might have $p$-torsion, which we find quite pretty.

 The properties of deformation rings like Cohen--Macaulayness, Gorensteinness, complete intersection play a role in the theory of derived deformation rings, \cite{venkatesh}. We hope that our results and methods will find an application in that theory.

\subsection{Organization} The paper is organized as follows. After introducing some notation in \S\ref{section-notations}, we recall in \S\ref{section-recall} some results established in \cite{pa12}, in particular the construction of a certain element $x\in R^{\psi}_{\rho}(\w,\tau)$. 
We will show in \S \ref{criterion} that $R^{\psi}_{\rho}(\w,\tau)$ is Cohen--Macaulay if and only if 
$(x, \varpi)$ is a regular sequence in $R^{\psi}_{\rho}(\w,\tau)$. In \S\ref{section-diagram} we prove a  key result on lattices in locally algebraic representations of $G$, which crucially uses an idea of Vign\'eras \cite{vi} of constructing lattices in locally algebraic representations by constructing lattices in 
the corresponding $G$-equivariant coefficient systems on the Bruhat--Tits tree. We prove the main theorem in \S\ref{section-proof}. 
In \S\ref{section-split} we extend the main result to the split generic case under some assumptions. In \S\ref{section-further}, we investigate some other properties about the Galois deformation rings related to Cohen-Macaulayness. Finally we  discuss some global applications in \S\ref{global_app}. In Appendix \ref{equip} we show that completed tensor product preserves equidimensionality. \textsc{Jack Shotton} shows in Appendix \ref{appendix_jack} that the rings $\bar{R}^{\square, \psi}_v$ with $v\nmid p$, which he studied in \cite{shotton} under the assumption $p>2$, remain Cohen--Macaulay if $p=2$.

\section{Notation}\label{section-notations}

 Let $G=\GL_2(\Q_p)$, $K=\GL_2(\Z_p)$, and $Z\cong\Q_p^{\times}$ be the center of $G$.
Let $P$ be the subgroup of upper triangular matrices in $G$. Given two characters $\delta_1,\delta_2:\Q_p^{\times}\ra L^{\times}$ (or to $k^{\times}$), we consider $\delta_1\otimes\delta_2$ as a character of $P$ sending a matrix $\smatr{a}b0d$ to $\delta_1(a)\delta_2(d)$.
 Define the following subgroups of $G$:
\[K_m:=\matr{1+p^m\Z_p}{p^m\Z_p}{p^m\Z_p}{1+p^m\Z_p},\ \ m\geq1\]
\[I:=\matr{\Z_p^{\times}}{\Z_p}{p\Z_p}{\Z_p^{\times}},\ \ I_m:=\matr{1+p^m\Z_p}{p^{m-1}\Z_p}{p^m\Z_p}{1+p^m\Z_p},\ \ m\geq 1.\]
If $H$ is a closed subgroup of $G$ and $\sigma$ is a smooth representation of $H$ on an $L$-vector space (or a $k$-vector space),  we denote by 
$\Ind_H^G \sigma$ the usual smooth induction. When $H$ is moreover open, we let $\cInd_{H}^G\sigma$ denote the compact induction, meaning the subspace of $\Ind_{H}^G\sigma$ consisting of functions whose support is compact modulo $H$.

Let $\mathrm{Mod}_{G}^{\rm sm}(\cO)$ be the category of smooth $G$-representations on $\cO$-torsion modules, and $\Mod_{G}^{\rm l,fin}(\cO)$ be its full subcategory consisting of locally finite objects. Here, an object $\tau\in\mathrm{Mod}^{\rm sm}_{G}(\cO)$ is said to be  \emph{locally finite} if for all $v\in \tau$ the $\cO[G]$-submodule generated by $v$ is of finite length.  Let $\Mod_{G}^{\rm sm}(k)$ and $\Mod_{G}^{\rm l,fin}(k)$ be respectively the full subcategory consisting of $G$-representations on $k$-modules.
If $\zeta$ is a continuous character from  $Z$ to $\cO^{\times}$ or to $k^{\times}$, we append the subscript $\zeta$ to denote the corresponding subcategory consisting of objects on which $Z$ acts by $\zeta$.

Let $\Mod_{G}^{\rm pro}(\cO)$ be the category of compact $\OK$-modules with an action of $\cO[G]$ such that the two actions coincide when restricted to $\cO[K]$. It is anti-equivalent to $\Mod_{G}^{\rm sm}(\cO)$ under Pontryagin dual $\tau\mapsto \tau^{\vee}:=\Hom_{\cO}(\tau,L/\cO)$, where if $\tau$ is a smooth representation then $\tau^{\vee}$ is equipped with the compact-open topology. We let $\mathfrak{C}(\cO)$, resp. $\mathfrak{C}(k)$ be the full subcategory of $\Mod_{G}^{\rm pro}(\cO)$ anti-equivalent to $\Mod_{G,\zeta}^{\rm l,fin}(\cO)$, resp. $\Mod_{G,\zeta}^{\rm l,fin}(k)$.

Denote by $G_{\Q_p}$ the absolute Galois group of $\Q_p$, $I_{\Q_p}$ its inertia subgroup, and $W_{\Q_p}$ the Weil group. A character of $\Q_p^{\times}$ will be viewed as a character of $G_{\Q_p}$ via local class field theory, and vice versa. Here we normalize the local reciprocity map $\Q_p^{\times}\simto W_{\Q_p}^{\rm ab}$ so that uniformizers correspond to geometric Frobenii. Denote by $\epsilon: G_{\Q_p}\ra \Z_p^{\times}$ the cyclotomic character and $\omega$ its reduction modulo $p$. \medskip

Recall that Colmez in \cite{co} has defined an exact and covariant functor $\VV$ from  the category of smooth, finite length $G$-representation on $\cO$-torsion modules  with central character $\zeta$ to the category of continuous finite length representations of $G_{\Q_p}$ on $\cO$-torsion modules. Following \cite[\S3]{pa12}, we define an exact covariant functor $\check{\VV}:\mathfrak{C}(\cO)\ra \mathrm{Rep}_{G_{\Q_p}}(\cO)$ as follows: if $M\in\mathfrak{C}(\cO)$ is of finite length, we let $\check{\VV}(M):=\VV(M^{\vee})^{\vee}(\epsilon\psi)$, where $\psi: G_{\Qp}\rightarrow \OO^{\times}$ is a character corresponding to $\zeta$ via the local class field theory. For general $M\in\mathfrak{C}(\cO)$, write $M=\plim M_i$ with $M_i$ of finite length and define $\check{\VV}(M):=\plim\check{\VV}(M_i)$.  On the other hand, if $\Pi$ is an admissible unitary $L$-Banach space representation with central character $\zeta$ and $\Theta$ is any open bounded $G$-invariant $\cO$-lattice in $\Pi$,  then it can be shown that the Schikhof dual $\Theta^d:=\Hom_{\OO} (\Theta, \OO)$ equipped with the weak topology lies in $\dualcat(\OO)$ and we let $\check{\VV}(\Pi):=\check{\VV}(\Theta^d)\otimes_{\cO}L$. One may further show that $\check{\VV}(\Pi)$ does not depend on the choice
of $\Theta$. Note that $\check{\VV}$ is exact and \emph{contravariant} on the category of admissible unitary $L$-Banach space representations of $G$ with central character $\zeta$.

\section{Preliminaries}\label{section-recall}

Let $\rho:G_{\Q_p}\ra \GL_2(k)$ be a continuous representation such that $\End_{G_{\Q_p}}(\rho)=k$ and if $p=3$ then we assume\footnote{If $p=2$ then $\omega$ is trivial and this case is excluded by requiring that $\rho$ has only scalar endomorphisms. If $p=3$ then $\omega^2=1$ and the case $\rho\sim \smatr{\delta}{*}0{\delta\omega}$ is also excluded by our assumption.}
 that 
$\rho\not\sim \bigl ( \begin{smallmatrix}\delta \omega & \ast \\ 0 & \delta \end{smallmatrix}\bigr )$
for any character $\delta: G_{\Qp}\rightarrow k^{\times}$. 
It will be convenient to divide such representations into two classes:
 we say $\rho$ is \emph{generic} if one of the following conditions holds: (1) $\rho$ is irreducible, (2)  $\rho\sim \smatr{{\delta}_2}*0{{\delta}_1}$ with $\delta_1^{-1}\delta_2\neq \omega^{\pm1}$, (3) $p\ge 5$ and $\rho\sim \smatr{\delta}*0{\delta\omega}$ for some $\delta:G_{\Q_p}\ra k^{\times}$; 
otherwise, we say $\rho$ is \emph{non-generic}, so that $p\ge 5$ and $\rho$ is of the form 
$\smatr{\delta\omega}*0{\delta}$ for some $\delta:G_{\Q_p}\ra k^{\times}$.

\subsection{Construction of a regular element $x$}\label{subsection-x}
Let $\psi:G_{\Q_p}\ra \cO^{\times}$ be a continuous character such that $\psi\epsilon \equiv \det(\rho)\ (\mathrm{mod}\ \varpi)$. Let $R^{\psi}_{\rho}$ be the universal deformation ring  parametrizing the deformations of $\rho$ with  determinant $\psi\epsilon$.
We fix  a $p$-adic Hodge type $(\w,\tau)$ as in the introduction, and let $R^{\psi}_{\rho}(\w,\tau)$ (resp. $R^{\psi,\rm cr}_{\rho}(\w,\tau)$) be the unique reduced $\cO$-flat quotient of $R^{\psi}_{\rho}$ parametrizing the deformations of $\rho$ which are potentially semi-stable (resp. potentially crystalline) of type $(\w,\tau,\psi)$.  As a special case of a theorem of Kisin \cite{ki08}, we have the following result.

 \begin{thm}\label{thm-kisin}
Whenever non-zero, $R^{\psi}_{\rho}(\w,\tau)$ (resp. $R^{\psi,\rm cr}_{\rho}(\w,\tau)$) is a local complete noetherian reduced $\cO$-flat algebra,  equidimensional of Krull dimension 2.
\end{thm}

The proof of Breuil-M\'ezard conjecture in \cite{pa12} gives another construction of the rings $R^{\psi}_{\rho}(\w,\tau)$ and $R^{\psi,\rm cr}_{\rho}(\w,\tau)$, which we will now recall. 
In the following we allow ourselves to replace $L$ by a finite extension.
By a result of Henniart \cite{he}, there is a smooth irreducible $L$-representation 
$\sigma(\tau)$ (resp.~$\sigma^{\rm cr}(\tau)$) of $K$, such that if $\pi$ is a smooth irreducible infinite dimensional $L$-representation of $G$, 
then $\Hom_{K}(\sigma(\tau),\pi)\neq0$ (resp. $\Hom_K(\sigma^{\rm cr}(\tau),\pi)\neq 0$) if and only if $\mathrm{LL}(\pi)|_{I_{\Q_p}}\cong \tau$ (resp. $\mathrm{LL}(\pi)|_{I_{\Q_p}}\cong \tau$ and the monodromy operator is trivial). Here $\mathrm{LL}(\pi)$
is the Weil-Deligne representation attached to $\pi$ by the classical local Langlands correspondence.  A more concrete description of $\sigma(\tau)$ and 
$\sigma^{\mathrm{cr}}(\tau)$ in the cases we are interested in is given in \S\ref{types_princ}.  Let 
$$\sigma(\w,\tau):=\sigma(\tau)\otimes \Sym^{b-a-1}L^2\otimes{\det}^a,$$
 $$\sigma^{\rm cr}(\w,\tau):=\sigma^{\rm cr}(\tau)\otimes \Sym^{b-a-1}L^2\otimes {\det}^a,$$ 
and choose a $K$-invariant $\cO$-lattice  $\Theta$ inside $\sigma(\w,\tau)$ 
(resp. inside $\sigma^{\rm cr}(\w,\tau)$).  According to \cite[\S6.1, \S6.2]{pa12} 
and  \cite[\S2.3]{blocksp2} if $p=2$ or $p=3$, there exists an object $N\in\CO$ with 
a faithful continuous action of $R^{\psi}_{\rho}$, which commutes with the action of $G$, 
such that the following hold:

\begin{enumerate}
\item[(a)] $N$ is  projective in $\Mod^{\pro}_{K, \zeta}(\OO)$;

\item[(b)] $N\widehat{\otimes}_{R^{\psi}_{\rho}}k$ is of finite length in $\mathfrak{C}(k)$, and is finitely generated over $\OK$;

\item[(c)] $\End_{\CO}(N)\cong R^{\psi}_{\rho}$ and ${\check{\VV}}(N)$ is isomorphic to $\rho^{\rm un,\psi}$ as $R^{\psi}_{\rho}\br{G_{\Qp}}$-module, where $\rho^{\rm un,\psi}$ is the universal deformation of $\rho$ with fixed determinant $\psi\epsilon$.
\end{enumerate}
Let $\Theta^d:= \Hom_{\OO}(\Theta, \OO)$  and define 
\[M(\Theta):=\Hom_{\cO}(\Hom_{\OK}^{\mathrm{cont}}(N,\Theta^d),\cO)\]
which is naturally an $R^{\psi}_{\rho}$-module. Part (b) implies that $M(\Theta)$ is a finitely generated $R^{\psi}_{\rho}$-module, see \cite[Prop.2.15]{pa12}.
The following theorem is proved in \cite{pa12}, \cite{blocksp2}.

\begin{thm}\label{thm-pa12-A} We have  $R^{\psi}_{\rho}(\w,\tau)\cong R^{\psi}_{\rho}/\mathrm{Ann}(M(\Theta))$, where $\mathrm{Ann}(M(\Theta))$ denotes the annihilator of $M(\Theta)$ in $R^{\psi}_{\rho}$.  Whenever nonzero, $M(\Theta)$ is a Cohen-Macaulay $R^{\psi}_{\rho}$-module  of Krull dimension 2.

Analogous  results hold for $R^{\psi,\rm cr}_{\rho}(\w,\tau)$, by taking $\Theta$ a $K$-invariant $\cO$-lattice inside $\sigma^{\rm cr}(\w,\tau)$.
\end{thm}
\begin{proof}
See \cite[Cor.6.5]{pa12}, \cite[Prop.2.33]{blocksp2} for $\rho$ generic and \cite[Thm.2.44, Cor. 6.22, Prop.6.23]{pa12} for $\rho$ non-generic.
\end{proof}

If $\lambda\in\mathrm{Mod}_{K}^{\rm sm}(\cO)$ is of finite length, we let
\[M(\lambda):=(\Hom_{\OK}^{\rm cont}(N,\lambda^{\vee}))^{\vee}.\]
Then $M(\lambda)$ is a finitely generated $R^{\psi}_{\rho}$-module, and if we apply this definition with $\lambda=\Theta/\varpi \Theta$ then we get a natural isomorphism
of $R^{\psi}_\rho$-modules $M(\Theta)/\varpi M(\Theta)\cong M(\Theta/\varpi \Theta)$, see \cite[Cor. 2.7, Lem.2.14]{pa12}.

For convenience, we   write $R$ for $R^{\psi}_{\rho}(\w,\tau)$ or $R^{\psi,\rm cr}_{\rho}(\w,\tau)$, and let $\Theta$ be a $K$-invariant $\cO$-lattice inside $V$ where $V$ is $\sigma(\w,\tau)$ or $\sigma^{\rm cr}(\w,\tau)$, depending on $R$.

\subsubsection{Generic case}\label{subsubsection-generic}
When $\rho$ is \emph{generic},  a regular element $x\in R^{\psi}_{\rho}$  for $M(\Theta)$ is constructed in \cite[\S5]{pa12}, such that $N/xN$ is a finitely generated $\OK$-module and projective in $\mathrm{Mod}_{K,\zeta}^{\rm pro}(\cO)$. We  briefly recall its construction. If $\rho$ is irreducible then
let $\pi$ be the unique smooth irreducible $k$-representation with central character such that $\check{\VV}(\pi^{\vee})=\VV(\pi)=\rho$.  If $\rho\sim \smatr{\delta_2}*0{\delta_1}$ with $\delta_1^{-1}\delta_2\neq \ide,\omega$ then let $\pi:=\Ind_{P}^G\delta_1\otimes\delta_2\omega^{-1}$, so that $\check{\VV}(\pi^{\vee})=\delta_1$. 
By definition, $N$ is  a projective envelope of $\pi^{\vee}$ in $\CO$, see \cite[\S6.1]{pa12}, so  $N/\varpi N$ is   a projective envelope of $\pi^{\vee}$ in $\mathfrak{C}(k)$.

\begin{thm}\label{thm-pa12-B} There is an element $x\in R^{\psi}_{\rho}$ such that the following hold:
\begin{enumerate}
\item[(i)] $x:N\ra N$ is injective, and for all $n\geq 1$, $N/x^nN$ is a finitely generated $\OK$-module and projective in $\mathrm{Mod}_{K,\zeta}^{\rm pro}(\cO)$; $\Hom_{\OK}^{\cont}(N/x^nN,\Theta^d)$
is a free  $\OO$-module of rank $n\cdot e$, where $e:=e(R/\varpi R)$ is the Hilbert--Samuel multiplicity of $R/\varpi R$;

\item[(ii)] $x$ is $M(\sigma)$-regular for any  smooth irreducible $k$-representation $\sigma$ of $K$ such that 
$M(\sigma)\neq 0$;

\item[(iii)] $x\notin \m^2+(\varpi)$, where $\m$ denotes the maximal ideal of $R^{\psi}_{\rho}$;

\item[(iv)]   the only prime ideal of $R$ containing $(x,\varpi)$ is the maximal ideal $\m_R$.
\end{enumerate}
\end{thm}

\begin{proof}

 It is shown in \cite[Thm.5.2]{pa12} that there exists an exact sequence in $\mathfrak{C}(k)$:
\[0\ra N/\varpi N\overset{\bar{x}}{\ra} N/\varpi N\ra \Omega^{\vee}\ra0,\]
with $\Omega$ being isomorphic to an injective envelope of $\pi|_K$ in $\Mod_{K,\zeta}^{\rm sm}(k)$. We note that there is no restriction on $p$ in \cite[\S 5]{pa12}. Let  $x\in R^{\psi}_{\rho}$ be  any lifting of $\bar{x}$. We will still  denote by $x$ its image in $R$.

(i) By \cite[Thm.5.2]{pa12}, $x:N\ra N$ is injective, and $N/xN$ is a projective envelope of $(\rsoc_K\pi)^{\vee}$ in $\mathrm{Mod}_{K,\zeta}^{\rm pro}(\cO)$, so is finitely generated and projective in $\mathrm{Mod}_{K,\zeta}^{\rm pro}(\cO)$. This implies that $N/x^nN\cong (N/xN)^{\oplus n}$ as $\OK$-module for all $n\geq 1$. In particular, $N/x^n N$ is a finitely generated $\OK$-module, which is projective in $\mathrm{Mod}_{K,\zeta}^{\rm pro}(\cO)$ and it is enough to prove the second statement for $n=1$. Projectivity of $N/xN$ implies that the rank of  $\Hom_{\OK}^{\cont}(N/xN,\Theta^d)$ is equal to 
the dimension of $\Hom_{\OK}^{\cont}(N/xN, (\Theta/\varpi \Theta)^{\vee})$ as a $k$-vector space, which again by projectivity of $N/xN$ is equal to 
$\sum_{\sigma}  m_\sigma  \dim_k M(\sigma)/ x M(\sigma)$, where the sum is taken over all irreducible $\sigma\in \mathrm{Mod}_{K,\zeta}^{\rm pro}(\cO)$ and $m_{\sigma}$ 
denotes the multiplicity with which $\sigma$ occurs as a subquotient of  $\Theta/\varpi \Theta$.   On the other hand, it follows from \cite[Thm.6.6]{pa12} if $p\ge 5$ and 
\cite[Thm.2.34]{blocksp2} if $p=2$ or $p=3$ that the Hilbert--Samuel multiplicity of $R$ is equal to $\sum_{\sigma}  m_\sigma  e(M(\sigma))$, where $e(M(\sigma))$ is the Hilbert--Samuel multiplicity of the module $M(\sigma)$. Therefore, it is enough to show that $e(M(\sigma))= \dim_k M(\sigma)/ xM(\sigma)$. This is done in the proof of  \cite[Thm.6.6]{pa12} and 
\cite[Thm.2.34]{blocksp2}, where it is shown that if $M(\sigma)\neq 0$ then  $e(M(\sigma))= \dim_k M(\sigma)/ xM(\sigma)=1$.

(ii) We may assume that $M(\sigma)$ is non-zero. Then \cite[Thm.6.6]{pa12}  implies that $M(\sigma)$ is a Cohen-Macaulay $R^{\psi}_{\rho}$-module of dimension 1. Since $N/xN$ is a projective $\OK$-module, the exact sequence $0\ra N\ra N\ra N/xN\ra0$ gives an exact sequence
\[0\ra M(\sigma)\overset{x}{\ra} M(\sigma)\ra  \Hom_{K}(N/xN,\sigma^{\vee})^{\vee}\ra0.\]
In particular, $x$ is $M(\sigma)$-regular. If $p=2$ or $p=3$ then the assertion is proved in the course of the proof of \cite[Prop.2.28]{blocksp2} with the same argument.

(iii)  Let $\sigma$ be a smooth irreducible $k$-representation of $K$ such that $M(\sigma)\neq0$ and let $\fa$ be the $R^{\psi}_{\rho}$-annihilator of $M(\sigma)$. The proof of \cite[Thm.6.6]{pa12}, in particular the exact sequence (38) and the proof of   \cite[Prop.2.28]{blocksp2} if $p=2$ or $p=3$ show that $M(\sigma)$ is a cyclic $R^{\psi}_{\rho}$-module,
$R^{\psi}_{\rho}/\fa\cong k\llbracket t\rrbracket$ and $x$ is mapped to $t$ via this isomorphism. This implies the result, since if $x\in \m^2+(\varpi)$, then the image of  $x$ in $R^{\psi}_{\rho}/\fa$ would belong to $(\m_{R^{\psi}_{\rho}/\fa})^2$, which is not the case.

(iv) We may assume $R$ is non-zero. Since $M(\Theta)$ is a faithful, finitely generated $R$-module, $V( (\varpi, x))$ 
is equal to the support of $M(\Theta)/ (x, \varpi) M(\Theta)$. It follows from part (i) that $M(\Theta)/(\varpi, x) M(\Theta)$ 
is a finite dimensional $k$-vector space. Hence  its support consists only of the maximal ideal.
\end{proof}

\subsubsection{Non-generic case} \label{subsubsection-non-generic}
We now construct $x\in R^{\psi}_{\rho}$ with similar properties in non-generic cases.  We assume that $p\ge 5$.
We distinguish two cases depending on $\rho$: \emph{peu ramifi\'e} or \emph{tr\`es ramifi\'e}. 
The assumption $p\ge 5$ implies that $H^2(G_{\Qp}, \ad^0 \rho)=0$, where $\ad^0 \rho$ denotes the trace zero endomorphisms of $\rho$,  so $R^{\psi}_{\rho}$ is formally smooth 
of relative dimension $3$ over $\OO$. In this case $N\in \dualcat(\OO)$ is constructed in \cite[\S 6.2]{pa12} as a deformation of $\beta^{\vee}$ to $R^{\psi}_{\rho}$, for a certain 
$\beta\in \Mod^{\sm}_{G, \zeta}(k)$. In particular, $N$ is flat over $R^{\psi}_{\rho}$. 

\begin{lem}\label{tres_peu} Assume that $R^{\psi}_{\rho}(\wt, \tau)\neq 0$ with $\wt=(a, a+1)$ and $\tau=\theta\oplus \theta$. If $\rho$ is  tr\`es ramifi\'e then 
$R^{\psi}_{\rho}(\wt, \tau)\cong \OO\llbracket u\rrbracket$. If $\rho$ is peu ramifi\'e then $R^{\psi}_{\rho}(\wt, \tau)\cong \OO\llbracket u, v\rrbracket/(uv)$.
\end{lem}
\begin{proof}  Let $R'=R^{\psi}_{\rho}(\wt, \tau)$ and let $\mm$ be its maximal ideal. 
After twisting we may assume that $\psi$ and $\theta$ are trivial, $\rho$ is isomorphic to $\smatr{\omega}*01$ and $a=0$. 

If $\rho$ is  tr\`es ramifi\'e then it is shown in the course of the proof of \cite[Prop.6.23]{pa12} that $\dim_k \mm/(\varpi, \mm^2)=1$. Since $R'$ is $\OO$-torsion free this implies that $R'\cong \OO\br{u}$. Alternatively, see Th\'eor\`eme 1.3(i) in \cite{bm02}.

If $\rho$ is peu ramifi\'e then it is shown in Th\'eor\`eme 1.3(i) in \cite{bm02} that there are two minimal prime ideals $\pp_1$, $\pp_2$ of $R'$ such that 
$R'/\pp_1$ and $R'/\pp_2$ are formally smooth of dimension $1$ over $\OO$ and the diagonal map $R'\rightarrow R'/\pp_1\times R'/\pp_2$ is injective.  It is shown in the course of the proof of \cite[Prop.6.23]{pa12} that $\dim_k \mm/(\varpi, \mm^2)=2$.
Hence there are elements $u, v\in \mm \setminus (\mm^2, \varpi)$ such that $u\in \pp_1$ and $ v\in \pp_2$. It is  shown in \cite[Prop.3.3.1]{eg} that $V(\pp_1)$ is the closure of the crystalline points, $V(\pp_2)$ is the closure of non-crystalline points and 
the ideals $\pp_1$ and $\pp_2$ are distinct modulo $\varpi$. 
This implies that the images of $u$ and $v$ in $\mm/(\mm^2, \varpi)$ are linearly independent and hence form a $k$-basis of $\mm/(\mm^2, \varpi)$.  Since 
$uv\in \pp_1\cap \pp_2= 0$, we conclude that there is a surjection $\OO\br{u,v}/(uv)\twoheadrightarrow R'$. This induces surjections
$$\OO[[u, v]]/( uv, u)\twoheadrightarrow R'/(u)\twoheadrightarrow R'/\pp_1, \quad \OO[[u, v]]/( uv, v)\twoheadrightarrow R'/(v)\twoheadrightarrow R'/\pp_2,$$
which have to be isomorphisms as both the source and the target are domains of dimension $1$. Thus $\pp_1=(u)$ and $\pp_2=(v)$. 
This implies that the composition  $\OO\br{u,v}/(uv)\twoheadrightarrow R'\rightarrow R'/\pp_1\times R'/\pp_2$ is injective and hence 
$\OO\br{u,v}/(uv)\cong R'$.
\end{proof}

Since twisting by a character $\delta: G_{\Qp}\rightarrow \OO^{\times}$ induces a natural isomorphism between $R^{\psi}_{\rho}$ and $R^{\psi\delta^2}_{\rho\bar{\delta}}$ 
it is enough to construct $x$ in the case when $\rho$ is isomorphic to $\smatr{\omega}*01$ and $\psi$ and $\zeta$ are trivial, which we now assume for the rest of the subsection. In this case
$R':=R^{\psi}_{\rho}((0,1), \Eins\oplus \Eins)\neq 0$. 

If $\rho$ is tr\`es ramifi\'e we choose $x\in R^{\psi}_{\rho}$ to be any element which maps to $u$ in $R'/(\varpi)\cong k\br{u}$
and if $\rho$ is  peu ramifi\'e we choose $x\in R^{\psi}_{\rho}$ to be any element which maps to $u+v$ in $R'/(\varpi)\cong k\br{u,v}/(uv)$ via the isomorphisms of Lemma \ref{tres_peu}. We note that in both cases $(\varpi, x)$ is a regular sequence in $R'$.

\begin{thm}\label{thm-x-nongeneric}
Assume $\rho$ is non-generic and let $x\in R^{\psi}_{\rho}$ be as  above. Then the properties (i),(ii),(iii),(iv) listed in Theorem \ref{thm-pa12-B} hold.
\end{thm}

\begin{proof}
We first prove (ii). If $\rho$ is tr\`es ramifi\'e then it is shown in \cite[Prop.6.20]{pa12} that $M(\sigma)\neq 0$ if and only if $\sigma\cong\mathrm{st}\overset{\rm def}{=}\Sym^{p-1}k^2$. Moreover, $M(\st)$ is a free $R'/(\varpi)$-module of rank $1$. Hence $x$ is $M(\st)$-regular and 
\begin{equation}\label{dim_tres}
\dim_k M(\st)/ x M(\st)=1.
\end{equation}
 If $\rho$ is peu ramifi\'e, then it is shown in \cite[Prop.6.20]{pa12} that $M(\sigma)\neq0 $ if and only if $\sigma=\ide$ or $\sigma=\mathrm{st}$. Moreover, 
$M(\st)$ is a free $R'/(\varpi)$-module of rank $1$, and $M(\Eins)$ is a free $R'/(\varpi, \pp_1)$-module of rank $1$, where $\pp_1$ is the closure of the crystalline points 
in $R'$. Since $R'/(\varpi, \pp_1)\cong R'/(\varpi, u)\cong k\br{v}$, we conclude that $x$ is $M(\st)$- and $M(\Eins)$-regular and
\begin{equation}\label{dim_peu} 
\dim_k M(\st)/ x M(\st)=2, \quad \dim_k M(\Eins)/xM(\Eins)=1.
\end{equation}

(i) Since $R^{\psi}_{\rho}$ is formally smooth over $\OO$, it is an integral domain and $x^n$ is $R^{\psi}_{\rho}$-regular for all $n\ge 1$. Since $N$ is flat over $R^{\psi}_{\rho}$, $x^n$ is also $N$-regular. It follows from \eqref{dim_tres}, \eqref{dim_peu} that $M(\sigma)/ x^n M(\sigma)$ is a finite dimensional $k$-vector space for all irreducible $\sigma\in \Mod^{\sm}_{K, \zeta}(\OO)$. This implies that $M(\lambda)/ x^n M(\lambda)$ is a finitely generated $\OO$-module for all $\lambda \in  \Mod^{\sm}_{K, \zeta}(\OO)$ of finite length.
It follows from  \cite[Lem.2.38]{pa12} that $N/x^nN$ is a finitely generated $\OK$-module. For each irreducible $\sigma\in \Mod^{\sm}_{K, \zeta}(\OO)$, the exact sequence $0\ra N\overset{x^n}{\ra} N\ra N/x^nN\ra0$ induces an exact sequence: 
\begin{equation}\label{sequence}
0\ra \Ext^1_{K, \zeta}(N/x^nN,\sigma^{\vee})^{\vee}\ra M(\sigma)\overset{x^n}{\ra} M(\sigma)
\ra \Hom_{K}(N/x^nN,\sigma^{\vee})^{\vee}\ra0,
\end{equation}
where the subscript $\zeta$ indicates that we compute the extension groups in $\Mod^{\pro}_{K, \zeta}(\OO)$.
It follow from part (ii) proved above that the multiplication by $x^n$ induces an injective map on $M(\sigma)$. Hence, 
\[\Ext^1_{K, \zeta}(N/x^nN,\sigma^{\vee})=0,\]
for all irreducible $\sigma\in \Mod^{\sm}_{K, \zeta}(\OO)$. Since every non-zero object in $\Mod^{\pro}_{K, \zeta}(\OO)$ has 
an irreducible quotient, this implies that $N/x^n N$ is equal to its own projective envelope in $\Mod^{\pro}_{K, \zeta}(\OO)$, and hence is 
projective. Since $x$ is $N$-regular and $N/xN$ is projective, we get an isomorphism $N/x^n N\cong (N/xN)^{\oplus n}$ as $\OK$-modules. 

It remains to show that the rank of $\Hom_{\OK}^{\cont}(N/xN, \Theta^d)$ is equal to the Hilbert--Samuel multiplicity of $R$. 
The argument  is the same as in the proof of part (i) of Theorem \ref{thm-pa12-B}: using the projectivity of $N/xN$ and \cite[Thm.6.24]{pa12} we deduce that it is enough to show that 
$e(M(\sigma))=\dim_k M(\sigma)/x M(\sigma)$ for all irreducible $\sigma\in  \Mod^{\sm}_{K, \zeta}(\OO)$. This follows from \eqref{dim_tres}, \eqref{dim_peu}
and the Hilbert--Samuel multiplicities of $M(\sigma)$ computed in  \cite[Thm.6.24]{pa12}.

(iii) As in the proof of Theorem \ref{thm-pa12-B}, it suffices to find a quotient ring $A$ of $R^{\psi}_{\rho}$, which is killed by $\varpi$, such that the image of $x$ does not lie in $\m_A^{2}$. By construction, we can take $A=R'/(\varpi)$.

(iv) Using (i), the same proof as in Theorem \ref{thm-pa12-B} works. 
\end{proof}
\begin{rem}
To deduce the projectivity of $N/xN$ in Theorem \ref{thm-x-nongeneric}, we can also  apply the criterion proved in \cite[Cor.2.41]{pa12} to $R/xR$ and $N/xN$.
\end{rem}

\subsubsection{Prime avoidance}\label{avoid}
  From now on, $\rho$ can be either generic or non-generic. 
  As $x$ is constructed as \emph{any} lifting of a certain element in $R^{\psi}_{\rho}/\varpi R^{\psi}_{\rho}$ (or in some quotient of $R^{\psi}_{\rho}/\varpi R^{\psi}_{\rho}$ when $\rho$ is non-generic),  we see that if $x'\in R^{\psi}_{\rho}$ is such that $x'\equiv x \mod \varpi$, then $x'$ also satisfies the properties (i)-(iv) listed in Theorem \ref{thm-pa12-B} or Theorem \ref{thm-x-nongeneric}. This allows us to choose one which behaves better.

\begin{lem}\label{lemma-primeavoidence}
Let $\p_1,...,\p_n$ be a finite set of maximal  ideals of $R^{\psi}_{\rho}[1/p]$. Then there is $x\in R^{\psi}_{\rho}$ satisfying the conditions of Theorem \ref{thm-pa12-B} in the generic case and the conditions of  Theorem \ref{thm-x-nongeneric} in the non-generic case, such that $x\notin\p_i$ for all $1\leq i\leq n$.
\end{lem}
\begin{proof}
It suffices to show that there exists $k\geq 1$, such that $x+\varpi^k$ is not contained in $\p_i$ for any $1\leq i\leq n$. If not, we can find $k<k'$ and $i\in\{1,...,n\}$ such that $x+\varpi^k,x+\varpi^{k'}\in \p_i$, which implies $\varpi^k-\varpi^{k'}\in\p_i$.  This is impossible since $\p_i$ does not contain $\varpi$ by assumption.
\end{proof}

We will need exclude one special case later, which leads us to consider the following condition on $\rho$ and on the type $(\w,\tau,\psi)$:
 \begin{equation}\label{Hyp}
  \tag{\textbf{H}}  \rho\ncong \smatr{{\delta}\omega}*0{{\delta}}, \text{  or } b-a>1, \text{ or } \tau\neq \theta\oplus\theta.
  \end{equation}
\begin{rem}
The condition \eqref{Hyp} comes from the restriction of \cite[Cor.4.21]{pa12}.  It follows from Lemma \ref{tres_peu}  
that if \eqref{Hyp} is not satisfied, then $R$ is a complete intersection, hence 
a Cohen-Macaulay ring, so there is no harm 
assuming \eqref{Hyp}.\end{rem}

When \eqref{Hyp} is satisfied, we use Lemma \ref{lemma-primeavoidence} to choose $x$ in such a way that it is not contained in the finite union of the maximal  ideals of $\Spec R[1/p]$ excluded in \cite[Cor.4.21]{pa12}. In particular, $x$ depends on the type $(\w,\tau)$. With this choice, $x$  has the following  additional
property.

\begin{prop}\label{prop-(H)}
Assume \eqref{Hyp} is satisfied and choose $x$  as above.  Then $R_{\p}$ is a regular local ring and $R_{\p}\cong M(\Theta)_{\p}$ for any prime ideal $\p$ of $R$ minimal over $(x)$.
\end{prop}
\begin{proof}
As $x$ is regular in $R$, Krull's principal ideal theorem implies that $\p$ has height $1$.  Since the Krull dimension of $R$ is $2$, $\p$ is not the maximal ideal of $R$.  
 We know by part (iv) of Theorems \ref{thm-pa12-B}, \ref{thm-x-nongeneric} that $\varpi\notin\p$, hence
  $\p R[1/p]$ is a maximal ideal of $R[1/p]$.   By the choice of $x$, the conclusion of \cite[Cor.4.21]{pa12} holds,
hence $R_{\p}$ is regular by \cite[Prop.2.30]{pa12}. Now $\dim_{\kappa(\p)} M(\Theta)\otimes_R \kappa(\p)$ is equal to 
$1$ according to \cite[Prop.4.14, 2.22]{pa12}. Hence, $M(\Theta)_{\p}$ is a cyclic $R_{\p}$-module. Since $R$ acts faithfully on $M(\Theta)$, 
$R_{\p}$ acts faithfully on $M(\Theta)_{\p}$, and so  $M(\Theta)_{\p}$ is a free $R_{\p}$-module of rank $1$.
If $p=2$ or $p=3$ then it is explained in the proof of \cite[Prop.2.33]{blocksp2} that the same argument works. 
\end{proof}

\begin{rem} In the potentially crystalline case we know that $R_{\pp}$ is a regular ring for all $\pp$ such that 
$\pp R[1/p]$ is a maximal ideal of $R[1/p]$ by \cite[Thm.3.3.8]{ki08} and one could deduce that $M(\Theta)_{\p}$ is a free $R_{\pp}$-module  
using the Auslander--Buchsbaum formula. To show that the rank is one, it is enough to do so for one $\p$ on each irreducible 
component, and this can be done in the same way as in the proof of Proposition \ref{prop-(H)}.
\end{rem} 

\begin{rem}\label{whyH} The hypothesis \eqref{Hyp} will be used as follows. We will see in section \ref{main-proof} that to every maximal ideal 
$\nn$ of $R[1/p]$ we may attach a unitary $\kappa(\nn)$-Banach space representation $\Pi(\kappa(\nn))$ of $G$. It is shown in \cite[Lem.4.16]{pa12}
that if \eqref{Hyp} is satisfied then for all except finitely many $\nn$ the space of locally algebraic vectors $\Pi(\kappa(\nn))^{\alg}$ in $\Pi(\kappa(\nn))$ 
is an irreducible representation of $G$. Thus we choose our $x$ to avoid the maximal ideals $\nn$ where  $\Pi(\kappa(\nn))^{\alg}$ is reducible. 
\end{rem}

\subsection{Completed tensor products and homomorphisms}\label{subsection-Banach}
In this section we show that the functor $\md \mapsto \md \wtimes_{R^{\psi}_\rho} N$ from compact $R^{\psi}_\rho$-modules to $\dualcat(\OO)$ is fully faithful. 

\begin{prop}\label{prop-mm-isom}
Let $\md$ be a compact $R^{\psi}_{\rho}$-module.
Then the natural map $m\mapsto (n\mapsto m\wtimes n)$ is an isomorphism
\[\md\simto\Hom_{\CO}(N,\md\widehat{\otimes}_{R^{\psi}_{\rho}}N).\]
\end{prop}
\begin{proof} If $\rho$ is generic then $N$ is isomorphic to a projective envelope of $\pi^{\vee}$ in $\CO$, as recalled in \S\ref{subsubsection-generic}. Since the proposition follows from \cite[Lem.2.9]{pa13} if $N$ is projective in $\CO$,  we may assume $\rho$ is non-generic for the rest of the proof. In this case $N$ is flat over $R^{\psi}_{\rho}$ by definition, see \cite[Def.6.11]{pa12}.

First assume that $\md$ is of finite length as an $\cO$-module. We argue by induction on the length of $\md$. If $\md\cong k$, we know by definition of $N$ that $k\widehat{\otimes}_{R^{\psi}_{\rho}}N\cong \beta^{\vee}$, where $\beta$ is defined in \cite[Lem.6.7]{pa12} such that $\check{\VV}(\beta^{\vee})={\VV}(\beta)=\rho$.  We need show that $\Hom_{\CO}(N,\beta^{\vee})\cong k$. By \cite[Thm.6.10]{pa12} and the definition of $N$, the functor $\check{\VV}$ induces an isomorphism $\Hom_{\CO}(N,\beta^{\vee})\cong \Hom_{\cO\br{G_{\Q_p}}}^{\rm cont}(\rho^{\rm un,\psi},\rho)$, where $\rho^{\rm un,\psi}$ is the universal deformation of $\rho$ with determinant $\psi\epsilon$. The assertion then follows from  \cite[Lem.4.1]{pa15}. 
If  the length of $\md$ is $\geq 2$ then let  $\md_1\subsetneq \md$ be a proper $R^{\psi}_{\rho}$-submodule
   and let $\md_2:=\md/\md_1$, so that $0\ra \md_1\ra \md\ra\md_2\ra0$. Since $N$ is flat over $R^{\psi}_{\rho}$, we get an exact sequence:
\[\Hom_{\CO}(N,\md_1 \wtimes_{R^{\psi}_{\rho}}N)\hookrightarrow\Hom_{\CO}(N,\md\wtimes_{R^{\psi}_{\rho}}N)\ra\Hom_{\CO}(N,\md_2\wtimes_{R^{\psi}_{\rho}}N).\]
By inductive hypothesis, we have $\md_i\simto \Hom_{\CO}(N,\md_i\widehat{\otimes}_{R^{\psi}_{\rho}}N)$ for $i=1,2$, hence the assertion using the snake lemma.

Now we treat the general case. Since $\md$ is compact, we may write $\md=\plim_{j}\md_{j}$, where the limit is taken over all the quotients of finite length. Since completed tensor product commutes with projective limits, we get
\begin{multline*}\Hom_{\CO}(N,\md\widehat{\otimes}_{R^{\psi}_{\rho}}N)\cong\Hom_{\CO}(N,\plim_j\md_j\widehat{\otimes}_{R^{\psi}_{\rho}}N)\\
\cong\plim_j\Hom_{\CO}(N,\md_j\widehat{\otimes}_{R^{\psi}_{\rho}}N)=\plim_j\md_j\cong\md.\end{multline*}
\end{proof}

\begin{cor}\label{corollary-passtoN}
 Let $\md_1$ and $\md_2$ be compact $R^{\psi}_\rho$-modules. 
 Then the map $\varphi\mapsto  \varphi\wtimes_{R^{\psi}_\rho} N$ induces an isomorphism 
$$ \Hom^{\cont}_{R^{\psi}_{\rho}}(\md_1, \md_2)\simto\Hom_{\CO}(\md_1\wtimes_{R^{\psi}_\rho} N,\md_2\wtimes_{R^{\psi}_{\rho}}N).$$
\end{cor}
\begin{proof} As in the proof of Proposition \ref{prop-mm-isom} it is enough to prove that the map is 
an isomorphism, when $\md_2$ is of finite length. In this case, $\md_2\wtimes_{R^{\psi}_{\rho}}N$ 
is of finite length in $\dualcat(\OO)$. If $\md_1=\prod_I R^{\psi}_{\rho}$ for some indexing set $I$ then 
$\md_1 \wtimes_{R^{\psi}_{\rho}} N \cong \prod_I N$ and so its Pontryagin dual is isomorphic to 
$\bigoplus_I N^{\vee}$. We have 
\begin{multline*}\Hom_{\CO}(\md_1\wtimes_{R^{\psi}_\rho} N,\md_2\wtimes_{R^{\psi}_{\rho}}N )
\cong\Hom_{G}\big((\md_2\widehat{\otimes}_{R^{\psi}_{\rho}}N)^\vee, \bigoplus_I N^{\vee}\big)\\
\cong \bigoplus_I \Hom_{G}((\md_2\widehat{\otimes}_{R^{\psi}_{\rho}}N)^\vee,  N^{\vee})\cong \bigoplus_I \Hom_{\dualcat(\OO)}(N, \md_2\widehat{\otimes}_{R^{\psi}_{\rho}}N), \end{multline*}
where the second isomorphism holds, because $(\md_2\wtimes_{R^{\psi}_{\rho}}N)^{\vee}$ is of finite length in $\Mod^{\sm}_G(\OO)$. A similar argument shows that 
$$\Hom^{\cont}_{R^{\psi}_{\rho}} (\md_1, \md_2)\cong \bigoplus_I \Hom^{\cont}_{R^{\psi}_{\rho}} (R^{\psi}_{\rho}, \md_2)\cong \bigoplus_I \md_2.$$
The assertion in this case now follows from Proposition \ref{prop-mm-isom}.
 The general case follows by choosing a presentation 
 $\prod_J R^{\psi}_{\rho} \rightarrow \prod_I R^{\psi}_{\rho} \rightarrow \md_1\rightarrow 0$ and applying the functors  $\Hom^{\cont}_{R^{\psi}_{\rho}}(\ast, \md_2)$
 and $\Hom_{\CO}(N\wtimes_{R^{\psi}_\rho} \ast ,N\wtimes_{R^{\psi}_{\rho}}\md_2)$ to it. 
 \end{proof}

\section{Representation theoretic input}\label{section-diagram}

In this section, we will prove a key result about the structure of bounded $G$-invariant $\cO$-lattices inside certain locally algebraic representations, 
using a result of Vign\'eras \cite{vi} on lattices in $G$-equivariant coefficient systems on the Bruhat--Tits tree as a key ingredient.

\subsection{Diagrams}\label{subsection-diagram}
 Let $\mathfrak{R}_0$ be the $G$-normalizer of $K$ so that $\mathfrak{R}_0=KZ$, and $\mathfrak{R}_1$ be the $G$-normalizer of $I$  so that $\mathfrak{R}_1$ is generated by $I$ and $t:=\smatr{0}1p0$ as a group. One checks that $\mathfrak{R}_0\cap \mathfrak{R}_1=IZ$.

Let $A$ be a commutative ring, typically $A=L,\cO$ or $k$. By a diagram $D$ (for $\GL_2$) of $A$-modules, we mean the data $(D_0,D_1,r)$, 
where $D_0$ is an $A[\mathfrak{R}_0]$-module, $D_1$ is an $A[\mathfrak{R}_1]$-module and $r:D_1\ra D_0$ is an $IZ$-equivariant 
homomorphism of $A$-modules. This data determines uniquely a $G$-equivariant coefficient system on the Bruhat--Tits tree of $G$. 
The homology of this coefficient system is computed by a two term complex of $G$-representations:
$$\partial: \cInd_{\mathfrak{R}_1}^GD_1\otimes\delta_{-1}\ra \cInd_{\mathfrak{R}_0}^GD_0,$$
where $\delta_{-1}$ denotes the smooth character of $\mathfrak{R}_1$ (to $A$) of order 2 sending $g$ to $(-1)^{v_p(\det g)}$,
see \cite[\S3]{pa09}. Explicitly the map $\partial$  is determined by the formula
\begin{equation}\label{equation-define-partial}
\partial([\mathrm{Id},x])=[\mathrm{Id},r(x)]-[t,r(t^{-1}\cdot x)]\in \cInd_{\mathfrak{R}_0}^GD_0,\ \ \ \forall x\in D_1\otimes \delta_{-1},\end{equation}
where $[g, x]$ denotes the unique function supported on $\mathfrak R_i g^{-1}$ and taking value $x$ on $g^{-1}$.
The kernel and cokernel of $\partial$ are denoted by $H_1(D)$  and $H_0(D)$ respectively, so that we have an exact sequence
\begin{equation}\label{equation-define-H0}
0\ra H_1(D)\ra \cInd_{\mathfrak{R}_1}^GD_1\otimes\delta_{-1}\overset{\partial}{\ra}\cInd_{\mathfrak{R}_0}^GD_0\ra H_0(D)\ra0.\end{equation}
A short exact sequence of diagrams of $A$-modules $0\ra D'\ra D\ra D''\ra0$ gives a long exact sequence of $G$-representations:
\begin{equation}\label{equation-sequence-H1-H0}
0\ra H_1(D')\ra H_1(D)\ra H_1(D'')\ra H_0(D')\ra H_0(D)\ra H_0(D'')\ra0.\end{equation}
For   an $A$-representation $\pi$ of $G$, we define a constant diagram $$\mathcal{K}(\pi):= (\pi|_{\mathfrak{R}_0},\pi|_{\mathfrak{R}_1},\id).$$ One has that $H_0(\mathcal{K}(\pi))\cong \pi$ and $H_1(\mathcal{K}(\pi))=0$.
\medskip

Let $\pi$ be a smooth  $k$-representation of $G$ of finite length and  with a central character.
Let $D=(D_0,D_1,r)$ be a sub-diagram of $\mathcal{K}(\pi)$, i.e. such that $D_1\overset{r}{\hookrightarrow} D_0\hookrightarrow \pi$ are injective.
We naturally have a $G$-equivariant morphism
\[\theta: H_0(D)\ra H_0(\mathcal K(\pi))\cong\pi.\]
The main result of this section is the following.
\begin{prop}\label{prop-diagram}
If  $H_0(D)$ is of finite length as a $G$-representation then $\theta$ is an injection.
\end{prop}

\begin{proof}
After replacing $\pi$ by the image of $\theta$, we may assume that $\pi$ is generated by $D_0$ as a $G$-representation.   Consider the exact sequence of diagrams
\[0\ra D\ra \mathcal{K}(\pi)\ra Q\ra0\]
where $Q$ is the quotient diagram $\mathcal{K}(\pi)/D$. It
 gives a long exact sequence of $G$-representations
\[0\ra H_1(Q)\ra H_0(D) \overset{\theta}{\ra} \pi\ra H_0(Q)\ra0.\]
Here we have used the facts that $H_0(\mathcal{K}(\pi))\cong\pi$ and $H_1(\mathcal{K}(\pi))=0$. Moreover, since  by assumption  $H_0(D)$  is of finite length as  a $G$-representation,  so is $H_1(Q)$. We need to show that this forces
\[H_1(Q)=0.\]
The next lemma allows us  to conclude, since by definition $H_1(Q)$ is a sub-rep\-re\-sen\-ta\-tion of $\cInd_{\mathfrak{R}_1}^G(Q_1\otimes\delta_{-1})$, if we write $Q=(Q_0,Q_1,r')$.\end{proof}

\begin{lem}
Let $M$ be a  smooth $k$-representation of $\mathfrak{R}_1$ on which $Z$ acts by a character and $V$ be a $G$-subrepresentation of $\cInd_{\mathfrak{R}_1}^GM$. If $V$ is of finite length, then $V=0$.
\end{lem}
\begin{proof} We assume $V$ is non-zero and search for a contradiction.
We may assume $V$ is irreducible, since any non-zero representation of finite length has an irreducible subrepresentation. If $f \in V \subset \cInd_{\mathfrak{R}_1}^GM$
then  $f$ is supported only on finitely many cosets $\mathfrak{R}_1g$ and hence its values are contained in a finitely generated subrepresentation of $M$. Since $V$ is irreducible we thus may assume
that  $M$ is finitely generated. Since $\mathfrak{R}_1$ is compact modulo centre, which acts by a character on $M$, this implies that $M$ is finite dimensional over $k$. 
It is enough to prove the statement with $\mathfrak R_1$ replaced by $\mathfrak R_0$ since the injection $M\hookrightarrow \Ind_{IZ}^{\mathfrak R_1} M|_{IZ}$ gives injections of $G$-representations: 
$$   \cInd_{\mathfrak{R}_1}^GM\hookrightarrow \cInd_{IZ}^ G (M|_{IZ})\cong \cInd_{\mathfrak R_0}^G (\Ind_{IZ}^{\mathfrak R_0} (M|_{IZ})). $$

So let $M$ be a  smooth finite dimensional representation of $\mathfrak R_0$ with a central character and let $V$ be an irreducible $G$-subrepresentation 
contained in $\cInd_{\mathfrak R_0}^G M$. We may choose an exhaustive filtration $\Fil^{\bullet} M$ of $M$ by $\mathfrak R_0$-subrepresentations
such that the graded pieces are irreducible representations of $\mathfrak R_0$. Then $\cInd_{\mathfrak R_0}^G (\Fil^{\bullet} M)$ is an exhaustive 
filtration of $\cInd_{\mathfrak R_0}^G M$ by $G$-subrepresentations, the graded pieces of which are isomorphic to 
$\cInd_{\mathfrak R_0}^G \sigma_i$ where $\sigma_i$ is an irreducible representation of $\mathfrak R_0$. 
Since $V$ is finitely generated it will be contained in $\cInd_{\mathfrak R_0}^G (\Fil^{i} M)$ for some $i\ge 0$ and since $V$ is non-zero its image 
in some graded piece will be non-zero. Hence we have reduced to a situation where $M$ is an irreducible representation of $\mathfrak R_0$.
In this case \cite[Prop.18]{BL} implies that if $W$ is any non-zero $G$-invariant subspace of  $\cInd_{\mathfrak R_0}^G M$ then 
the subspace $W^{I_1}$ of $I_1$-invariants of $W$ is  infinite dimensional. However, it follows from the classification of irreducible representations of $G$
in  \cite{BL, Br03} that $V^{I_1}$ is finite dimensional. This contradiction finishes the proof. 
\end{proof}

The proof of Proposition \ref{prop-diagram} also yields the following statement.

\begin{cor}   $\ker(\theta)$ does not contain non-zero finite length subrepresentations.
\end{cor}

\begin{rem} The proof of Proposition \ref{prop-diagram} works for $G=\GL_2(F)$, where $F$ is a finite extension of $\Qp$.
However, if $F\neq \Qp$, in view of the main result of \cite{schraen}, we don't expect to find interesting examples, where the conditions of the proposition are satisfied
and $D_0$ is  finite dimensional.
\end{rem}

\begin{lem}\label{Omod} Let $\varphi: M\rightarrow M'$ be a morphism of $\OO$-modules. Assume that $M$ is a free $\OO$-module and 
$M'$ is $\OO$-torsion free. If  $\varphi\otimes_\OO L$ is surjective and $\varphi\otimes_{\OO} k$ is injective, then $\varphi$ is an isomorphism. 
\end{lem}
\begin{proof} 
We first show that $\varphi$ is injective. Let $K$  be the kernel  of $\varphi$. Since $M'$ is $\OO$-torsion free, so is the image of $\varphi$. Therefore the morphism   $ K\otimes_{\OO}k\ra M\otimes_{\OO}k$ is injective and  we deduce that $K\otimes_{\OO}k=0$ as $\varphi\otimes_{\OO}k$ is injective by assumption. On the other hand, since $\OO$ is a PID, submodules of free modules are free. Hence, $K$ is a free $\OO$-module and $K\otimes_{\OO}k=0$ implies $K=0$. 

Now let $C$ be the cokernel of $\varphi$ and consider the exact sequence of $\OO$-modules $0\rightarrow M\rightarrow M'\rightarrow C\rightarrow 0$.
Since $M'$ is $\OO$-torsion free and $\varphi$ is injective modulo $\varpi$, 
 we obtain that $C[\varpi]$ the submodule of elements in $C$ killed by $\varpi$ is zero.   This implies that the map $C\rightarrow C\otimes_{\OO} L$ is an injection; but $C\otimes_{\OO}L=0$ as $\varphi\otimes_{\OO}L$ is surjective,  hence $C=0$  as desired. 
 \end{proof}

We  now consider locally algebraic representations.   The next result plays a crucial role in the proof of our main theorem.

\begin{thm}\label{thm-diagram}
Let $\Phi$ be a locally algebraic $L$-representation of $G$ of finite length. Assume that $\Phi$  carries a bounded $G$-invariant $\cO$-lattice $\Phi^0$, which is finitely generated as an $\cO[G]$-module and such that $\Phi^0\otimes_{\cO}k$ is of finite length as a $k[G]$-module. Let $X=(X_0,X_1,r)\hookrightarrow \mathcal{K}(\Phi)$ be a sub-diagram of finite dimensional $L$-vector spaces and let $\mathcal{X}:=(\mathcal{X}_0,\mathcal{X}_1,r)$ where $\mathcal{X}_i=X_i\cap \Phi^0$. If $H_0(X)\cong \Phi$, then $H_0(\mathcal{X})\cong \Phi^0$ and $H_0(\mathcal X\otimes_{\OO} k)\cong \Phi^0\otimes_\OO k$.
\end{thm}
\begin{proof}
The  assumptions imply that the condition 2) in \cite[Prop.0.1]{vi}  holds for $\mathcal{X}$, so $H_0(\mathcal{X})$ is a free $\cO$-module such that
\[H_0(\mathcal{X})\otimes_{\cO} L\cong H_0(X),\ \  H_0(\mathcal{X})\otimes_{\cO}k\cong H_0(\mathcal{X}\otimes_{\cO}k),\]
see the discussion after Lemma 0.7 \emph{loc. cit.}, and  \cite[\S4]{pa09}.
Since both $\Phi^0$ and $H_0(\mathcal{X})$ are lattices in $H_0(X)\cong \Phi$ and are finitely generated as $\cO[G]$-modules, they are commensurable. Since $\Phi^0\otimes_{\cO}k$ is of finite length as a $k[G]$-module, so is $H_0(\mathcal{X})\otimes_{\cO}k$ and has the same length as $\Phi^0\otimes_{\cO}k$ by \cite[Lem.4.3]{pa10}.    However, it follows from the definition of $\mathcal{X}$   that the morphisms $\mathcal{X}_1\otimes_{\cO} k\hookrightarrow \mathcal{X}_0\otimes_{\cO}k\hookrightarrow \Phi^0\otimes_{\cO} k$ are injective, so Proposition \ref{prop-diagram} implies that the $G$-equivariant morphism $H_0(\mathcal{X}\otimes_{\cO}k)\ra \Phi^0\otimes_{\cO}k$ is injective. Hence the map is an isomorphism, as both the source and the target have the same length.  The result follows from Lemma \ref{Omod} with $M=H_0(\mathcal{X})$ and $M'=\Phi^0$.
\end{proof}
\subsection{Types for principal series}\label{types_princ} We recall some results about principal series  types from \cite{he}. Let $\theta_1, \theta_2:\Z_p^{\times}\ra L^{\times}$ be smooth characters. Let  $c$ be the conductor of $\theta_1\theta_2^{-1}$  so that $c\in\Z_{\geq0}$ is the smallest integer such that $\theta_1\theta_2^{-1}$ is trivial on $1+p^c \Zp$.
Consider the type $\tau=\theta_1\oplus \theta_2$ of $I_{\Q_p}$, where we view $\theta_1$ and $\theta_2$  as  characters of $I_{\Q_p}$ via the local class field theory. Then using \cite{he} the representation $\sigma^{\cri}(\tau)$ defined in \S \ref{subsection-x} maybe described as follows. 
If $c=0$ then $\theta_1=\theta_2 (=:\theta)$ and $\sigma^{\cri}(\tau)=\theta\circ \det$. If $c\ge 1$ then 
$\sigma^{\rm cr}(\tau)$ is just the induction $\Ind_{J_c}^K\theta_1\otimes \theta_2$, where $J_c$ denotes the subgroup $\smatr{\Z_p^{\times}}{\Z_p}{p^c\Z_p}{\Z_p^{\times}}$;
in this case the representations $\sigma(\tau)$ and $\sigma^{\cri}(\tau)$ coincide.

\begin{lem}\label{K_c-invariants} Assume that $\tau=\theta\oplus \ide$ with $\theta\neq \ide$. Let $\pi$ be a smooth admissible representation of $G$ on an $L$-vector space generated as a $G$-representation by its $\sigma^{\cri}(\tau)$-isotypic subspace. Then $\pi^{K_c}$ as a $K$-representation is isomorphic to a finite direct sum of copies of $\sigma^{\cri}(\tau)$.
\end{lem}
\begin{proof} Since $\pi$ is admissible its $\sigma^{\cri}(\tau)$-isotypic subspace is finite dimensional. Hence $\pi$ is finitely generated as a $G$-representation, and since $\pi$ is admissible it  is of finite length.

Since $\theta\neq \ide$ by assumption the representation $\sigma^{\mathrm{cr}}(\tau)$ is not only typical for the Bernstein component corresponding to $\tau$, but is a type for that 
component in the sense of \cite{bk-types}. This implies  that $\Hom_K(\sigma^{\cri}(\tau), \pi')\neq 0$ for all irreducible subquotients of $\pi$. Hence, if $\pi'$ is an irreducible subquotient of 
$\pi$ then $\pi'\cong \Ind_P^G \chi_1\otimes\chi_2$, such that $\chi_2$ is unramified and $\chi_1|_{\Zp^{\times}}= \theta$. 
Moreover, we have an isomorphism of $K$-representations 
$$(\Ind_P^G \chi_1\otimes\chi_2)^{K_c}\cong \Ind_{(K\cap P)K_c}^K \chi_1\otimes \chi_2\cong \sigma^{\mathrm{cr}}(\tau).$$
Since taking $K_c$-invariants defines an exact functor on the category of smooth $G$-representations on $L$-vector spaces and since the category of smooth $K$-rep\-re\-sen\-ta\-tions
on $L$-vector spaces is semi-simple, we conclude that 
$\pi^{K_c}$ is isomorphic to a finite direct sum of copies of $\sigma^{\mathrm{cr}}(\tau)$. 
\end{proof}

\begin{rem} Since $\sigma(\theta_1\oplus \theta_2)\cong \sigma(\theta_1\theta_2^{-1}\oplus \ide) \otimes \theta_2\circ \det$, if $\theta_1\neq \theta_2$ then we may twist by a character 
to get to a situation where Lemma \ref{K_c-invariants} applies.
\end{rem}

\section{The main theorem}\label{section-proof}

We keep the notation of the previous section. In this section we prove the main theorem \ref{theorem-main}. We write $R$ for either $R^{\psi}_{\rho}(\wt,\tau)$ or 
$R^{\psi, \mathrm{cr}}_\rho(\w, \tau)$ and $V$ for either $\sigma(\w, \tau)$ or $\sigma^{\mathrm{cr}}(\w, \tau)$ to simplify the notation. In \S\ref{ideals}, \S\ref{criterion}
we place no restriction on the type $\tau$. We assume that $\tau$ is a principal series type in the end of \S\ref{main-proof}. 
By Lemma \ref{tres_peu}, if \eqref{Hyp} is not satisfied, then $R$ is a Cohen-Macaulay ring. \emph{So we may assume  \eqref{Hyp} holds in the rest and choose $x\in R^{\psi}_{\rho}$ as in Proposition \ref{prop-(H)}.}

\subsection{The ideals $\fa_n$}\label{ideals} 

Let $\Theta$ be a $K$-invariant $\cO$-lattice in $V$. Let $\fa\subset R^{\psi}_{\rho}$ be the annihilator of the $R^{\psi}_{\rho}$-module $M(\Theta)$ and for each $n\in\Z_{\geq0}$,  let $\fa_n$  be  the annihilator of the $R^{\psi}_{\rho}$-module
\[M(\Theta)/x^nM(\Theta)\cong \Hom_{\cO}(\Hom_{\OK}^{\cont}(N/x^nN,\Theta^d),\cO).\]
Then $\fa,\fa_n$ do not depend on the choice of $\Theta$.  We have $\fa\subseteq \fa_{n+1}\subseteq \fa_{n}$ for any $n\geq 1$, so that there are natural surjections
\begin{equation}\label{equation-R-surjection}
R \twoheadrightarrow R^{\psi}_{\rho}/\fa_{n+1}\twoheadrightarrow R^{\psi}_{\rho}/\fa_{n}.\end{equation}
Remark that, since $\Theta$ is  $\cO$-flat, $M(\Theta)/x^nM(\Theta)$ is $\cO$-flat by Theorem \ref{thm-pa12-B}(i) and Theorem \ref{thm-x-nongeneric}, hence $R^{\psi}_{\rho}/\fa_n$ is an $\cO$-flat algebra for each $n\geq 1$.

\begin{prop}\label{prop-R=limit}
We have $\fa=\bigcap_{n\geq 1}\fa_n$ and $R\cong \plim_{n\geq 1}R^{\psi}_{\rho}/\fa_n$.
\end{prop}
\begin{proof}
Since $R^{\psi}_{\rho}$ is a noetherian local ring, $M(\Theta)$ is a finitely generated $R^{\psi}_{\rho}$-module and $x\in \m$,  we get a natural isomorphism by the Artin-Rees lemma:
\[M(\Theta)\cong \plim_{n\geq 1}M(\Theta)/x^nM(\Theta).\]
This implies the assertion by definition of $\fa$ and $\fa_n$.
\end{proof}

\begin{rem}\label{rem-x}
For $m\leq n$, one checks that $x^{n-m}\fa_{m}\subset \fa_n$, hence there is a natural morphism $x^{n-m}: R^{\psi}_{\rho}/\fa_{m}\ra R^{\psi}_{\rho}/\fa_{n}$,  given by multiplication by $x^{n-m}$,  whose composition with the projection $R^{\psi}_{\rho}/\fa_n\twoheadrightarrow R^{\psi}_{\rho}/\fa_{n-m}$ is zero.
\end{rem}

Since  $x^n\in \fa_n$,  the surjection (\ref{equation-R-surjection}) factors through  $R/x^nR\twoheadrightarrow R^{\psi}_{\rho}/\fa_n$; the target being $\cO$-flat, it further factors through
\[\alpha_n: (R/x^nR)_{\rm tf}\twoheadrightarrow R^{\psi}_{\rho}/\fa_n,\]
where $(R/x^nR)_{\rm tf}$ denotes the largest $\cO$-flat quotient of $R/x^nR$.

\begin{prop}\label{prop-rank=ne}
(i) The surjection $\alpha_n$ is an isomorphism.

(ii) The  ring $R^{\psi}_{\rho}/\fa_n$ is a finite flat  $\cO$-algebra of rank $n\cdot e$, where $e:=e(R/\varpi R)$.
\end{prop}
\begin{proof}
(i) By definition of $\fa_n$, it suffices to show that $(R/x^nR)_{\rm tf}$ acts faithfully  on $M(\Theta)/x^nM(\Theta)$.

Let $\p\in \Spec R$ be a minimal prime ideal over $(x)$.  As $x$ is regular in $R$, Krull's principal ideal theorem implies that $\p$ has height 1; in particular, $\p$ is not the maximal ideal of $R$ because $R$ is of Krull dimension 2.   We know by Theorem \ref{thm-pa12-B}(iii) and Theorem \ref{thm-x-nongeneric} that $\varpi\notin\p$, hence $\p$ gives rise to an element in $\Spec R[1/p]$. Then we have $R_{\p}\cong M(\Theta)_{\p}$ by Proposition \ref{prop-(H)}, which induces an isomorphism
\begin{equation}\label{equation-isom-x^n}
(R/x^nR)_{\p}\cong (M(\Theta)/x^nM(\Theta))_{\p}.
\end{equation}
In particular, $(R/x^nR)_{\p}$ is $\cO$-flat and acts faithfully on $(M(\Theta)/x^nM(\Theta))_{\p}$.

Now let $u\in R/x^nR$ and assume that it acts trivially on $M(\Theta)/x^nM(\Theta)$. Then its image in $(R/x^nR)_{\p}$ must be $0$ for each $\p$ as above, hence $u$ lies in the kernel of \[R/x^nR\ra \prod_{\p}(R/x^nR)_{\p},\]
where $\p$ runs over all minimal prime ideals of $R$ over $(x)$.
But this kernel, being supported only (possibly) at $\m_{R}$, is $0$-dimensional, hence is exactly the $\cO$-torsion submodule of $R/x^nR$ because each $(R/x^nR)_{\p}$ is $\cO$-flat.

(ii) Since $(R/x^nR)\otimes_{\cO}L$ is an Artinian ring, we have an isomorphism  $(R/x^nR)\otimes_{\cO}L\cong \prod_{\p}(R/x^nR)_{\p}$, where $\p$ runs over all minimal prime ideals of $R$ over $(x)$.   It then follows from (\ref{equation-isom-x^n})  that
\begin{multline*}  \mathrm{rank}_{\cO}(R/x^nR)_{\rm tf}=\sum_{\p}\dim_{L}(R/x^nR)_{\p} \\
=\sum_{\p}\dim_{L}(M(\Theta)/x^nM(\Theta))_{\p}=\mathrm{rank}_{\cO}M(\Theta)/x^nM(\Theta)=n\cdot e,\end{multline*}
 where the last equality comes from Theorem \ref{thm-pa12-B} and Theorem \ref{thm-x-nongeneric}.
\end{proof}

Set $Q_n=N/\fa_n N$. Then $Q_n$ is a quotient of $N/x^nN$ because $x^n\in\fa_n$. We note that although $R^{\psi}_{\rho}/\fa_n$ is $\cO$-flat, $Q_n$ need not  be because $N$ is not always flat over $R^{\psi}_{\rho}$, see the proof of Lemma \ref{nearly_there}
for details.

\begin{lem}\label{lem-surjective}
Let $m\in\Z_{\geq 1}$ and $\Xi$ be a $K$-invariant $\cO$-lattice of $\oplus_{i=1}^{m}V^d$. Then the morphism  $x^{n-1}:Q_1 \ra Q_n$ induces a surjection
\[\Hom_{\OK}^{\cont}(Q_n,\Xi)\twoheadrightarrow \Hom_{\OK}^{\cont}(Q_1,\Xi).\]
\end{lem}
\begin{proof}
We first claim that the natural injection
 \[\Hom_{\OK}^{\cont}(Q_n,\Xi)\hookrightarrow \Hom_{\OK}^{\cont}(N/x^nN,\Xi)\]
 is an isomorphism. Indeed, if $f\in \Hom_{\OK}^{\cont}(N/x^nN,\Xi)$ which we view as an element in $\Hom_{\OK}^{\rm cont}(N,\oplus_{i=1}^mV^d)$ in an obvious way, and if $a\in \fa_n$, then $f\circ a=0$ by definition of $\fa_n$, so that $f$ factors as $N\ra N/aN\ra \Xi$. Since this holds for all  $a\in\fa_n$, we see that $f$ factors as $N\ra N/\fa_nN\ra \Xi$, hence the result.

The short exact sequence $0\ra N/xN\overset{x^{n-1}}{\longrightarrow} N/x^nN\ra N/x^{n-1}N\ra0$ induces
\begin{multline*}
0\ra \Hom_{\OK}^{\cont}(N/x^{n-1}N,\Xi)
\ra \Hom_{\OK}^{\cont}(N/x^nN,\Xi)\\ \ra \Hom_{\OK}^{\cont}(N/xN,\Xi)\ra0
\end{multline*}
which is still exact since $N/x^{n-1}N$ is  projective in $\Mod_{K,\zeta}^{\rm pro}(\cO)$. The result then follows from the claim.
\end{proof}

\subsection{A criterion for Cohen--Macaulayness}\label{criterion} In Lemma \ref{lem-criterion} below we devise a criterion for $R$ to be Cohen--Macaulay. This criterion will be checked in the next subsection to complete the proof of Theorem \ref{theorem-main}. 

\begin{lem} \label{lemma-CM-equiv}
The following are equivalent:
\begin{itemize}
\item[(i)] $R$ is Cohen--Macaulay;
\item[(ii)] $(x, \varpi)$ is a regular sequence in $R$;
\item[(iii)] the chain complex 
$0\ra R \overset{x}{\ra} R\ra R^{\psi}_{\rho}/\fa_1\ra0$
is exact;
\item[(iv)] the chain complex: 
\begin{equation}\label{equation-complex-fa}
  0\ra R^{\psi}_{\rho}/\fa_{n-1}\overset{x}{\ra} R^{\psi}_{\rho}/\fa_n\ra R^{\psi}_{\rho}/\fa_1\ra0
\end{equation}
is exact for all $n\ge 1$. 
\end{itemize}
\end{lem} 
\begin{proof} Since $(x, \varpi)$ is a system of parameters for $M(\Theta)$ and $R$ acts faithfully on $M(\Theta)$, 
$(x, \varpi)$ is a system of parameters for $R$. Hence (i) is equivalent to (ii). The sequence $(x, \varpi)$ is regular 
if and only if multiplication by $x$ is injective and $R/xR$ is $\OO$-torsion free. Proposition \ref{prop-rank=ne}(i) implies that (ii) is
equivalent to (iii). Since $R\cong \plim_{n\geq 1} R^{\psi}_{\rho}/\fa_n$ by Proposition \ref{prop-R=limit} and the Mittag-Leffler condition is satisfied, by passing to the limit we see that 
(iv) implies (iii). On the other hand, if $R$ is Cohen--Macaulay then $(x^n, \varpi)$ is a regular sequence for all $n\ge 1$, hence 
$R^{\psi}_{\rho}/\fa_n= (R/x^nR)_{\mathrm{tf}}=R/x^nR$ for all $n\ge 1$. Since $x$ is $R$-regular we deduce that (\ref{equation-complex-fa}) is exact for all $n\ge 1$. 
Hence (i) implies (iv).
\end{proof}

\begin{lem}\label{lem-equivalent}
The following conditions are equivalent:
\begin{enumerate}
\item[(i)]
 The complex (\ref{equation-complex-fa}) is exact for all $n\geq 1$;

\item[(ii)]
$x: (R^{\psi}_{\rho}/\fa_{n-1})\otimes_{\cO}k \ra (R^{\psi}_{\rho}/\fa_n)\otimes_{\cO}k$ is injective for all $n\geq 1$;

\item[(iii)]  $x^{n-1}:(R^{\psi}_{\rho}/\fa_{1})\otimes_{\cO}k\ra (R^{\psi}_{\rho}/\fa_n)\otimes_{\cO}k$ is injective for all $n\geq 1$.
 \end{enumerate}
\end{lem}
\begin{proof}
Denote by $C$ the complex (\ref{equation-complex-fa}) with the term $R^{\psi}_{\rho}/\fa_1$ in degree $0$. Since all the terms in $C$ are $\cO$-torsion free, we get an exact sequence of complexes
\[0\ra C\overset{\varpi}{\ra} C\ra C/\varpi C\ra0.\]
Passing to homology and using the fact that $H_0(C)=0$ we obtain a long exact sequence
\[0\ra H_2(C)\overset{\varpi}{\ra} H_2(C)\ra H_2(C/\varpi C)\ra H_1(C)\overset{\varpi}{\ra} H_1(C)\ra H_1(C/\varpi C)\ra0.\]
The condition  (ii) is equivalent to $H_2(C/\varpi C)=0$. Hence (i) implies (ii). 

Since $(R^{\psi}_{\rho}/\fa_n)\otimes_{\cO}k$ is of dimension $n\cdot e$ over $k$ by Proposition \ref{prop-rank=ne}(ii), if (ii) holds,  
then for dimension reasons $C/\varpi C$ is an exact sequence of $k$-vector spaces, which implies that $H_1(C/\varpi C)=0$
 and therefore (i) holds  by Nakayama's lemma.
 
Part (ii) trivially implies (iii). If (iii) holds then  Proposition \ref{prop-rank=ne}(ii) implies that the chain complex of $k$-vector spaces 
$$0\rightarrow (R^{\psi}_{\rho}/\fa_{1})\otimes_{\cO}k\overset{x^{n-1}}{\longrightarrow}
 (R^{\psi}_{\rho}/\fa_n)\otimes_{\cO}k \rightarrow (R^{\psi}_{\rho}/\fa_{n-1})\otimes_{\cO}k\rightarrow 0$$
is exact for all $n\ge 1$. The snake lemma applied to the following diagram: 
\[\xymatrix{0\rightarrow(R^{\psi}_{\rho}/\fa_{1})\otimes_{\cO}k\ar[d]^{=} \ar[r]^{x^{n-2}} & (R^{\psi}_{\rho}/\fa_{n-1})\otimes_{\cO}k\ar[d]^x \ar[r]&
(R^{\psi}_{\rho}/\fa_{n-2})\otimes_{\cO}k\ar[d]^x \rightarrow 0\\
0\rightarrow  (R^{\psi}_{\rho}/\fa_{1})\otimes_{\cO}k\ar[r]^{x^{n-1}} & (R^{\psi}_{\rho}/\fa_{n})\otimes_{\cO}k\ar[r]&
(R^{\psi}_{\rho}/\fa_{n-1})\otimes_{\cO}k\rightarrow 0}\]
shows that  for all $n\ge1$ we have 
\begin{multline*}
\Ker( (R^{\psi}_{\rho}/\fa_{n-1})\otimes_{\cO}k \overset{x}{\ra} (R^{\psi}_{\rho}/\fa_n)\otimes_{\cO}k) \\ \cong \Ker( (R^{\psi}_{\rho}/\fa_{n-2})\otimes_{\cO}k \overset{x}{\ra} (R^{\psi}_{\rho}/\fa_{n-1})\otimes_{\cO}k).
\end{multline*}
Part (iii) applied with $n=2$ says that $(R^{\psi}_{\rho}/\fa_{1})\otimes_{\cO}k \overset{x}{\ra} (R^{\psi}_{\rho}/\fa_2)\otimes_{\cO} k$ is injective. Hence (ii) holds by an obvious induction. 
\end{proof}

Let $Q_n=N/\fa_n N$ be as in the previous subsection. By Proposition \ref{prop-mm-isom} we have a natural isomorphism 
\begin{equation}\label{equa-isom-Qn}\Hom_{\CO}(N,Q_n\otimes_{\cO}k)\cong (R^{\psi}_{\rho}/\fa_n)\otimes_{\cO}k.
\end{equation}

\begin{lem}\label{lem-criterion}
Assume that there exist both an object $\Upsilon\in \Mod^{\rm pro}_{K,\zeta}(\cO)$ and a continuous   $\OK$-linear map $\phi: Q_1\rightarrow \Upsilon$, such that the following conditions hold:
\begin{enumerate}
\item[(i)] for all $n\ge 1$ the map $x^{n-1}:Q_1 \ra Q_n$ induces 
a surjection
\[\Hom_{\OK}^{\cont}(Q_n,\Upsilon)\twoheadrightarrow \Hom_{\OK}^{\cont}(Q_1,\Upsilon);\]

\item[(ii)] the quotient morphism $\bar{\phi}:Q_1\otimes_{\cO}k\ra \Upsilon\otimes_{\cO}k$
induces an injection
\begin{equation}\label{equation-restriction}
\Hom_{\CO}(N,Q_1\otimes_{\cO}k)\overset{\bar{\phi}\circ}{\hookrightarrow}\Hom_{\OK}^{\cont}(N,\Upsilon\otimes_{\cO}k).\end{equation}
\end{enumerate}
  Then the equivalent conditions in Lemma \ref{lem-equivalent} hold.
\end{lem}
\begin{proof}
By (i), we can  choose  $\psi\in\Hom_{\OK}^{\cont}(Q_n,\Upsilon)$ such that $\phi=\psi\circ x^{n-1}$. Tensoring with $k$ gives the following commutative diagram
\[\xymatrix{Q_1\otimes_{\cO}k\ar^{x^{n-1}}[r]\ar^{\bar{\phi}}[d]& Q_n\otimes_{\cO}k\ar^{\bar{\psi}}[ld]\\
\Upsilon\otimes_{\cO}k&}.\]
Using this we will show that
 the injectivity of (\ref{equation-restriction}) implies the injectivity of
\[x^{n-1}:\Hom_{\CO}(N,Q_1\otimes_{\cO}k)\ra \Hom_{\CO}(N,Q_n\otimes_{\cO}k),\]
hence the result by the isomorphism \eqref{equa-isom-Qn} above.
Indeed, let $$f\in \Hom_{\CO}(N,Q_1\otimes_{\cO}k)$$ be such that $x^{n-1}\circ f=0$, then $\bar{\psi}\circ x^{n-1}\circ f=0$, hence $\bar{\phi}\circ f=0$, which implies $f=0$ by (\ref{equation-restriction}).  
\end{proof}

\subsection{Proof of the main theorem}\label{main-proof} Following \cite[\S4]{pa13} we  define a left exact contravariant functor $\md \mapsto \Pi(\md)$ 
from the category of $R^{\psi}_\rho[1/p]$-modules of finite length to the category of admissible unitary Banach space representations of $G$ by letting 
\[\Pi(\md):=\Hom_{\cO}^{\cont}(\md^0\widehat{\otimes}_{R^{\psi}_{\rho}}N ,L),\]
where $\md^0$ is any finitely generated $R^{\psi}_{\rho}$-submodule  of $\md$, such that $\mathrm{m}^0\otimes_\OO L\cong \mathrm{m}$.

\begin{lem}\label{the_last_straw} If $\md$ is an $(R^\psi_\rho/\mathfrak a_n)[1/p]$-module of finite length then 
$$ \dim_L \Hom_K(V, \Pi(\md))= \dim_L \md.$$
 \end{lem}
 \begin{proof} Using Chinese remainder theorem we may assume that $\md$ is supported on a single maximal ideal $\nn$ of $R[1/p]$. Proposition 2.22 of \cite{pa12} implies that 
  $$\dim_L \Hom_K(V, \Pi(\md))= \dim_L \md \otimes_{R^{\psi}_\rho} M(\Theta)=\dim_L \md\otimes_R M(\Theta).$$
  Since $\nn$ contains $\fa_n$ it will contain $x$ chosen as  in \S\ref{avoid}. It follows from Proposition \ref{prop-(H)} that the localization of $M(\Theta)[1/p]$ at $\nn$ is a free $(R[1/p])_{\nn}$-module of rank $1$. Hence $\md\otimes_R M(\Theta)\cong \md$.
  \end{proof}
 
Let $\n$ be a maximal ideal of $ R[1/p]$ and let $\kappa(\n)$ be its residue field. Since $\kappa(\nn)$ is a finite dimensional $L$-vector space we may apply the functor $\Pi$ to  $\kappa(\nn)$.
Then $\Pi(\kappa(\n))$   is a $\kappa(\n)$-Banach space representation such that 
$\check{\VV}(\Pi(\kappa(\n)))\cong \rho_{\n}^{\rm un}$. We denote by $\Pi(\kappa(\n))^{\rm alg}$ the subspace of locally algebraic vectors in $\Pi(\kappa(\n))$. 
Let us assume that $\n\in\Spec (R^{\psi}_{\rho}/\fa_1)[1/p]$, where $\fa_1$ is the $R^{\psi}_{\rho}$-annihilator of $M(\Theta)/x M(\Theta)$ and $x$ is chosen as  in \S\ref{avoid}.  Then $\Pi(\kappa(\n))^{\rm alg}$ is \emph{irreducible}, see Remark \ref{whyH}.

\begin{prop}\label{happy_chinese_new_year} If $\n$ is as above then one of the following holds: 
\begin{enumerate}
\item[(i)] $\Pi(\kappa(\n))$ is an absolutely irreducible non-ordinary $\kappa(\n)$-Banach space representation;

\item[(ii)] there is a non-split exact sequence of admissible $\kappa(\n)$-Banach space representations of $G$:
\[0\ra \Pi_1\ra \Pi(\kappa(\n))\ra \Pi_2\ra0,\]
such that both $\Pi_1$ and $\Pi_2$ are absolutely irreducible. In this case, $\rho$ is reducible (non-split), and if $\rho\cong \smatr{\delta_2}*0{\delta_1}$,
 then  $\check{\VV}(\Pi_i)$ is a character congruent to $\delta_i$ modulo $\varpi$. Moreover, $\Pi_1^{\alg}\cong \Pi(\kappa(\nn))^{\alg}$ and $\Pi_2^{\alg}=0$.
\end{enumerate}
\end{prop}
\begin{proof} The claim follows from \cite[Prop.4.9]{pa12}, which describes the structure of  $\Pi(\kappa(\n))$ for an arbitrary maximal ideal $\n$ of $R[1/p]$.
Note that  part (ii) (b) of that proposition can not occur as in this case $\Pi(\kappa(\n))^\alg=0$ and this would contradict   \cite[Prop.4.12]{pa12} and part (ii) (c) cannot occur as we assume \eqref{Hyp}. The last assertion in (ii) follows from the proof of \cite[Prop.4.12]{pa12}.
\end{proof}

Recall that $R^{\psi}_{\rho}/\fa_1$ is a finite free $\cO$-module of rank $e$ by Proposition 
\ref{prop-rank=ne}, where $e$ is the Hilbert--Samuel multiplicity of $R/(\varpi)$. 
Let $\Pi=\Pi(\md)$ with $\md^0=R^{\psi}_{\rho}/\fa_1$, $\md=(R^{\psi}_{\rho}/\fa_1)[1/p]$ so that $\Pi$ is an admissible unitary $L$-Banach space 
representation of $G$ and $\Pi^0:=\Hom_{\cO}^{\cont}(\md^0\widehat{\otimes}_{R^{\psi}_{\rho}}N ,\OO)$ is its  unit ball.  Unraveling the definitions we have
$$\Pi= \Hom^{\cont}_{\OO}(N/\fa_1 N, L)=\Hom_{\OO}^{\cont}(Q_1, L)=\Hom^{\cont}_{\OO}(R/xR\wtimes_{R^{\psi}_{\rho}} N , L),$$
$$\Pi^0=\Hom^{\cont}_{\OO}(N/\fa_1 N, \OO)=\Hom_{\OO}^{\cont}(Q_1, \OO)=\Hom^{\cont}_{\OO}(R/xR\wtimes_{R^{\psi}_{\rho}} N , \OO).$$
 
\begin{lem}\label{isotypic-sub-gen} The $V$-isotypic subspace in $\Pi^{\rm alg}$ generates it as a $G$-rep\-re\-sen\-ta\-tion.
\end{lem}
\begin{proof}  After enlarging $L$, we may assume that $\kappa(\nn)=L$ for all maximal ideals of $(R^{\psi}_{\rho}/\fa_1)[1/p]$ and that each irreducible subquotient of $\Pi^{\rm alg}$ is absolutely irreducible. 
Since $(R^{\psi}_{\rho}/\fa_1)[1/p]$ is of dimension $e$ over $L$ we may choose a filtration of  length $e$ of $(R^{\psi}_{\rho}/\fa_1)[1/p]$ by 
submodules such that each graded piece is isomorphic to $\kappa(\nn)$ for some $\nn\in \mSpec (R^{\psi}_{\rho}/\fa_1)[1/p]$. This induces a filtration $\{\Pi_i\}_{0\le i\le e}$
of $\Pi$ by closed subrepresentations, such that $\Pi_0=0$, $\Pi_e=\Pi$ and $\Pi_i/\Pi_{i-1}$ is isomorphic to a closed subrepresentation of $\Pi(\kappa(\nn))$ 
for some $\nn\in \mSpec (R^{\psi}_{\rho}/\fa_1)[1/p]$.  

The subspace $\Pi_i^{\alg}$ of locally algebraic vectors in $\Pi_i$ is equal to  $\Pi_i \cap \Pi^{\alg}$. 
Since $\Pi(\kappa(\n))^{\rm alg}$ is irreducible for any $\n\in\mSpec(R^{\psi}_{\rho}/\fa_1)[1/p]$ by the choice of $x$, for all $1\le i\le e$ the quotient
$\Pi^{\alg}_i/\Pi^{\alg}_{i-1}$ is either zero or isomorphic to $\Pi(\kappa(\nn))^{\alg}$. Since $\Hom_K(V, \Pi(\kappa(\nn))^{\alg})$ is one dimensional by Lemma \ref{the_last_straw}
we conclude that
$$\dim_L \Hom_K(V,  \Pi^{\alg}_i/\Pi^{\alg}_{i-1})\le 1, \quad \dim_L \Hom_K(V,  \Pi^{\alg}_i) \le i,$$
for all $1\le i\le e$. Since $\Pi_e=\Pi$ and $\Hom_K(V, \Pi^{\alg})$ is $e$-dimensional by Lemma \ref{the_last_straw}, we conclude that the above inequalities are equalities. In particular, 
for $1\le i\le e$ we obtain an exact sequence of locally algebraic representations:
\begin{equation}\label{locally_alg_vec_exact}
0 \rightarrow \Pi_{i-1}^{\alg}\rightarrow \Pi_i^{\alg} \rightarrow (\Pi_i/ \Pi_{i-1})^{\alg} \rightarrow 0
\end{equation}
with $(\Pi_i/ \Pi_{i-1})^{\alg}\cong \Pi(\kappa(\nn_i))^{\alg}$ for some $\nn_i\in \mSpec (R^{\psi}_{\rho}/\fa_1)[1/p]$.

Since $\Hom_K(V, \Pi(\kappa(\nn))^{\alg})\neq 0$ and $\Pi(\kappa(\nn))^{\alg}$ 
is irreducible $\Pi^{\alg}_i/\Pi^{\alg}_{i-1}$ is generated as a $G$-representation by its $V$-isotypical subspace.
Since the category of locally algebraic $K$-representations is semi-simple, we conclude that the $G$-subrepresentation of $\Pi^{\alg}_i$ generated by the $V$-isotypical subspace
surjects onto $\Pi^{\alg}_i/\Pi^{\alg}_{i-1}$. Let $X_0'$ be the $V$-isotypical subspace in $\Pi^{\alg}|_K$.  Then inductively we obtain that 
$\langle G\cdot X_0'\rangle$ surjects onto $\Pi^{\alg}/\Pi^{\alg}_j$ for all $0\le j\le e-1$. Since $\Pi^{\alg}_0=0$ we conclude that $X_0'$  generates 
$\Pi^{\rm alg}$ as a $G$-representation.
\end{proof}

\begin{lem}\label{last_lemma} 
Let $B$ be the closure of $\Pi^{\alg}$ in $\Pi$. If $\Pi'$ is an irreducible subquotient of $\Pi/ B$ then $(\Pi')^{\alg}=0$.
\end{lem}
\begin{proof} Let $B_i:=B\cap \Pi_i$, where $\{\Pi_i\}_{0\le i\le e}$ is the filtration  constructed in the proof of Lemma \ref{isotypic-sub-gen}. Then $B_i^{\alg}= B^{\alg} \cap \Pi_i= \Pi_i^{\alg}$. We claim that if $\Pi'$ is an irreducible subquotient of $\Pi_i/ B_i$ then $(\Pi')^{\alg}=0$. Since $\Pi_e=\Pi$ this implies the Lemma. The claim is proved by induction on $i$ using Proposition \ref{happy_chinese_new_year} and 
\eqref{locally_alg_vec_exact} for the induction step. 
\end{proof}

\begin{lem}\label{nearly_there}
Let $B$ be the closure of $\Pi^{\alg}$ in $\Pi$ and let $B^0$ be its unit ball. Then the natural morphism 
$$\Hom_{\dualcat(\OO)}(N, Q_1\otimes_{\OO} k)\ra  \Hom_{\dualcat(\OO)}(N, (B^0\otimes_{\OO} k)^{\vee}) $$ 
is an isomorphism. 
\end{lem}
\begin{proof} Let $Q_1^{\tf}$ be the maximal $\OO$-torsion free quotient of $Q_1$. If $\rho \not\cong \bigl ( \begin{smallmatrix}\delta  & \ast \\ 0 & \delta \omega \end{smallmatrix}\bigr )$ then $N$ is $R^{\psi}_{\rho}$-flat. Since $R^{\psi}_{\rho}/\mathfrak{a}_1$ is $\OO$-torsion free and 
$Q_1= N\wtimes_{R^{\psi}_{\rho}} R^{\psi}_{\rho}/\mathfrak a_1$ by definition, we deduce that $Q_1=Q_1^{\tf}$. If  $\rho \cong \bigl ( \begin{smallmatrix}\delta  & \ast \\ 0 & \delta \omega \end{smallmatrix}\bigr )$ then $N$ is $R^{\psi}_{\rho}$-flat, when considered as an object in the quotient category, obtained
by quotienting out $\dualcat(\OO)$ by the full subcategory consisting of representations on which $\SL_2(\Qp)$ acts trivially, see 
\cite[Lem. 10.42, Cor. 10.43]{pa13}. Hence, $\SL_2(\Qp)$-acts trivially on the kernel of $Q_1\twoheadrightarrow Q_1^{\tf}$ and thus trivially 
on the kernel of  $Q_1\otimes_{\OO} k \twoheadrightarrow Q_1^{\tf}\otimes_{\OO} k$. Since in this case $N$ is a projective envelope of $(\Ind_P^G\omega\delta\otimes\delta\omega^{-1})^{\vee}$,  Lemma 10.27 of \cite{pa13} implies that 
$$\Hom_{\dualcat(\OO)}(N, Q_1\otimes_{\OO} k) \overset{\cong}{\longrightarrow} \Hom_{\dualcat(\OO)}(N, Q_1^{\tf}\otimes_{\OO} k).$$
Since $\Pi^0= \Hom_{\OO}^{\cont}( Q_1, \OO)$ we have $Q_1^{\tf}\cong (\Pi^0)^d$. Thus $$Q_1^{\tf}\otimes_{\OO} k\cong (\Pi^0)^d\otimes_{\OO} k \cong (\Pi^0\otimes_{\OO} k)^{\vee}.$$

 If $\rho$ is irreducible then it follows from Proposition \ref{happy_chinese_new_year}  that all the irreducible subquotients of $\Pi$ have non-zero locally algebraic vectors, thus Lemma \ref{last_lemma}
 implies that $B=\Pi$ and we are done.
 
  Assume that  $\rho\cong \smatr{\delta_2}*0{\delta_1}$ and let $\Pi'$ be an irreducible subquotient of 
  $\Pi/B$. It follows from Lemma \ref{last_lemma} and Proposition \ref{happy_chinese_new_year} (ii) that 
 $\Pi'\cong \Pi_2$ for some $\n\in\Spec (R^{\psi}_{\rho}/\fa_1)[1/p]$ and hence $\check{\VV}(\Pi')$ is a character congruent to $\delta_2$ modulo $\varpi$. This implies that all the irreducible subquotients of $\cV( ((\Pi^0/ B^0) \otimes_{\OO} k)^{\vee})$ are isomorphic to $\delta_2$.   
 
 Let  $\phi: N\rightarrow M$ be a non-zero homomorphism in $\dualcat(\OO)$. Then some irreducible $S\in \dualcat(\OO)$ appearing in the 
 $G$-cosocle of $N$ will be contained as a subquotient in the image of $\phi$.
 It follows from  \cite[Prop.6.1,  Lem 6.13]{pa12}, \cite[Prop.2.27]{blocksp2}  that the condition (N1) in \cite[\S 4]{pa12} is satisfied, so that we have 
 $\Hom^{\cont}_{\SL_2(\Qp)}(N, \Eins)=0$. This implies that $\SL_2(\Qp)$ does not act trivially on $S$  and so  $\cV(S)\neq 0$. Since $\cV$ is an exact functor we deduce that 
 $\cV(\phi): \cV(N)\rightarrow \cV(M)$ is non-zero. Hence, $\cV$ induces an injection $ \Hom_{\dualcat(\OO)}( N, M)\hookrightarrow \Hom^{\cont}_{G_{\Qp}}(\cV(N), \cV(M)).$

Since the extension $0\rightarrow \delta_2\rightarrow \rho\rightarrow \delta_1\rightarrow 0$ is non-split, $\Hom_{G_{\Qp}}(\rho, \delta_2)$ is zero.  Since $\cV(N)$ is the tautological deformation of $\rho$ to $R^{\psi}_{\rho}$, the graded pieces of the $\mm$-adic filtration on $\cV(N)$ will be isomorphic  as  $G_{\Qp}$-representations to a direct sum of copies of $\rho$, where $\mm$ is the maximal ideal of $R^{\psi}_{\rho}$. Hence, $$\Hom_{G_{\Qp}}( \cV(N)/ \mm^n \cV(N), \delta_2)=0$$ for all $n\ge 1$ and so 
$\Hom_{G_{\Qp}}^{\cont}( \cV(N), \delta_2)=0$. This implies that $$\Hom^{\cont}_{G_{\Qp}}(\cV(N), \cV( ((\Pi^0/ B^0) \otimes_{\OO} k)^{\vee}))=0$$ and thus 
$\Hom_{\dualcat(\OO)}(N, ((\Pi^0/ B^0) \otimes_{\OO} k)^{\vee})=0$. By applying   $\Hom_{\dualcat(\OO)}( N, \ast)$ to the exact sequence
$$0\rightarrow ((\Pi^0/ B^0) \otimes_{\OO} k)^{\vee}\rightarrow (\Pi^0\otimes_{\OO} k)^{\vee}\rightarrow (B^0 \otimes_{\OO} k)^{\vee}\rightarrow 0,$$
we see that the morphism in the statement is injective; moreover 
if $\rho$ is generic, then $N$ is  projective in $\dualcat(\OO)$ and hence  the injection is an isomorphism. 

Now assume $\rho$ is non-generic. After twisting by a character  we may assume that $\delta_1=1$ and $\delta_2=\omega$. It suffices to show 
\[\Ext^1_{\CO}(N,((\Pi^0/B^0)\otimes_{\OO}k)^{\vee})=0.\]
Since all the irreducible subquotients of $(\Pi^0/B^0)\otimes_{\OO}k$ are isomorphic to $\pi_{\alpha}:=\Ind_P^G\omega\otimes\omega^{-1}$, it suffices to show
\[\Ext^1_{\CO}(N,\pi_{\alpha}^{\vee})=0.\]
Since $N$ is the universal deformation of $\beta^{\vee}$, see \S\ref{subsubsection-non-generic}, it is enough to show that
\[\Ext^1_{\CO}(\beta^{\vee},\pi_{\alpha}^{\vee})=\Ext^1_{G}(\pi_{\alpha},\beta)=0.\]
This follows from \cite[Lem.6.8]{pa12}.
\end{proof}

\begin{thm}\label{thm-BB} If $\tau$ is a principal series type then every bounded $G$-invariant $\cO$-lattice of $\Pi(\kappa(\n))^{\rm alg}$ is finitely generated as an $\cO[G]$-module, and its reduction modulo $\varpi$ is of finite length. The closure of $\Pi(\kappa(\n))^{\rm alg}$ in $\Pi(\kappa(\n))$ is isomorphic to its universal 
unitary completion. 
\end{thm}
\begin{proof}  In most cases the assertion follows from \cite[Cor.5.3.4]{bb}. For the rest see the proof of Proposition 6.13 in \cite{durham}. We also note that
it follows from Proposition \ref{happy_chinese_new_year} that the closure of $\Pi(\kappa(\n))^{\alg}$ in $\Pi(\kappa(\nn))$ is equal to  $\Pi(\kappa(\nn))$
if part (i) holds and is equal to $\Pi_1$ if part (ii) holds. Moreover, it follows from \cite[Prop.4.9]{pa12} that  $\Pi_1$ is  a unitary parabolic induction of a unitary character.
\end{proof}
Let  $\Pi^{\rm alg,0}:=\Pi^{\rm alg}\cap \Pi^0$ so that it is a bounded $G$-invariant $\cO$-lattice of $\Pi^{\rm alg}$.  Let $W:=\Sym^{b-a-1}L^2\otimes{\det}^a$. 
 Since passing to locally algebraic vectors is a left exact functor, it follows from \cite[Thm.0.20]{co} that $\Pi^{\rm alg}\cong\pi^{\rm sm}\otimes W$ for some smooth admissible $L$-representation $\pi^{\rm sm}$ of $G$.

\begin{prop} \label{prop-Pi^alg}
Assume that $\tau\cong\theta\oplus \ide$, with $\theta$ a smooth character of $\Z_p^{\times}$ of conductor $c$.  Set $c':=\max\{1,c\}$ and consider the diagram $X=(X_0,X_1,\id)$ where \[X_0:=(\pi^{\rm sm})^{K_{c'}}\otimes W, \ \ X_1:=(\pi^{\rm sm})^{I_{c'}}\otimes W.\]
Then the following hold:
\begin{itemize}
\item[(i)]  There is a $G$-equivariant isomorphism $H_0(X)\cong \Pi^{\rm alg}$.

\item[(ii)] We have a $G$-equivariant isomorphism  $\Pi^{\rm alg,0}\otimes_{\cO}k\cong H_0(\mathcal{X})\otimes_{\cO}k$, where $\mathcal{X}=(\mathcal{X}_0,\mathcal{X}_1,\id):=X\cap \Pi^0$.
\end{itemize}
\end{prop}
\begin{proof}
(i) By \cite[Thm.V.1]{ss97}, it suffices to show that $X_0$ generates $\Pi^{\rm alg}$ as a $G$-representation. 
If $\theta=\ide$ then $\sigma^{\mathrm{cr}}(\tau)$ is the trivial representation of $K$. If $\theta\neq \ide$ then 
$\sigma^{\rm cr}(\tau)=\Ind_{J_c}^K\theta\otimes \ide$, where $J_c$ denotes the subgroup $\smatr{\Z_p^{\times}}{\Z_p}{p^c\Z_p}{\Z_p^{\times}}$ and the 
character $\theta \otimes \ide$ maps $\bigl( \begin{smallmatrix} a & b\\ c & d\end{smallmatrix} \bigr)$  to $\theta(a)$.
In both cases $K_{c'}$ acts trivially $\sigma^{\mathrm{cr}}(\tau)$ and hence the $V$-isotypic subspace of $\Pi^{\alg}$ is contained in $X_0$ and the assertion follows from Lemma \ref{isotypic-sub-gen}.

(ii) Let $\{\Pi_i\}_{0\le i\le e}$ is the filtration  constructed in the proof of Lemma \ref{isotypic-sub-gen}. Let $\Pi_i^0:= \Pi_i\cap \Pi^0$ and  $\Pi_i^{\alg, 0}:= \Pi_i^{\alg}\cap \Pi^0$. Since the Banach space representations $\Pi_i$ are admissible $\Pi_i^0/\Pi_{i-1}^0$ 
is an open bounded $\OO$-lattice in $\Pi_i/\Pi_{i-1}$, which we may assume to be isomorphic to $\Pi(\kappa(\nn))$. Hence $\Pi_i^{\alg, 0}/\Pi_{i-1}^{\alg, 0}$ is a 
bounded $G$-invariant $\OO$-lattice in $\Pi(\kappa(\nn))^{\alg}$. It follows from Theorem \ref{thm-BB} that $\Pi_i^{\alg, 0}/\Pi_{i-1}^{\alg, 0}$ is a finitely generated $\OO[G]$-module
and $(\Pi_i^{\alg, 0}/\Pi_{i-1}^{\alg, 0})\otimes_{\OO} k$ is a $G$-representation of finite length. Inductively we obtain that 
 $\Pi^{\rm alg,0}$ is finitely generated as an $\cO[G]$-module  and $\Pi^{\rm alg,0}\otimes_{\cO}k$ is a $G$-representation of finite length. The result then follows from (i) and Theorem \ref{thm-diagram}.
\end{proof}

Now we are ready to prove Theorem \ref{theorem-main-intro} stated in the introduction.

\begin{thm}\label{theorem-main}
Assume that $\End_{G_{\Qp}}(\rho)=k$ and if $p=3$ assume further that $\rho\not\sim \bigl ( \begin{smallmatrix}\chi \omega & \ast \\ 0 & \chi \end{smallmatrix}\bigr )$ for any character $\chi: G_{\Qp}\rightarrow k^{\times}$.
 Then, for any $p$-adic Hodge type $(\w,\theta_1 \oplus \theta_2,\psi)$ with $\theta_1\neq \theta_2$, where $\theta_1, \theta_2: I_{\Qp}\rightarrow L^{\times}$
are   characters with open kernel, which extend to $W_{\Qp}$,  the potentially crystalline deformation ring $R^{\psi,\rm cr}_{\rho}(\w, \theta_1\oplus \theta_2)$ is Cohen-Macaulay, whenever it is non-zero.
\end{thm}
\begin{proof}
By assumption, the type $\tau$ is of the form $\theta_1\oplus\theta_2$, where $\theta_i$ is a smooth character of $I_{\Q_p}$ which extends to $W_{\Q_p}$. Since everything is compatible with twisting by characters, we may assume $\tau$ is of the form $\theta\oplus\ide$, so that Proposition \ref{prop-Pi^alg} applies. We keep the notation introduced in the proposition and its proof.

Let $\Upsilon$ be $\mathcal{X}_0^d$, which is a quotient of $(\Pi^0)^d\cong Q_1$. To prove the theorem, it suffices to  verify  the conditions in Lemma \ref{lem-criterion} for $\Upsilon$. First, we have $\Pi^{\rm alg,0}\otimes_{\cO}k\cong H_0(\mathcal{X}\otimes_{\cO}k)$ by Proposition \ref{prop-Pi^alg}(ii) and hence a surjection 
$\cInd_{KZ}^G \mathcal{X}_0\otimes_{\cO}k \twoheadrightarrow \Pi^{\rm alg,0}\otimes_{\cO}  k$. Since the centre $Z$ acts everywhere by the same central character 
the restriction map
\[\Hom_{G}(\Pi^{\rm alg,0}\otimes_{\cO}k,N^{\vee})\ra \Hom_K(\mathcal{X}_0\otimes_{\cO}k,N^{\vee})\]
is injective.  Dually we obtain an injection 
\begin{equation}\label{inject_again}
\Hom_{\dualcat(\OO)}(N, (\Pi^{\rm alg,0}\otimes_{\cO}k)^{\vee})\hookrightarrow \Hom_{\OK}^{\cont}( N, \Upsilon \otimes_{\OO} k).
\end{equation}
Let $B$ the closure of $\Pi^{\alg}$ in $\Pi$. Then $B^0\otimes_{\OO} k\cong \Pi^{\alg, 0}\otimes_{\OO} k$. It follows from Lemma \ref{nearly_there} and \eqref{inject_again} that the 
surjection $Q_1\twoheadrightarrow \Upsilon$ induces an injection
$$\Hom_{\dualcat(\OO)}( N, Q_1\otimes_{\OO} k)\hookrightarrow \Hom^{\cont}_{\OK}( N, \Upsilon\otimes_{\OO} k),$$
so that part (ii) of Lemma \ref{lem-criterion} holds. 

Since $\theta\neq \ide$ by assumption it follows from Lemma \ref{K_c-invariants} that  $X_0$ is isomorphic to a finite direct sum of copies of $V$ and part (i) of Lemma 
\ref{lem-criterion} follows from  Lemma \ref{lem-surjective}.
 \end{proof}

 \begin{rem}
We have just proved that $(x,\varpi)$ is a regular sequence in $R$, hence $(\varpi,x)$ is also a regular sequence in $R$. However, we don't know a direct proof of this fact. 
 \end{rem}
 \section{Generic split case} 
 \label{section-split}
 Let $\delta_1, \delta_2: G_{\Qp}\rightarrow k^{\times}$ be distinct characters. It was shown in \cite{ht, pa15} that the Breuil--M\'ezard conjecture for 
 non-split representations $\rho_1=\bigl ( \begin{smallmatrix} \delta_1& \ast \\ 0 & \delta_2\end{smallmatrix} \bigr)$ and 
 $\rho_2=\bigl ( \begin{smallmatrix}\delta_1 & 0\\ \ast & \delta_2\end{smallmatrix} \bigr)$ implies the Breuil--M\'ezard conjecture for the split representation 
 $\rho=\bigl ( \begin{smallmatrix} \delta_1 & 0 \\ 0 & \delta_2\end{smallmatrix} \bigr)$. In view of the arguments of those papers, it is tempting to speculate 
 that if both rings $R^{\psi}_{\rho_1}(\wt, \tau)$ and $R^{\psi}_{\rho_2}(\wt, \tau)$ are Cohen--Macaulay then so is the framed potentially
 semi-stable 	ring   $R^{\square, \psi}_{\rho}(\wt, \tau)$. We tried to prove this using ideas connected to \cite[Ex.18.13]{eis}, but 
 ran into  difficulties controlling the intersection of reducible and irreducible loci in $\Spec  R^{\psi}_{\rho_1}(\wt, \tau)$.
 However, there is one easy case, where our argument works: if $\delta_1\delta_2^{-1}\neq \omega^{\pm 1}$ and the reducible locus is empty. If $p$ is large then 
 the reducible locus is empty in most of cases, see Remark \ref{most_cases}.
 
 \begin{lem}\label{reducible_locus} Assume $R^{\square, \psi}_{\rho}(\wt, \tau)\neq 0$. There is $x\in \mSpec R^{\square, \psi}_{\rho}(\wt, \tau)[1/p]$ such that the corresponding representation $\rho_x$ is reducible if and only 
 if $\tau=\theta_1\oplus \theta_2$ is a principal series type and at least one of the characters 
 $\theta_1 \epsilon^b$, $\theta_1 \epsilon^a$ is congruent to $\delta_1|_{I_{\Qp}}$ or $\delta_2|_{I_{\Qp}}$ modulo $\varpi$, where $\wt=(a, b)$. 
\end{lem}
\begin{proof} If $\rho_x$ is reducible then there is an exact sequence $0\rightarrow \chi_1\rightarrow \rho_x \rightarrow \chi_2\rightarrow 0$
of Galois representations such that $\chi_1\chi_2= \psi\epsilon$ and either $\chi_1\equiv \delta_1$ and $\chi_2\equiv \delta_2$ modulo $\varpi$ or $\chi_1\equiv \delta_2$ and $\chi_2\equiv \delta_1$ modulo $\varpi$. Since $\rho_x$ is potentially semi-stable both characters $\chi_1$ and $\chi_2$ are potentially semi-stable and Hodge--Tate weight 
of $\chi_1$ is $b$ and Hodge--Tate weight of $\chi_2$ is $a$. This is forced upon us if $\rho_x$ is non-split and we may assume that if $\rho_x$ is split.
Hence, $\chi_1\epsilon^{-b}$ and $\chi_2\epsilon^{-a}$ have open kernel 
 and if we let $\theta_1= \chi_1\epsilon^{-b}|_{I_{\Qp} }$ and $\theta_2 =\chi_2\epsilon^{-a}|_{I_{\Qp}}$ then the  Galois type of $\rho_x$ is isomorphic to
 $\theta_1\oplus \theta_2$. Hence, either $\theta_1 \epsilon^b \equiv \delta_1|_{I_{\Qp} }$ or $\theta_1 \epsilon^b \equiv \delta_2 |_{I_{\Qp} }\pmod{\varpi}$.
 We note that $\det \rho_x= \chi_1\chi_2$, hence $\theta_1 \theta_2 \epsilon^{a+b}\equiv\delta_1\delta_2|_{I_{\Qp}} \pmod{\varpi}$.
 So that  $\theta_2 \epsilon ^b \equiv \delta_1|_{I_{\Qp}} \pmod{\varpi}$ if and only if $\theta_1 \epsilon^a \equiv \delta_2|_{I_{\Qp}}\pmod{\varpi}$
 and $\theta_2 \epsilon ^b \equiv \delta_2|_{I_{\Qp}} \pmod{\varpi}$ if and only if $\theta_1 \epsilon^a \equiv \delta_1|_{I_{\Qp}}\pmod{\varpi}$. 
For the converse direction using class field theory extend characters $\theta_1\epsilon^b$ and $\theta_2 \epsilon^a$ to characters $\chi_1$, $\chi_2$ of $G_{\Qp}$
such that $\chi_1\chi_2=\psi\varepsilon$. Then the split representation $\chi_1\oplus \chi_2$ will give the required maximal ideal.  
\end{proof}   
 
\begin{defn} The reducible locus in $\Spec  R^{\square, \psi}_{\rho}(\wt, \tau)[1/p]$ is empty if $\rho_x$ is an irreducible representation 
of $G_{\Qp}$ for all $x\in \mSpec R^{\square, \psi}_{\rho}(\wt, \tau)[1/p]$. 
\end{defn}
 
 \begin{lem}\label{split3} Let $(A,\mm)$ be a local noetherian ring,  let $B=A[[x,y]]/(xy-c)$ with $c\in \mm$.  
If $A$ is Cohen--Macaulay then $B$ is also Cohen--Macaulay. 
\end{lem}
\begin{proof} By letting $z=x-y$ we obtain $B=A[[y,z]]/(y^2+zy -c)$. Thus $B$ is a free $A[[z]]$-module of rank $2$ with basis $1, y$. 
If $a_1, \ldots, a_d$ is a regular system of parameters for $A$ then $a_1, \ldots, a_d, z$ is a regular system 
of parameters for $B$.
\end{proof}

\begin{prop}\label{gen_split_cm} Assume that $\delta_1\delta_2^{-1}\neq \omega^{\pm 1}$,  $R^{\square, \psi}_{\rho}(\wt, \tau)\neq 0$ and the 
reducible locus in $\Spec R^{\square, \psi}_{\rho}(\wt, \tau)[1/p]$ is empty. Then $R^{\psi}_{\rho_1}(\wt, \tau)\cong R^{\psi}_{\rho_2}(\wt, \tau)$ and 
$$ R^{\square, \psi}_{\rho}(\wt, \tau)\cong R^{\psi}_{\rho_1}(\wt, \tau)[[x, y, z, w]]/ (xy-c)$$
for an element $c$ in the maximal ideal of $R^{\psi}_{\rho_1}(\wt, \tau)$. In particular, if $R^{\psi}_{\rho_1}(\wt, \tau)$ is Cohen--Macaulay 
then $R^{\square, \psi}_{\rho}(\wt, \tau)$ is Cohen--Macaulay. 
\end{prop}
\begin{proof} The first assertion follows from \cite[Lem.6.4, 6.5]{pa15} by noting that the assumption that the reducible locus is empty implies that 
the ideals denoted by $I_1^{\ps}$, $I_2^{\ps}$, $I_1^{\ver}$, $I_2^{\ver}$ in \cite{pa15} are unit ideals and \cite[Rem.7.4]{pa15}, which says that the framed deformation ring 
is formally smooth of relative dimension $2$ over the versal deformation ring. The last assertion follows from Lemma \ref{split3}.
\end{proof} 

\begin{thm}\label{gen_split} Let $\tau=\theta_1\oplus \theta_2$ be a principal series type with $\theta_1\neq \theta_2$ and let $\wt=(a, b)$ with $a<b$. Assume that 
$\delta_1\delta_2^{-1}\neq \omega^{\pm 1}$ and $\theta_1$ is not congruent modulo $\varpi$ to any of the four characters 
$\delta_1\omega^{-a}$, $\delta_1\omega^{-b}$, $\delta_2\omega^{-a}$, $\delta_2\omega^{-b}$. If $R^{\square, \psi}_{\rho}(\wt, \tau)\neq 0$ then 
it is Cohen--Macaulay.
\end{thm}
\begin{proof} Lemma \ref{reducible_locus} and the assumption on $\theta_1$ implies that the reducible locus is empty. 
 If $R^{\square, \psi}_{\rho}(\wt, \tau)\neq 0$ then $R^{\psi}_{\rho}(\wt, \tau)\neq 0$ by
Proposition \ref{gen_split_cm}. Since Theorem \ref{theorem-main} implies that  $R^{\psi}_{\rho}(\wt, \tau)$ is Cohen--Macaulay we deduce from 
Proposition \ref{gen_split_cm} that $R^{\square, \psi}_{\rho}(\wt, \tau)$ is Cohen--Macaulay. 
 \end{proof} 
 
 \begin{rem}\label{most_cases} If we fix Hodge--Tate weights $\wt=(a, b)$ and a character $\psi: G_{\Qp}\rightarrow \OO^{\times}$ such that 
 $\psi\epsilon \equiv \delta_1\delta_2\pmod{\varpi}$ and $\psi\epsilon^{-b-a-1}$ has open kernel then it follows from the Breuil--M\'ezard 
 conjecture proved in this case in  \cite{ki09}, \cite{blocksp2} that $R^{\square,\psi}_{\rho}(\wt, \tau)\neq 0$ for any principal series type 
 $\tau=\theta_1 \oplus \theta_2$, such that $\theta_1\theta_2 \epsilon^{a+b}= \psi\epsilon|_{I_{\Qp}}$ and the conductor of $\theta_1\theta_2^{-1}$ is at least $2$. 
 Let $c$ be the conductor of $\psi\epsilon^{-b-a-1}$. Hence, if the conductor of $\theta$ is at least $\max(2, c+1)$ and we let $\tau= \theta \oplus \theta^{-1} \psi\epsilon^{-b-a-1}$ then 
 $R^{\square, \psi}_{\rho}(\wt, \tau)\neq 0$ and conversely all principal series types $\tau$ with the conductor $c(\tau)\ge  \max(c+1, 2)$ 
 such that  $R^{\square, \psi}_{\rho}(\wt, \tau)\neq 0$ are of the form $\theta \oplus \theta^{-1} \psi\epsilon^{-b-a-1}$ as above. 
So we are free to choose $\theta$ as we like  provided its conductor is large and Theorem \ref{gen_split} applies if we avoid $4$ out of 
possible $p-1$ congruence classes for $\theta$ modulo $\varpi$. 
 \end{rem}

 \section{Further discussions}
  \label{section-further}

We go back to the setting in Section \ref{section-recall} and discuss some other properties of potentially semi-stable deformation rings  related to the Cohen-Macaulayness. The notation is as in Section \ref{section-recall}. In particular, $\rho:G_{\Q_p}\ra \GL_2(k)$ is such that $\End_{G_{\Q_p}}(\rho)=k$ and if $p=3$, then we assume that $\rho\not\sim \bigl ( \begin{smallmatrix}\delta \omega & \ast \\ 0 & \delta \end{smallmatrix}\bigr )$, for any character $\delta: G_{\Qp}\rightarrow k^{\times}$. Let $(\w,\tau,\psi)$ be a fixed $p$-adic Hodge type,  $V=\sigma(\w,\tau)$ (resp. $\sigma^{\rm cr}(\w,\tau)$), and  $R=R^{\psi}_{\rho}(\w,\tau)$ (resp. $ R^{\psi,\rm cr}_{\rho}(\w,\tau)$) the corresponding potentially semi-stable (resp. crystalline) deformation ring.

\subsection{Cyclicity of the module $M(\Theta)$}\label{subsection-cyclicity}
Inspired by the  results of \cite{egs} we investigate the following question. 

\begin{ques}
Does there exist a $K$-invariant $\cO$-lattice $\Theta$ inside $V$  such that $M(\Theta)$ becomes a cyclic module over $R$?
\end{ques} 

 A positive answer to this question would imply  our main Theorem \ref{theorem-main}: since $R$ acts faithfully on $M(\Theta)$, we would deduce that $M(\Theta)$ is a free $R$-module of rank $1$, hence Theorem \ref{theorem-main-intro} would follow  from \cite[Lem.2.33]{pa12}  which says that $M(\Theta)$ is a Cohen-Macaulay module. We show in Theorem \ref{prop-cyclic} below that this is the case if 
 we consider crystalline deformations with small Hodge--Tate weights, so that $V=\Sym^{b-a-1} L^2\otimes \det^a$ with $1\le b-a\le 2p$. In the next subsection we will show that these rings are complete intersection when $\rho$ is generic. It is not clear to us whether $\Theta$ can be chosen so that the module $M(\Theta)$ is cyclic in general. 

By Nakayama's lemma, $M(\Theta)$ is a cyclic $R$-module if and only if $M(\Theta)/\m_RM(\Theta)$ is a cyclic $k$-module,   which holds if and only if 
\begin{equation}\label{equation-cyclic=dim1}\tag{\textbf{C}}\dim_k\Hom_{K}(\Theta/\varpi\Theta, (k\widehat{\otimes}_{R^{\psi}_{\rho}}N)^{\vee})\leq 1.\end{equation}
The above dimension is equal to $0$ if and only if  $M(\Theta)=R=0$. To ease the notation we let 
$$ \kappa(\rho):= (k\widehat{\otimes}_{R^{\psi}_{\rho}}N)^{\vee}.$$
If $\rho$ is either irreducible or $\rho\sim \bigl ( \begin{smallmatrix}\delta_2 & \ast \\ 0 & \delta_1 \end{smallmatrix}\bigr )$ with $\delta_1^{-1}\delta_2\neq \omega^{\pm 1}$ 
then $\kappa(\rho)$ is what Colmez calls \textit{atome automorphe} attached to $\rho$ in \cite[VII.4]{co}, if 
$\rho\sim \bigl ( \begin{smallmatrix}\delta \omega & \ast \\ 0 & \delta \end{smallmatrix}\bigr )$ then $\kappa(\rho)$ is the representation denoted by 
$\beta$ in \cite[\S6.2]{pa12}, if $\rho\sim \bigl ( \begin{smallmatrix}\delta  & \ast \\ 0 & \delta \omega \end{smallmatrix}\bigr )$ then the structure of 
$\kappa(\rho)$ is discussed in \cite[Prop.6.21, Rem.6.22, 7.6]{6auth2}.

\begin{defn}
A Serre weight of $\rho$ is a smooth irreducible representation $\sigma$ of $K$ such that 
\[\Hom_{K}(\sigma,\kappa(\rho))\neq0.\] 
Let $\mathcal{D}(\rho)$ denote the set of Serre weights of $\rho$.
\end{defn}
Up to normalization, the above definition of Serre weights of $\rho$ coincides with the usual one given in \cite{bdj}. We refer to \cite[Rem.6.2, 6.3]{pa12}, \cite[Rem.2.29]{blocksp2} and the proof of \cite[Prop.6.20]{pa12}
for the description of $\mathcal{D}(\rho)$.  We write $\overline{V}$ for the semi-simplification of $\Theta/\varpi\Theta$ for any $K$-invariant $\OO$-lattice $\Theta\subset V$ and $\JH(\overline{V})$ the set of Jordan-H\"older factors of $\overline{V}$. 
The above discussion gives the following simple criterion for $M(\Theta)$ to be cyclic.  

\begin{lem}\label{lemma-cyclic}
If, taking into account of multiplicities, $\JH(\overline{V})$ contains at most one Serre weight of $\rho$, then $M(\Theta)$ is a cyclic $R$-module.
\end{lem}
\begin{proof} 
By the above discussion, it suffices to note that in all cases the $K$-socle of $\kappa(\rho)$ is multiplicity free. Indeed, if $\rho$ is generic, this follows from \cite[Rem.6.2, 6.3]{pa12}, \cite[Rem.2.29]{blocksp2}; if $\rho$ is non-generic, this follows from the proof of \cite[Prop.6.20]{pa12}, where the representation $\kappa(\rho)$ is denoted by $\beta$.
\end{proof}

We will use Lemma \ref{lemma-cyclic} to show that crystalline deformation rings in small HT-weights are Cohen-Macaulay, which can be regarded as a complement of our main theorem. We need the following results of Breuil \cite{Br03} and Morra \cite{mor, mor2}. We assume $p\geq 3$ because this is imposed in \cite{mor, mor2}.
\begin{thm}  \label{theorem-morra} 
Assume $p\geq 3$. Let $\pi$ be an absolutely irreducible smooth admissible $k$-representation of $G$.
\begin{enumerate}
\item[(i)] Assume $\pi$ is a supersingular representation. Then 

\begin{enumerate}
\item[(1)] viewed as an $I$-representation, $\pi^{I_1}$ is isomorphic to $\theta\oplus \theta^s$ for some smooth character $\theta:I\ra k^{\times}$ where $\theta^s:I\ra k^{\times}$ is the character obtained by conjugating $\theta$ by $\smatr01p0$; 

\item[(2)]   $\rsoc_K(\pi)$ is always of length 2 and is isomorphic to $(\Ind_I^K\theta)^{\rm ss}$; moreover,
\[(\pi/\rsoc_K(\pi))^{K_1}\cong \Ind_{I}^K\theta\alpha\oplus \Ind_I^K\theta^s\alpha.\]
where  $\alpha:I\ra k^{\times}$ is the character sending $\smatr{a}{b}{pc}{d}$ to $\bar{a}\bar{d}^{-1}$;

\item[(3)] for any smooth character $\theta':I\ra k^{\times}$ such that $\theta'\neq \theta'^s$, there exists no $K$-equivariant embedding $\Ind_I^K\theta'\hookrightarrow \pi$; moreover 
\[\dim_k\Hom_K(\Ind_I^K\theta',\pi)\leq 1.\]
\end{enumerate} 

\item[(ii)] Assume  $\pi=\Ind_P^G\chi$ is a  principal series, then $\pi^{K_1}\cong\Ind_{I}^{K}\chi$ and
 \[(\pi/\pi^{K_1})^{K_1}\cong \Ind_{I}^{K} {\chi}\alpha.\]

\item[(iii)] Assume $\pi=\Sp\otimes \chi\circ\det$ is a special series representation, where $\Sp$ denotes the Steinberg representation, then $\rsoc_K(\pi)=\pi^{K_1}\cong \Sym^{p-1}k^2\otimes\chi\circ\det$ and 
\[(\pi/\rsoc_K(\pi))^{K_1}\cong \Ind_I^K\chi\alpha.\] 
\end{enumerate}
\end{thm}
\begin{proof}
(i) (1) is just \cite[Cor.4.1.4]{Br03} and (2) is a consequence of \cite[Thm.1.1]{mor}. For (3), the first statement follows from \cite[Thm.1.4]{mor2} which determines $\pi^{K_1}$; note that the assumption $\theta'\neq \theta'^s$ ensures that $\Ind_I^K\theta'$ is indecomposable (see \cite[\S2]{bp}). The second statement follows from (1) by Frobenius reciprocity. 

(ii) This is a consequence of \cite[Thm.1.2]{mor}. More precisely, in \emph{loc. cit.} the socle filtration of $\pi|_K$ is determined, from which we can determine  $\pi^{K_1}$ and $(\pi/\pi^{K_1})^{K_1}$.

(iii) As in (ii), this  is a consequence of \cite[Thm.1.2]{mor}. 
\end{proof}

\begin{thm}\label{prop-cyclic}
Assume $p\geq 3$. Let $\w=(a,b)$ with $1\leq b-a\leq 2p$ and $\tau=\ide\oplus\ide$. Assume $R=R^{\psi,\rm cr}_{\rho}(\w,\tau)$ is non-zero. Then there exists a $K$-invariant $\cO$-lattice $\Theta$ in $V=\sigma^{\rm cr}(\w,\tau)$ such that $M(\Theta)$ is a cyclic $R$-module. In particular, $R$ is Cohen-Macaulay.
\end{thm}
\begin{proof}
After twisting by a character  we may assume $a=0$. Then the representation $V$ is just $\Sym^{b-1}L^2$ with $1\le b\le 2p$. It admits a natural $K$-invariant $\OO$-lattice, i.e. $\Sym^{b-1}\OO^2$, whose mod $\varpi$ reduction is $\Sym^{b-1}k^2$. We first  recall from \cite{gl} and \cite[\S5.1]{Br03II} some results about  the structure of $\Sym^{b-1}k^2$. 
\begin{enumerate}
\item[(o)] By \cite[(3.3), p.434]{gl}, $\Sym^{b-1}k^2$ is injective  as a $k$-representation of $\GL_2(\F_p)$ if $p\ |\ b$, and is a direct sum of an injective representation and a non-injective indecomposable representation otherwise. 
\item[(a)] If $1\leq b\leq p$, $\Sym^{b-1}k^2$ is irreducible.

\item[(b)] If $b=p+1$, then $\Sym^{p}k^2$ is of length $2$, uniserial and fits into an exact sequence
\[0\ra \Sym^{1}k^2\ra \Sym^{p}k^2\ra\Sym^{p-2}k^2\otimes{\det}\ra0;\]
moreover $\Sym^pk^2\cong \Ind_{I}^K(\ide\otimes\omega)$ as $K$-representations. In fact, the exact sequence is established in \cite[\S5.1]{Br03II}, which must be non-split by (o); since $\Ext^1_{\GL_2(\F_p)}(\Sym^{p-2}k^2\otimes\det,\Sym^{1}k^2)$ is $1$-dimensional and $\Ind_I^K(\ide\otimes \omega)\cong \Ind_{P(\F_p)}^{\GL_2(\F_p)}(\ide\otimes\omega)$ indeed provides such a non-split extension, we get the desired isomorphism. 

\item[(c)] If $p+2\leq b\leq 2p-2$, then $\Sym^{b-1}k^2$ is of length $3$ and fits into an exact sequence
\begin{multline}\label{equation-breuil-Sym}0\ra (\Sym^{b-2-p}k^2\otimes{\det})\oplus \Sym^{b-p}k^2\ra \Sym^{b-1}k^2 \\ \ra\Sym^{2p-1-b}k^2\otimes{\det}^{b-p}\ra0;\end{multline}
which is indecomposable by (o). Moreover, one shows that the cokernel of $(\Sym^{b-2-p}k^2\otimes\det\hookrightarrow \Sym^{b-1}k^2)$ is isomorphic to $\Ind_I^K(\ide\otimes \omega^{b-p})$ by the same argument as in part (b) .

\item[(d)] If $b=2p-1$,  the exact sequence \eqref{equation-breuil-Sym} in (c) still holds, except that the factor $\Sym^{b-p}k^2=\Sym^{p-1}k^2$ is injective, so that $\Sym^{b-1}k^2$ is the direct sum of  $\Sym^{p-1}k^2$ and a non-injective representation which is a non-split extension of $\Sym^{0}k^2$ by $\Sym^{p-3}k^2\otimes\det$.

\item[(e)] If $b=2p$, then $\Sym^{2p-1}k^2$ is injective as a $\GL_2(\F_p)$-representation by (o), and is isomorphic to the injective envelope of $\Sym^{p-2}k^2\otimes {\det}$ by writing down its Jordan-H\"older factors. Moreover, from \cite[\S2, \S3]{bp} and using (b) we have two exact sequences
\begin{equation}\label{equation-case-d1}0\ra \Sym^{p-2}k^2\otimes\det\ra \Sym^{2p-1}k^2\ra \Sym^pk^2\ra0\end{equation}
\begin{equation}\label{equation-case-d2}0\ra \Ind_I^K(\omega\otimes \ide)\ra \Sym^{2p-1}k^2\ra \Sym^{p-2}k^2\otimes\det\ra0.\end{equation}
\end{enumerate}

We first treat the case when $\rho$ is irreducible. In this case $\kappa(\rho)$ is just the representation $\pi$ defined in \S\ref{subsubsection-generic}, which is a  supersingular representation.  Moreover, $\rho$ always has two Serre weights, which arise as the mod $\varpi$ reduction of a characteristic zero principal series of $\GL_2(\F_p)$.  We treat separately the above five cases (a)-(e).
\begin{enumerate}
\item[(a)] In this case we take $\Theta$ to be any $K$-invariant $\OO$-lattice in $V$,  and the result follows from Lemma \ref{lemma-cyclic}.

\item[(b)] In this case, $R$ is non-zero if and only if $\mathcal{D}(\rho)=\{\Sym^{1}k^2,\Sym^{p-2}k^2\otimes\det\}$. We take $\Theta=\Sym^{b-1}\OO^2$ and need verify (\ref{equation-cyclic=dim1}). Since $\Sym^{p}k^2\cong \Ind_I^K(\ide\otimes\omega)$ and $\ide\otimes\omega\neq (\ide\otimes\omega)^s$, (\ref{equation-cyclic=dim1}) holds by  Theorem \ref{theorem-morra}(i)(3).

\item[(c)] If $\Sym^{b-2-p}k^2\otimes\det\in \mathcal{D}(\rho)$, we take $\Theta$ to be any $K$-invariant $\OO$-lattice in $V$  and the result follows from Lemma \ref{lemma-cyclic}. If $$\mathcal{D}(\rho)=\{\Sym^{b-p}k^2,\Sym^{2p-1-b}k^2\otimes{\det}^{b-p}\},$$ then we take $\Theta=\Sym^{b-1}\OO^2$ and apply the same argument as in (b).

\item[(d)] If $b=2p-1$, the situation is different from (c) since $\omega^{b-p}=\ide$; we take $\Theta$ to be the lattice constructed in Lemma \ref{lemma-lattices}(i) below.   To check the condition (\ref{equation-cyclic=dim1})  let $f:\Theta/\varpi\Theta\ra \pi$ be a non-zero $K$-equivariant morphism. Since $\rsoc_K(\pi)=\Sym^0k^2\oplus\Sym^{p-1}k^2$, if $f$ is not injective, then it must factor as $\Theta/\varpi\Theta\twoheadrightarrow \Sym^0k^2\hookrightarrow \pi$; moreover this indeed gives a non-zero element in $\Hom_K(\Theta/\varpi\Theta,\pi)$.  Therefore, to verify (\ref{equation-cyclic=dim1}), it suffices to show any $f:\Theta/\varpi\Theta\ra \pi$ can not be injective. Indeed,   since $\rsoc_K(\pi)=\Sym^0k^2\oplus\Sym^{p-1}k^2$ and $\rsoc_K(\Theta/\varpi\Theta)=\Sym^{p-1}k^2$ by construction, if $f$ were injective, we would get an embedding
\[A_1\hookrightarrow \pi/\rsoc_K(\pi)\]
where $A_1$ is defined by (\ref{equation-A1}) below. But this would contradict Theorem \ref{theorem-morra}(i)(2).  

\item[(e)] In this case $R$ is non-zero if and only if $\mathcal{D}(\rho)=\{\Sym^{1}k^2,\Sym^{p-2}k^2\otimes\det\}$. We take $\Theta=\Sym^{2p-1}\OO^2$.  To verify (\ref{equation-cyclic=dim1}), let $f:\Sym^{2p-1}k^2\ra \pi$ be a non-zero $K$-equivariant morphism.  First note that $f$ can not be injective. Otherwise the restriction of $f$ to $\Ind_I^K(\omega\otimes \ide)$ would give an embedding $\Ind_I^K(\omega\otimes \ide)\hookrightarrow\pi$, which would contradict Theorem \ref{theorem-morra}(i)(3) since $\omega\otimes \ide \neq(\omega\otimes \ide)^s$ (recall $p\geq 3$). Second, using (\ref{equation-case-d1}) and the case (b), the image of $f$ can not be $\Sym^pk^2$, hence must be $\Sym^{p-2}k^2\otimes\det$. In other words, any $f:\Theta/\varpi\Theta\ra \pi$ vanishes on $\Ind_I^K(\omega\otimes \ide)$ via (\ref{equation-case-d2}). Since $\dim_k\Hom_K(\Sym^{p-2}k^2\otimes\det,\pi)=1$,  (\ref{equation-cyclic=dim1}) is verified. 
\end{enumerate}

Next assume $\rho$ is reducible. The set $\mathcal{D}(\rho)$ of Serre weights of $\rho$ consists of $\{\Sym^0k^2,\Sym^{p-1}k^2\}$ if $\rho$ is peu ramifi\'e or $\rho\sim \smatr{\omega\chi_2}{*}0{\chi_1}$ with $\chi_1,\chi_2$ distinct  unramified characters, and  consists of one single element otherwise. This already treats the cases (a), (b), (c) and the tr{\`es} ramifi\'e case in (d)  using Lemma \ref{lemma-cyclic}, because it is easy to check that  $|\JH(\overline{V})\cap \mathcal{D}(\rho)|=1$ (recall $p\geq 5$).
 
 In case (d), we may assume $\rho$ is peu ramifi\'e or $\rho\sim\smatr{\omega\chi_2}{*}0{\chi_1}$ with $\chi_1,\chi_2$   distinct unramified characters.  We will verify (\ref{equation-cyclic=dim1}) with 
$\Theta$  the lattice constructed in Lemma \ref{lemma-lattices}(i) below.
Recall that
\begin{enumerate}
\item[-] if $\rho$ is peu ramifi\'e,  the $G$-representation $\kappa(\rho)$ is uniserial of length $3$ 
with Jordan-H\"older factors being $\Sp, \ide,\pi_{\alpha}$, where $\pi_{\alpha}:=\Ind_P^G\omega\otimes\omega^{-1}$, see \cite[\S6.2]{pa12} where $\kappa(\rho)$ is denoted by $\beta$.

\item[-] if $\rho\sim\smatr{\omega\chi_2}{*}0{\chi_1}$ with $\chi_1,\chi_2$ distinct unramified characters, $\kappa(\rho)$ is of length $2$ and of the form $0\ra \pi\ra \kappa(\rho)\ra \pi'\ra0$, where $\pi,\pi'$ are both principal series with $\pi$ an unramified one, see \cite[\S8]{pa13} and \cite[\S6.1]{pa12}.
\end{enumerate} 
We observe that, although $\kappa(\rho)$ are not isomorphic as $G$-representations in the above two cases, their restrictions to $K$ are closely related: we have $\pi'|_K\cong \pi_{\alpha}|_K$ and $\pi|_K\cong \pi_{\Eins}|_K$, where $\pi_{\Eins}:=\Ind_{P}^G\Eins\otimes\Eins$  which is isomorphic to a non-split extension of  $\Sp$ by $\ide$.
By Lemma \ref{lemma-lattices}(i), we know
\[\Hom_K(\Theta/\varpi\Theta,\Sp)=\Hom_{K}(\Theta/\varpi\Theta,\pi_{\alpha})=0,\]
hence in both cases $\Hom_K(\Theta/\varpi\Theta,\kappa(\rho))\hookrightarrow  \Hom_K(\Theta/\varpi\Theta,\ide)$. 
Since the latter space is one-dimensional, \eqref{equation-cyclic=dim1} is verified. 

Now we treat the last case (e) (for $\rho$ reducible). If $\mathcal{D}(\rho)=\{\Sym^1k^2\}$, we take $\Theta$ to be any $K$-invariant $\cO$-lattice and apply Lemma \ref{lemma-cyclic}. Assume $\mathcal{D}(\rho)=\{\Sym^{p-2}k^2\otimes\det\}$ in the rest. Since $\rho$ is generic, $\kappa(\rho)$ fits in a non-split extension 
\[0\ra \pi\ra \kappa(\rho)\ra\pi'\ra0\]
where $\pi,\pi'$ are both principle series representations with $K$-socle isomorphic to $\Sym^{p-2}k^2\otimes\det$. Moreover, $\pi^{K_1}\cong \pi'^{K_1}\cong 
\Ind_I^K \omega\otimes\Eins$ as $K$-representations, and we have an exact sequence of $K$-representations:
\[0\ra \Sym^{p-2}k^2\otimes\det\ra \Ind_I^K \omega\otimes\Eins\ra \Sym^{1}k^2\ra0.\]
Since $p\geq 3$, Theorem \ref{theorem-morra}(ii) implies that 
\[\Hom_K(\Sym^{p-2}k^2\otimes\det,\pi/\pi^{K_1})=0.\]
Therefore the natural morphism
\[\gamma: \Ext^1_{K,\zeta}(\Sym^{p-2}k^2\otimes\det,\pi^{K_1})\ra \Ext^1_{K,\zeta}(\Sym^{p-2}k^2\otimes\det,\pi)\]
is injective. 
We  claim that $\gamma$ is an isomorphism.  By \cite[Thm.7.16(ii)]{bp}, the second space has dimension 2 over $k$, so it suffices to show 
\[\dim_k\Ext^1_{K,\zeta}(\Sym^{p-2}k^2\otimes\det,\pi^{K_1})=2.\]
For this we construct explicitly two linearly independent extension classes: $c_1$ comes from $\Sym^{2p-1}k^2$ (see (\ref{equation-case-d2})) and $c_2$ comes from the push-out of a non-zero extension  in $\Ext^1_{K,\zeta}(\Sym^{p-2}k^2\otimes\det,\Sym^{p-2}k^2\otimes\det)$ which does exist by Lemma \ref{lemma-lattices}(ii) below. 
 Because $K_1$ acts trivially on $c_1$ and not trivially on $c_2$, they are linearly independent  and actually span the $k$-vector space $\Ext^1_{K,\zeta}(\Sym^{p-2}k^2\otimes\det,\pi^{K_1})$ for dimension reasons. This proves the claim.  As a byproduct, we obtain  \[\dim_k\Ext^1_{K,\zeta}(\Sym^{p-2}k^2\otimes\det,\Sym^{p-2}k^2\otimes\det)=1.\]

Since $\rsoc_K(\kappa(\rho))=\Sym^{p-2}k^2\otimes\det$, the extension $0\ra\pi \ra\kappa(\rho)\ra\pi'\ra0$ induces a non-zero element in $\Ext^1_{K,\zeta}(\Sym^{p-2}k^2\otimes\det,\pi)$, which by the claim equals to $\gamma(u_1c_1+u_2c_2)$ for  some $u_1,u_2\in k$. We observe that, since $\rsoc_K(\kappa(\rho))\cong \Sym^{p-2}k^2\otimes\det$, $\kappa(\rho)^{K_1}$ embeds into $\Sym^{2p-1}k^2$. Moreover, by construction of $c_1, c_2$, we see that $\kappa(\rho)^{K_1}\cong\Sym^{2p-1}k^2$ if and only if $u_2=0$.

Now we are ready to prove the proposition in this case.  
If $u_2\neq 0$, we take $\Theta=\Sym^{2p-1}\OO^2$. Since $\dim_k\Hom_K(\Sym^{p-2}k^2\otimes\det,\kappa(\rho))=1$ and $\Sym^{p-2}k^2\otimes\det$ appears twice in $\Theta/\varpi\Theta$, to verify (\ref{equation-cyclic=dim1})  it suffices to show that there does not exist any $K$-equivariant embedding $\Sym^{2p-1}k^2\hookrightarrow \kappa(\rho)$. But, if  such an embedding existed, $\Sym^{2p-1}k^2$ would be contained in $\kappa(\rho)^{K_1}$, hence is equal  to $\kappa(\rho)^{K_1}$, which contradicts the assumption $u_2\neq 0$.   If $u_2=0$, then we take   $\Theta$ to be the $\OO$-lattice constructed in Lemma \ref{lemma-lattices}(ii) below.  Again, to check (\ref{equation-cyclic=dim1}), we only need to show that there does not exist any $K$-equivariant embedding $\Theta/\varpi\Theta\hookrightarrow \kappa(\rho)$. Let $f$ be such an embedding. The image of $f$ can not be contained in $\pi$: otherwise, since  $(\Theta/\varpi\Theta)^{K_1}=\Sym^{p-2}k^2\otimes\det$, we would get an embedding \[\Sym^{p-2}k^2\otimes\det\hookrightarrow (\Theta/\varpi\Theta)/(\Theta/\varpi\Theta)^{K_1}\hookrightarrow\pi/\pi^{K_1}\]
which would contradict Theorem \ref{theorem-morra}(ii) since $p\geq 3$. Moreover, this shows that $\mathrm{im}(f)\cap \pi=\Sym^{p-2}k^2\otimes\det$, hence   $\mathrm{im}(f)/(\mathrm{im}(f)\cap \pi)\cong \Sym^pk^2$ embeds in $\pi'$ and in fact identifies with $\pi'^{K_1}$. This contradicts the assumption that the image of the extension class of $0\ra \pi\ra\kappa(\rho)\ra\pi'\ra0$ is equal to  $\gamma(u_1c_1)$ in $\Ext^1_{K,\zeta}(\Sym^{p-2}k^2\otimes\det,\pi)$. 
 \end{proof} 
\begin{rem}
When $\rho$ is reducible and $\mathcal{D}(\rho)=\{\Sym^{p-2}k^2\otimes\det\}$, it would be interesting to exactly determine $\kappa(\rho)^{K_1}$.   
\end{rem}

\begin{lem}\label{lemma-lattices}
Keep the notation in Theorem \ref{prop-cyclic} and its proof.
\begin{itemize}
\item[(i)] If $b=2p-1$, there exists a $K$-invariant $\OO$-lattice $\Theta\subset V$ such that $\Theta/\varpi\Theta$ is uniserial with socle and cosocle being respectively $\Sym^{p-1}k^2$ and $\Sym^0k^2$. Moreover, 
\begin{equation}\label{equation-lemma-lattices}\Hom_K(\Theta/\varpi\Theta, \Sp)=\Hom_K(\Theta/\varpi\Theta,\pi_{\alpha})=0.\end{equation}
 \item[(ii)] If $b=2p$, there exists a $K$-invariant $\OO$-lattice $\Theta\subset V$ such that $\Theta/\varpi\Theta$ is uniserial, with socle and cosocle being respectively $\Sym^{p-2}k^2\otimes\det$ and $\Sym^1k^2$. In particular, $\Ext^1_{K,\zeta}(\Sym^{p-2}k^2\otimes\det,\Sym^{p-2}k^2\otimes\det)\neq0$. 
\end{itemize}
\end{lem}
\begin{proof}
First observe that, since $V$ is an irreducible $K$-representation over $L$, the same proof as in \cite[Prop.VII.4.8]{co} shows that if $V$ contains a $K$-invariant $\OO$-lattice whose mod $\varpi$ reduction is a (possibly split) extension of  $A_2$ by $A_1$ for some $A_1,A_2\in \Mod_K^{\rm sm}(k)$, then there exists another $K$-invariant $\OO$-lattice such that its  mod $\varpi$ reduction is a \emph{non-split} extension of $A_1$ by $A_2$. 

(i) The natural lattice $\Sym^{2p-2}\OO^2$ has mod $\varpi$ reduction isomorphic to $A_1\oplus A_2$, where  $A_2=\Sym^{p-1}k^2$ and $A_1$ is the unique non-split extension
\begin{equation}\label{equation-A1}0\ra \Sym^{p-3}k^2\otimes\det\ra A_1\ra\Sym^0k^2\ra0.\end{equation}
The above observation implies the existence of a lattice $\Theta$ such that $\Theta/\varpi\Theta$ is a non-split extension of  $A_1$ by $A_2$, i.e. the class of $\Theta/\varpi\Theta$ in $\Ext^1_K(A_1,A_2)$ is non-zero. In particular, $\Theta/\varpi\Theta$ admits $\Sym^{p-1}k^2$ as a sub-representation and $\Sym^0k^2$ as a quotient.  Using the fact $\Ext^1_K(\Sym^0k^2,\Sym^{p-1}k^2)=0$ by \cite[Cor.5.6]{bp}, we see that the induced extension between $\Sym^{p-3}k^2\otimes\det$ and $\Sym^{p-1}k^2$ must be non-split. Therefore $\Theta/\varpi\Theta$ is uniserial, with socle being $\Sym^{p-1}k^2$ and cosocle being $\Sym^0k^2$.

We now prove (\ref{equation-lemma-lattices}). If $f\in \Hom_K(\Theta/\varpi\Theta,\Sp)$ is  non-zero then it is injective, as $\rsoc_K(\Sp)=\Sym^{p-1}k^2$. Hence $f$  induces an injection $A_1\hookrightarrow \Sp/\rsoc_K(\Sp)$, which contradicts Theorem \ref{theorem-morra}(iii). 
If $f\in\Hom_K(\Theta/\varpi\Theta,\pi_{\alpha})$ is non-zero then $f$ must factor through $A_1\hookrightarrow \pi_{\alpha}$, since $\rsoc_K(\pi_{\alpha})\cong \Sym^{p-3}k^2\otimes\det$. This contradicts Theorem \ref{theorem-morra}(ii).

(ii) Let $A_2=\Sym^{p-2}k^2\otimes\det$ and $A_1$ be the unique non-split extension $0\ra \Sym^{p-2}k^2\otimes\det\ra A_1\ra \Sym^1k^2\ra0$. The natural lattice $\Sym^{2p-1}\OO^2$ has a mod $\varpi$ reduction of the form
\[0\ra A_1\ra \Sym^{2p-1}k^2\otimes\det\ra A_2\ra0.\] So by the observation, we get another lattice $\Theta$ whose mod $\varpi$ reduction fits in  a non-split extension
\[0\ra A_2\ra \Theta/\varpi\Theta\ra A_1\ra0.\]
Using the fact that $\Ext^1_K(\Sym^1k^2,\Sym^{p-2}k^2\otimes\det)$ is of dimension $1$ over $k$, see \cite[Cor.5.6]{bp}, 
one checks that $\Theta$  satisfies the desired properties.  
\end{proof}

\subsection{Gorensteinness} It is natural to ask if the ring $R$ is Gorenstein or a complete intersection ring. We first note that these properties often coincide in our situation. 

\begin{prop}\label{gor_ci}
Assume that $R^{\psi}_{\rho}$ is  formally smooth over $\cO$. For any $p$-adic Hodge type $(\w,\tau)$ such that $R^{\psi}_{\rho}(\w,\tau)$ is non-zero, the ring $R^{\psi}_{\rho}(\w,\tau)$ is Gorenstein if and only if it is a complete intersection.
\end{prop}
\begin{proof}
It follows from a theorem of Serre which says that if $A$ is a regular local ring and $J\subset A$ is an ideal of codimension 2, then $A/J$ is Gorenstein if and only if $A/J$ is a complete intersection, see \cite[Cor.21.20]{eis}.
\end{proof}

\begin{prop}\label{prop-criterion1-Gorenstein}
Let $B$ be the closure in the Banach space  $\Hom_{\OO}^{\cont}(N/ xN, L)$ of the $G$-subrepresentation generated by the $V$-isotypic subspace. 
If $R$ is Cohen--Macaulay then  
 $R$ is Gorenstein if and only if 
 $$\dim_k\Hom_G(B^0 \otimes_{\OO} k, \kappa(\rho))=1.$$
 \end{prop}
 
\begin{proof}
Since $R$ is Cohen-Macaulay by assumption, $(x,\varpi)$ is a regular sequence in $R$ by Lemma \ref{lemma-CM-equiv}. So $R$ is Gorenstein if and only if $R/(x,\varpi)$ is Gorenstein, which holds if and only if 
\[\dim_k\Hom_{R^{\psi}_{\rho}}(k,R/(x,\varpi))=1.\]
It follows from the definition of $\mathfrak a_1$ that  Banach spaces   $\Hom_{\OO}^{\cont}(N/ xN, L)$ and $\Hom_{\OO}^{\cont}(N/ \mathfrak{a}_1 N, L)$ have the same $V$-isotypic subspace. Hence $B$ coincides with $B$ in Lemma \ref{nearly_there}. We have isomorphisms 
\begin{equation}
\begin{split}
 \Hom_G&(B^0 \otimes_{\OO} k, \kappa(\rho))\cong \Hom_{\dualcat(\OO)}(k \wtimes_{R^{\psi}_{\rho}} N, (B^0\otimes_{\OO} k)^{\vee})\\ &\cong 
 \Hom_{\dualcat(\OO)}(k \wtimes_{R^{\psi}_{\rho}} N, Q_1\otimes_{\OO} k)\cong \Hom_{R^{\psi}_{\rho}}(k,  R^{\psi}_{\rho}/(\mathfrak a_1, \varpi)).
 \end{split} 
 \end{equation}
Indeed, the first isomorphism is obtained by taking Pontryagin dual, the second  is obtained 
by taking the isomorphism in Lemma \ref{nearly_there} and taking $\mm$-torsion, and the third  follows from Corollary \ref{corollary-passtoN}.
It follows from Lemma \ref{lemma-CM-equiv}(iii) that $R^{\psi}_{\rho}/(\mathfrak a_1, \varpi)=R/(x, \varpi)$.
\end{proof}

\begin{rem}
Although it looks pretty, the criterion in Proposition \ref{prop-criterion1-Gorenstein} is rather hard to check in practice. For example, if we are in the 
setting of Theorem \ref{theorem-main} then we know that $R$ is Cohen--Macaulay and it follows from Proposition \ref{prop-Pi^alg} that the representation $B^0 \otimes_{\OO} k$ is isomorphic to $H_0(\mathcal{X})\otimes_{\OO} k$, 
thus we have an exact sequence: 
$$ 0\rightarrow \Hom_G(B^0 \otimes_{\OO} k, \kappa(\rho))\rightarrow 
\Hom_{\mathfrak{R}_0}( \mathcal{X}_0\otimes_{\OO} k, \kappa(\rho))\rightarrow \Hom_{\mathfrak{R}_1}( \mathcal{X}_1\otimes_{\OO} k, \kappa(\rho)).$$
Now $\mathcal X_0$ is just the lattice in the $V$-isotypic subspace of $\Hom_{\OO}^{\cont}(N/ xN, L)$ induced by the natural norm on this space. Determining this lattice explicitly when the multiplicity is large seems to be a very hard question. 
\end{rem} 

If $\rho$ is \emph{generic} then $N$ is  projective in $\dualcat(\OO)$. In this case, we are able to give another criterion for $R$ to be Gorenstein using information from $\cO$-lattices in $V$. Before stating it, we prove the following lemma.

\begin{lem}\label{lemma-xN=0}
Assume $N$ is projective in $\CO$. Let $Q\in\CO$ be a quotient of $N/xN$.  
Then $\Hom_{\CO}(N,Q)$ is a cyclic $R^{\psi}_{\rho}$-module and the natural morphism
\[\Hom_{\CO}(N/xN,Q)\ra \Hom_{\CO}(N,Q) \]
 is an isomorphism.
\end{lem}
\begin{proof}
Since $N$ is projective in $\CO$, applying $\Hom_{\CO}(N,*)$ to the natural  composite map $r: N\twoheadrightarrow N/xN\twoheadrightarrow Q$ gives a surjection:
\begin{equation}\label{equation-lemma-cyclic}\Hom_{\CO}(N,N)\ra \Hom_{\CO}(N,Q)\ra0.
\end{equation}
Since $R^{\psi}_{\rho}\cong \End_{\CO}(N)$, this implies that $\Hom_{\CO}(N,Q)$   is a  cyclic $R^{\psi}_{\rho}$-module. Using again (\ref{equation-lemma-cyclic}), 
we can  lift any $\iota\in \Hom_{\CO}(N,Q)$ to $y\in R^{\psi}_{\rho}=\End_{\CO}(N)$, and get
the following diagram \[\xymatrix{0\ar[r]& N\ar^{x}[r]\ar_y[d]&N \ar_y[d]\ar^{\iota}[rd]& &\\
0\ar[r]&N\ar^{x}[r]&N\ar^{r}[r]&Q\ar[r]&0}\]
which is commutative as $R^{\psi}_{\rho}$ is a commutative ring.
Then we see that $\iota\circ x=0$ since $r\circ x=0$.
 \end{proof}

 \begin{prop}\label{prop-Gorenstein-criterion}
 Assume that $\rho$ is generic and that there exist both a $K$-invariant $\OO$-lattice $\Theta$ in $V$ and an $\OK$-linear injective morphism 
$\theta:\Theta/\varpi\Theta\hookrightarrow (N/xN)^{\vee}$
such that the following conditions are satisfied: 
 \begin{enumerate}
 \item[(i)] $\dim_k\Hom_{K}(\Theta/\varpi\Theta, \kappa(\rho))=1$;
 \item[(ii)]
 if we let $W$ be the $G$-subrepresentation generated by the image of $\theta$,  then $[W^{\rm ss}:\pi]$, the multiplicity of $\pi$ in $W^{\rm ss}$, is equal to $e$. Here recall that $\pi$ is the $k$-representation of $G$ defined in \S\ref{subsubsection-generic}.
 \end{enumerate}
Then $R$ is Gorenstein. 
 \end{prop}
 \begin{proof}
As discussed in \S\ref{subsection-cyclicity}, the assumption (i) ensures that $M(\Theta)$ is a free $R$-module of rank $1$ and $R$ is Cohen-Macaulay. Therefore $(x,\varpi)$ is a regular sequence in $R$ by Lemma \ref{lemma-CM-equiv}, and $R$ is Gorenstein if and only if $R/(x,\varpi)$ is Gorenstein if and only if we have an isomorphism  
\[\Hom_{k}(R/(x,\varpi),k)\cong R/(x, \varpi)\]
of $R$-modules. The isomorphism $R \cong M(\Theta)$ implies that this is again equivalent to $\Hom_{k}(M(\Theta)/(x,\varpi),k)$ being a cyclic $R$-module. 
Since \[M(\Theta)/(x,\varpi) \cong\Hom_{k}\bigl(\Hom_{\OK}^{\rm cont}(N/(x,\varpi),\Theta^d/\varpi\Theta^d),k\bigr)\] we obtain that 
\[\Hom_{k}(M(\Theta)/(x,\varpi),k)\cong \Hom_{\OK}^{\rm cont}(N/(x,\varpi),(\Theta/\varpi\Theta)^{\vee}). \]

We claim that the following composite map is an isomorphism of $R^{\psi}_{\rho}$-modules: 
\begin{multline} \label{equation-f-composite}f:\Hom_{\CO}(N/(x,\varpi),W^{\vee})\ra \Hom_{\OK}^{\rm cont}(N/(x,\varpi),W^{\vee})\\
\ra \Hom_{\OK}^{\rm cont}(N/(x,\varpi),(\Theta/\varpi\Theta)^{\vee}),\end{multline}
where the first map is the restriction from $G$-equivariant morphisms to  $K$-equivariant morphisms. By definition of $W$, $f$ is injective. On the other hand,  
 let $d_1$ and $d_2$ be respectively the  dimension over $k$ of the first and last term in \eqref{equation-f-composite}. Then $d_2=e$ by Theorem \ref{thm-pa12-B}(iii). By Lemma \ref{lemma-xN=0}, the natural morphism
\[\Hom_{\CO}(N/(x,\varpi),W^{\vee})\simto\Hom_{\CO}(N,W^{\vee})  \] 
is an isomorphism. Since $N$ is a projective envelope of $\pi^{\vee}$ in $\CO$, we have $\dim_k\Hom_{\CO}(N,W^{\vee})=[W^{\rm ss}:\pi]$, so $d_1=e$ by the assumption (ii). This proves the claim and the proposition follows from Lemma \ref{lemma-xN=0}.  
\end{proof}

\begin{prop}\label{gore}
Assume $\rho$ is generic. Let $\w=(a,b)$ with $1\leq b-a\leq 2p$ and $\tau=\ide\oplus\ide$. Then $R$ is Gorenstein whenever non-zero.
\end{prop}  
\begin{proof}
We use the notation in \S\ref{subsection-cyclicity}. First note that if $e\leq 2$, then $R$ is Gorenstein since the local ring $R/(\varpi,x)$ is   of length 2 as a module over itself. 

 Since $\rho$ is assumed to be generic, the only case  when $e\geq 3$ is that $b-a=2p$ and $\rho$ is irreducible with  $\mathcal{D}(\rho)=\{\Sym^1k^2,\Sym^{p-2}k^2\otimes\det\}$; moreover we have exactly $e=3$. Let $\Theta=\Sym^{2p-1}\OO^2$, so that $\Theta/\varpi\Theta$ is of length 3, isomorphic to the injective envelope of $\Sym^{p-2}k^2\otimes\det$ in the category of finite dimensional $k$-representations of $\GL_2(\F_p)$.
Since $(N/xN)^{\vee}$ is injective in $\Mod_{K,\zeta}^{\rm sm}(\OO)$ and
$\Hom_{K}(\Sym^{p-2}k^2\otimes\det,(N/xN)^{\vee})\neq0$, there exists a $K$-equivariant \emph{injection} $\theta: \Theta/\varpi\Theta\hookrightarrow (N/xN)^{\vee}$. 
Let $W$ be the sub-representation generated by $\mathrm{im}(\theta)$. Since $\rho$ is irreducible as remarked above, $\pi$ is a supersingular representation of $G$ and \cite[Prop.5.42]{pa13} implies that  all the irreducible subquotients of $W$  are isomorphic to $\pi$. 

We claim that $W$ has length 3. To show this, we first observe that any non-zero morphism $\Sym^{2p-1}k^2\ra \pi$ vanishes on $\Ind_I^K(\omega\otimes \ide)$ by the proof of Theorem \ref{prop-cyclic}, hence  factors as
 \[\Theta/\varpi\Theta\overset{\mathrm{pr}}{\twoheadrightarrow} \Sym^{p-2}k^2\otimes\det\hookrightarrow \pi,\]
 where $\mathrm{pr}$ is the natural projection in (\ref{equation-case-d2}).
  This implies that $W$ has length $>1$,  otherwise $\theta$ would induce an injection $\Theta/\varpi\Theta\hookrightarrow \pi$.   
 If $W$ has length $2$, i.e. 
 \[0\ra \pi_1\ra W\ra \pi_2\ra0,\]
 with $\pi_1\cong\pi_2\cong \pi$,
then by the observation again we get \[\mathrm{im}(\theta)\cap \pi_1=\ker(\mathrm{pr})=\Ind_I^K(\omega\otimes \ide).\] But this contradicts Theorem \ref{theorem-morra}(i) since $\omega\otimes \ide\neq(\omega\otimes \ide)^s$.  This proves the claim and finishes the proof by Proposition \ref{prop-Gorenstein-criterion}.
\end{proof}

\begin{prop}\label{Cayley-Hamilton} Let $\wt$ and $\tau$ be arbitrary and let $\mm$ be the maximal ideal of $R$. If $\dim_k \mm/( \varpi, \mm^2)\le 2$ then 
$R\cong \OO[[x, y]]/( y^e+ a_{e-1} y^{e-1}+\cdots+ a_0)$, where $a_i$ lie in the maximal ideal of $\OO[[x]]$. In particular, $R$ is a complete intersection ring. 
\end{prop}
\begin{proof}  It follows from the construction of $x$ in Theorems \ref{thm-pa12-B}, \ref{thm-x-nongeneric} that its image in $\mm/( \varpi, \mm^2)\le 2$  is non-zero. If $\dim_k \mm/( \varpi, \mm^2)=1$ then since $x$ is $M(\Theta)$-regular, we deduce that $R\cong \OO[[x]]$. Let us assume that  
$\dim_k \mm/( \varpi, \mm^2)=2$ and let $y\in R$ be an element, such that the images of $x$ and $y$ build a $k$-basis of $\mm/( \varpi, \mm^2)$. 
Thus we have a surjection $\OO[[x,y]]\twoheadrightarrow R$. Since $M(\Theta)$ is a free $\OO[[x]]$-module of rank $e$, by fixing a basis we may represent 
the action of $y$ on $M(\Theta)$ by a matrix with entries in $\OO[[x]]$. If $y^e+ a_{e-1} y^{e-1}+\cdots+ a_0$ is the characteristic polynomial of that matrix
then, since $R$ acts faithfully on $M(\Theta)$, we deduce that the surjection $\OO[[x,y]]\twoheadrightarrow R$ factors through the surjection 
$C:= \OO[[x, y]]/( y^e+ a_{e-1} y^{e-1}+\cdots+ a_0)\twoheadrightarrow R$. Since $y$ lies in $\mm$ it acts on $M(\Theta)/(\varpi, x)$ nilpotently.
This implies that $a_i$ are contained in the maximal ideal of $\OO[[x]]$. Now $C/(\varpi)$ is a complete intersection ring of dimension $1$ with Hilbert--Samuel multiplicity equal to $e$. Since it cannot have embedded prime ideals any proper $1$-dimensional quotient of $C/(\varpi)$ will have a strictly smaller Hilbert--Samuel multiplicity. This implies that the surjection $C\twoheadrightarrow R$ induces an isomorphism modulo $\varpi$, and hence is an isomorphism.
\end{proof}

\begin{cor} Assume that $\rho$ is generic. Let $\w=(a,b)$ with $1\leq b-a\leq 2p$ and $\tau=\ide\oplus\ide$. If  $R$ is non-zero then $R\cong \OO[[x,y]]/( y^e+ a_{e-1} y^{e-1}+\cdots+ a_0)$, where $a_i$ lie in the maximal ideal of $\OO[[x]]$.
\end{cor} 
\begin{proof} Proposition \ref{gore} implies that $R/(\varpi, x)$ is Gorenstein and the dimension of $R/(\varpi, x)$ as a $k$-vector space is equal to $e$, which is $1$, $2$ or $3$. In all the cases, 
$\dim_k \mm/ (\varpi, x, \mm^2)=1$. This is obvious if $e=1$ or $e=2$. If $e=3$ then $\dim_k \mm/ (\varpi, x, \mm^2)=2$ would imply 
that $\mm^2$ is zero  in $R/(\varpi, x)$ and hence $\Hom_R( k, R/(\varpi, x))\cong \mm/ (\varpi, x, \mm^2)$ is $2$-dimensional, which contradicts Gorensteinness of $R/(\varpi, x)$.
The assertion follows from Proposition \ref{Cayley-Hamilton}.
\end{proof}

\section{Global applications}\label{global_app}

Let $F$ be a totally real field in which $p$ splits completely, let $\Sigma$ and $S$ be   finite sets  of places of $F$ containing all the places above $p$ and all the archimedean places, such that  $\Sigma \subset S$. 
We fix an algebraic closure $\overline{F}$ of $F$. Let $G_{F,S}$ be 
the absolute Galois group of the maximal extension of $F$ in $\overline{F}$ which is unramified outside $S$. 
Let $\rhobar: G_{F, S}\rightarrow \GL_2(k)$ be a continuous irreducible representation, which we assume  to be modular. 
  We will combine the result of the local part together with the results of Jack Shotton \cite{shotton}, \S \ref{appendix_jack} concerning the potentially semi-stable deformation rings at $\ell\neq p$ to prove that certain quotients of  the universal  deformation ring of $\rhobar$ parameterizing deformations which are potentially semistable of a given inertial type at $v\in \Sigma$ and given Hodge--Tate weights at $v\mid p$  are $\OO$-torsion free.  The strategy of the proof, given the local results, is well known, see for example \cite[\S 5]{sn} or 
Remark after Lemma 4.6 in \cite{kw2}. However, the result has not been stated in the generality that we prove it here, since the Cohen--Macaulayness of local potentially semi-stable deformation rings was known 
only in a few cases. We will then explain that these quotients are reduced. This is a byproduct of Kisin's approach to Taylor--Wiles method. This will allow us 
to upgrade Kisin's $R[1/p]=\mathbb{T}[1/p]$ theorem to an integral version $R=\mathbb T$. 
 
\subsection{Presentations of global deformation rings}\label{present}
 
 Let $\chi_{\cyc}: G_{F, S}\rightarrow \OO^{\times}$ be the global $p$-adic cyclotomic character. Let $\psi: G_{F, S}\rightarrow \OO^{\times}$ be a continuous character such
 that $\psi \chi_{\cyc}$ is congruent to $\det \rhobar$ modulo $\varpi$. If $p=2$ then we additionally assume that $\psi$ is totally even, 
 which means that its restriction to $G_{F_v}$ is trivial for all archimedean places $v$. If $p>2$ then this follows from the assumption
 on modularity of $\rhobar$.

 For each place $v$ of $F$ we fix an algebraic closure $\overline{F}_v$ of $F_v$ and an embedding 
 $\overline{F} \hookrightarrow \overline{F}_v$, which extends the embedding $F\hookrightarrow F_v$. This induces a continuous homomorphism of Galois groups 
 $G_{F_v}\hookrightarrow G_F \twoheadrightarrow  G_{F,S}$. This allows us to consider $\rhobar$ as a representation of $G_{F_v}$.
 
 We fix a basis of the underlying $k$-vector space $V$ of $\rhobar$. For each $v\in S$ let $R_v^{\square}$ be the universal framed deformation ring of $\rhobar_v:=\rhobar|_{G_{F_v}}$
 and let 
 $R_v^{\square, \psi}$ be the quotient of $R_v^{\square}$ parameterizing deformations of determinant $\psi\chi_{\cyc}$. We fix a subset $\Sigma$ of the set of finite places in $S$, which contains all the places above $p$ and all the infinite places. 
 Let $R^{\square, \psi}_{\Sigma}:=
 \wtimes_{v\in \Sigma} R_v^{\square, \psi}$, where the completed tensor product is taken over $\OO$.
 
  We define $R^{\psi}_{F, S}$ to be the quotient of the universal deformation ring of $\rhobar$, parameterizing deformations with determinant 
 $\psi\chi_{\cyc}$.  
 Denote by $R^{\square, \psi}_{F, S}$ the complete local $\OO$-algebra representing the functor which assigns to an artinian, augmented $\OO$-algebra 
$A$  the set of isomorphism classes of tuples $\{V_A,\beta_w\}_{w\in \Sigma}$, where $V_A$ is a deformation of $\rhobar$ to $A$ with determinant $\psi\chi_{\cyc}$ and $\beta_w$ is a lift of a chosen basis of $V_k$ to a basis of $V_A$. The map $\{V_A,\beta_w\}_{w\in \Sigma}\mapsto \{V_A, \beta_v\}$
 induces a homomorphism of $\OO$-algebras $R^{\square, \psi}_v\rightarrow R^{\square, \psi}_{F, S}$ for every $v\in \Sigma$ and 
 hence a homomorphism of $\OO$-algebras $R^{\square, \psi}_\Sigma \rightarrow R^{\square, \psi}_{F, S}$.

 Let $\ad \rhobar$ be  the $G_F$-representation obtained by letting $G_F$ act on 
$\End_k(V)$ by conjugation and let
$\ad^0(\rhobar)$ be  the subspace of endomorphisms having trace zero. 
We note that $\ad \rhobar\cong \rhobar \otimes (\rhobar)^*$ as $G_F$-representations, and hence $(\ad \rhobar)^*\cong \ad \rhobar$. 

\begin{lem}\label{H20} Let $\zeta_p$ be a primitive $p$-th root of unity. If $p>2$ and $\rhobar|_{G_{F(\zeta_p)}}$ is absolutely irreducible then $H^0(G_F, (\ad^0\rhobar)^*(1))=0$.
\end{lem}
\begin{proof}
Since the restriction of $\rhobar$ to $G_{F(\zeta_p)}$ is absolutely irreducible by assumption, 
$\Hom_{G_{F(\zeta_p)}}(\rhobar, \rhobar)$ is one dimensional. This implies that 
$\Hom_{G_F}( \rhobar, \rhobar \otimes \bar{\chi})=0$ for any non-trivial character $\bar{\chi}: G_F/ G_{F(\zeta_p)} \rightarrow k^{\times}$.
Since $p>2$ the cyclotomic character is not trivial modulo $\varpi$ and thus  $H^0(G_F, (\ad\rhobar)^*(1))$ is zero.  Since $p>2$, $\ad^0\rhobar$ is a direct summand of $\ad \rhobar$. Thus $H^0(G_F, (\ad^0\rhobar)^*(1))$ is also zero.
\end{proof}

\begin{prop}\label{presentation} If $\Sigma \neq S$ or $H^0(G_F, (\ad^0\rhobar)^*(1))=0$  then
for some non-negative integer $r$ there is an isomorphism of $R^{\square, \psi}_{\Sigma}$-algebras: 
$$   R^{\square, \psi}_{F, S} \cong R^{\square, \psi}_{\Sigma}[[ x_1, \ldots, x_{r+|\Sigma| -1}]]/ (f_1, \ldots, f_r),$$
where $f_1, \ldots, f_{r}\in R^{\square, \psi}_{\Sigma}[[ x_1, \ldots, x_{r+|\Sigma| -1}]]$.
\end{prop}
\begin{proof} This follows from \cite[Prop.4.1.5]{kisin_cdm} as we will now explain. Since all the infinite places of $F$ lie in $\Sigma$ the integer denoted 
by $s$ in \cite[Prop.4.1.5]{kisin_cdm} is zero. Moreover, as explained in \textit{loc. cit.}
our assumptions imply that 
the map denoted by $(\dagger)$ in \cite[Prop.4.1.5]{kisin_cdm} is injective, hence the conditions
of \cite[Prop.4.1.5]{kisin_cdm} are satisfied. 
\end{proof}

\begin{prop}\label{abstract_prop} For each $v\in \Sigma$ let $\bar{R}^{\square, \psi}_v$ be an equidimensional, $\OO$-torsion free quotient of $R^{\square, \psi}_v$ such that 
its dimension is $5$ if $v\mid p$, $3$ if $v\mid \infty$ and $4$ otherwise. Assume that the conditions of Proposition \ref{presentation} hold and that the deformation problem defined by $\bar{R}^{\square, \psi}_v$ does not 
depend on the framing. Let 
$$\bar{R}_\Sigma^{\square, \psi}:=\wtimes_{v\in \Sigma} \bar{R}_v^{\square, \psi}, \quad 
\bar{R}^{\square, \psi}_{F, S}:= \bar{R}_\Sigma^{\square, \psi}\otimes_{R_{\Sigma}^{\square, \psi}} R^{\square, \psi}_{F,S},$$
and let $\bar{R}^{\psi}_{F, S}$ be the image of 
$R^{\psi}_{F, S}$ under the natural map $R^{\psi}_{F,S}\rightarrow R^{\square, \psi}_{F, S}\rightarrow \bar{R}^{\square,\psi}_{F, S}$, where the first arrow is obtained by forgetting the framing. 

 If $\bar{R}^{\psi}_{F, S}$ is a finitely generated $\OO$-module then its rank is at least $1$ and every irreducible component of $\Spec \bar{R}_\Sigma^{\square, \psi}$ contains a point of $\mSpec \bar{R}^{\psi}_{F, S}[1/p]$. If we additionally assume that $\bar{R}^{\square, \psi}_v$ is Cohen--Macaulay for all $v\in \Sigma$ then  $\bar{R}^{\psi}_{F, S}$ is $\OO$-torsion free. 
\end{prop}

\begin{proof}

The ring $\bar{R}^{\square, \psi}_{F, S}$ parameterises all the framed deformations of $\rhobar$ satisfying the local deformation
conditions imposed at $v\in \Sigma$. Proposition \ref{presentation} implies that 
\begin{equation}\label{presentation2} 
\bar{R}^{\square, \psi}_{F, S}\cong \bar{R}_\Sigma^{\square, \psi}[[x_1, \ldots, x_{r+|\Sigma|-1}]]/(\bar{f}_1, \ldots, \bar{f}_{r}), 
\end{equation}
where $\bar{f}_1, \ldots, \bar{f}_{r} \in \bar{R}^{\square, \psi}_{\Sigma}[[ x_1, \ldots, x_{r+|\Sigma| -1}]]$ are the images of 
$f_1, \ldots, f_{r}$ of Proposition \ref{presentation} under the natural map.
By forgetting the framing we obtain a map $R^{\psi}_{F, S}\rightarrow R^{\psi, \square}_{F,S}$, which induces an isomorphism
$$R^{\square, \psi}_{F, S}\cong R^{\psi}_{F, S}[[y_1, \ldots, y_{4|\Sigma|-1}]].$$ Since the local deformation conditions at $v\in \Sigma$ do not interfere with framing we have:
\begin{equation}\label{presentation3} 
\bar{R}^{\square, \psi}_{F, S}\cong \bar{R}^{\psi}_{F, S}[[y_1, \ldots, y_{4|\Sigma|-1}]].
\end{equation}

Since $\bar{R}^{\square, \psi}_v$ is $\OO$-torsion free for all $v\in \Sigma$ it follows from Lemma \ref{equidim} that 
the Krull dimension of $\bar{R}_{\Sigma}^{\square, \psi}$ is equal to 
$$1+\sum_{v\in \Sigma} \dim \bar{R}_v^{\square, \psi}[1/p]= 1+ 4 \sum_{v\mid p} 1 + 2\sum_{v\mid \infty}1+3( |\Sigma|- \sum_{v\mid p} 1-\sum_{v\mid\infty} 1)= 3|\Sigma|+1,$$
where in the last equality we have used the assumption that $F$ is totally real and $p$ splits completely in $F$. It follows from \eqref{presentation2} that 
$\dim \bar{R}^{\square, \psi}_{F, S}\ge 4|\Sigma|$ and \eqref{presentation3} implies that $\dim \bar{R}^{\psi}_{F, S}\ge 1$. Since 
$\bar{R}^{\psi}_{F, S}$ is a finitely generated $\OO$-module by assumption, we deduce that the rank of $\bar{R}^{\psi}_{F, S}$ 
 as an $\OO$-module is at least $1$.  
 
 Choose elements $\bar{f}_{r+1}, \ldots, \bar{f}_{r+4|\Sigma|-1}\in \bar{R}^{\square, \psi}_\Sigma[[x_1, \ldots, x_{r+|\Sigma|-1}]]$ which map to elements 
 $y_1, \ldots, y_{4|\Sigma|-1}$ via the isomorphisms \eqref{presentation2} and \eqref{presentation3}. Since 
 \begin{equation}\label{present5}
 \bar{R}^{\psi}_{F, S}\cong \bar{R}^{\square, \psi}_{\Sigma}[[x_1, \ldots, x_{r+|\Sigma|-1}]]/(\bar{f}_1, \ldots, \bar{f}_{r+4|\Sigma|-1})
 \end{equation}
 and $\bar{R}^{\psi}_{F, S}$ is a finitely generated $\OO$-module of rank at least $1$, we deduce that 
 $\bar{f}_1, \ldots, \bar{f}_{r+4|\Sigma|-1},\varpi$ is a system of parameters in 
 $\bar{R}^{\square, \psi}_\Sigma[[x_1, \ldots, x_{r+|\Sigma|-1}]]$. Since $\bar{R}^{\square, \psi}_v$ are equidimensional for all $v\in \Sigma$, 
 Lemma \ref{equidim} implies that $\bar{R}^{\square, \psi}_{\Sigma}$ is equidimensional, and hence the same holds for the formal power series ring over $\bar{R}^{\square, \psi}_{\Sigma}$.   Lemma 3.9 of \cite{blocksp2} implies that 
every irreducible component of $\Spec \bar{R}^{\square, \psi}_{\Sigma}[[x_1, \ldots, x_{r+|\Sigma|-1}]]$ contains a 
maximal ideal of $\bar{R}_{F, S}^{\psi}[1/p]$. Since every irreducible component of  $$\Spec \bar{R}^{\square, \psi}_{\Sigma}[[x_1, \ldots, x_{r+|\Sigma|-1}]]$$ is of the form $V(\pp[[x_1,\ldots, x_{r+|\Sigma|-1}]])$ for a minimal prime $\pp$ of $\bar{R}^{\square, \psi}_{\Sigma}$, we deduce that every irreducible component 
of $\Spec \bar{R}_{\Sigma}^{\square, \psi}$ meets $\Spec \bar{R}^{\psi}_{F, S}[1/p]$.
 
If $\bar{R}_{v}^{\square, \psi}$ is Cohen--Macaulay for each $v\in \Sigma$ then $\bar{R}_{\Sigma}^{\square, \psi}$ is also Cohen--Macaulay. 
The proof of this claim, which we leave as an exercise for the reader, uses that the rings $\bar{R}_{v}^{\square, \psi}$ 
are $\OO$-torsion free, which implies that the functor $\wtimes_{\OO} \bar{R}_v^{\square, \psi}$ is exact.
 Since a system of parameters in a Cohen--Macaulay ring is a regular sequence by \cite[Theorem 17.4]{mat},  it follows from \eqref{present5} that  $\varpi$ is 
 $\bar{R}^{\psi}_{F, S}$-regular, and so $\bar{R}^{\psi}_{F, S}$ is $\OO$-torsion free. 
\end{proof}
\subsection{Global potentially semi-stable deformation rings}\label{global_pst} We will now specify the quotients $\bar{R}^{\square, \psi}_v$ in Proposition \ref{abstract_prop}.

If $v$ is infinite then we let $\bar{R}^{\square, \psi}_v=R^{\square, \psi}_v$. Let $c\in \Gal(\mathbb{C}/\mathbb{R})$ be the complex conjugation. If $p>2$ then the assumption that 
$\rhobar$ is modular implies that $\det \rhobar(c)=-1$, hence $\psi(c)=1$. Then $\bar{R}^{\square, \psi}_v\cong \OO[[x, y,z]]/( (1+x)^2+yz-1)\cong \OO[[y,z]]$ 
and the universal framed deformation is given by sending $c$ to the matrix $\bigl( \begin{smallmatrix} 1+ x& y\\ z & -1-x \end{smallmatrix} \bigr)$. If $p=2$ then 
$\psi$ is totally even by assumption and $R^{\square, \psi}_v$ is an integral domain of dimension $3$,  which is a complete intersection ring, and hence Cohen--Macaulay, see \cite[Prop.2.5.6]{kisin_serre_2}. 

For each finite $v\in \Sigma$ we fix a semisimple representation $\tau_v: I_{F_v} \rightarrow \GL_2(L)$, where $I_v$ is the 
inertia subgroup of $G_{F_v}$.

If  $v\nmid p$  we let $\bar{R}_v^{\square, \psi}$ be the maximal reduced and $p$-torsion free quotient of $R_v^{\square, \psi}$ all of whose $\overline{L}$-points give rise to representations $\rho$ of $G_{F_v}$, such that the semisimplification 
of the restriction of $\rho$ to $I_{F_v}$ is isomorphic to $\tau_v$. If $\bar{R}_v^{\square, \psi}$ is non-zero then it is equidimensional of dimension $4$
and the irreducible components of $\bar{R}_v^{\square, \psi}[1/p]$ are formally smooth, \cite[Thm.4.1.1]{pil}. The ring $\bar{R}_v^{\square, \psi}$ is Cohen--Macaulay by \cite[\S 5.5]{shotton} in the case $p>2$ and by Theorem  \ref{thm:CM} proved by Jack Shotton in the appendix below in the case  $p=2$.

For each $v\mid p$ we additionally fix a pair of integers $\wt_v=(a_v,b_v)$ with $a_v<b_v$.   We let $\bar{R}_v^{\square, \psi}$ be the maximal reduced and $p$-torsion free quotient of $R_v^{\square, \psi}$ all of whose $\overline{L}$-points give rise to representations $\rho$ of $G_{F_v}$, which are potentially semistable of type $(\mathbf{w}_v, \tau_v, \psi)$. If $\bar{R}_v^{\square, \psi}$ is non-zero then $\bar{R}_v^{\square, \psi}$ is equidimensional of dimension $5$ by \cite[Thm.3.3.4]{ki08}.
If $\End_{G_{F_v}}(\rhobar)=k$ then $\bar{R}_v^{\square, \psi}$ is formally smooth of relative dimension $3$ over $R^{\psi}_{\rhobar_v}(\wt_v, \tau_v)$. 
Thus if we are in the setting of Theorems \ref{theorem-main}, \ref{gen_split} then $\bar{R}_v^{\square, \psi}$ is Cohen--Macaulay.

Once we have specified the quotients $\bar{R}_v^{\square, \psi}$ for all $v\in \Sigma$ we let $\bar{R}_\Sigma^{\square, \psi}$, $\bar{R}^{\square, \psi}_{F, S}$ and $\bar{R}^{\psi}_{F, S}$ be as in 
Proposition \ref{abstract_prop}.  We will refer to $\bar{R}^{\psi}_{F, S}$ as a global potentially semi-stable deformation ring.  We will show that 
$\bar{R}^{\psi}_{F, S}$ is a finitely generated $\OO$-module and $\bar{R}^{\psi}_{F, S}[1/p]$ is reduced. 

\subsection{Odd primes} In this subsection we assume that $p>2$. We further assume that $\rhobar$ is modular, 
the restriction of $\rhobar$ to $G_{F(\zeta_p)}$ is irreducible,  if $p=5$ then we further assume that the projective image of $\rhobar|_{G_{F(\zeta_p)}}$ is not isomorphic to $A_5$ and if $p=3$ then we further assume $\rhobar_v\not\sim \bigl ( \begin{smallmatrix}\chi \omega & \ast \\ 0 & \chi \end{smallmatrix}\bigr )$ for any character $\chi: G_{F_v}\rightarrow k^{\times}$ and for any $v\mid p$. We note that Lemma \ref{H20} implies that the assumptions in Proposition \ref{presentation} are satisfied.

\begin{prop}\label{finiteO}  The ring  $\bar{R}^{ \psi}_{F, S}$ is a finitely generated $\OO$-module. 
\end{prop}
\begin{proof} If $\wt_v=(0,1)$ for all $v\mid p$ then the assertion follows from Lemma 4.4.3(3) of \cite{gee_kisin}. We will deduce the general 
case from this result and (a weak form of) the geometric Breuil--M\'ezard conjecture proved in \cite{eg, ht, pa12}.

We fix a place $v\mid p$. Let $\theta_{v,1}, \theta_{v, 2}:I_{F_v}\rightarrow L^{\times}$ be smooth characters such that $\theta_{v,1}\theta_{2,v}=\psi|_{I_{F_v}}$.  We view $\theta_{v,1}$ and $\theta_{v,2}$ as  characters of 
$\OO_{F_v}^{\times}= \Zp^{\times}$ via the class field theory. Let $c$ be the smallest integer such that $\theta_{v,1}\theta_{v,2}^{-1}$ is 
trivial on $1+p^c \Zp$.  We assume that $c\ge 2$. Let $J_c=\biggl( \begin{smallmatrix} \Zp^{\times} & \Zp\\ p^c\Zp & \Zp^{\times}\end{smallmatrix}\biggr)$,  
$K=\GL_2(\Zp)$ and $\theta_v: J_c\rightarrow L^{\times}$ be the character $\bigl( \begin{smallmatrix} x & y\\ z & w\end{smallmatrix}\bigr)\mapsto \theta_{v,1}(x) \theta_{v,2}(w)$. The representation $\sigma(\tau_v):=\Ind_{J_c}^K \theta_v$ has the following property: 
if $\pi$ is an irreducible smooth $\overline{L}$-representation of $\GL_2(\Qp)$  then 
$\Hom_K(\sigma(\tau), \pi)\neq 0$ if and only if $r|_{I_{\Qp}}\cong \tau_v:=\theta_{v,1}\oplus\theta_{v,2}$, where $r$ is the Weil--Deligne representation corresponding to 
$\pi$ via the classical local Langlands correspondence.
The central character 
of $\sigma(\tau_v)$ is equal to $\theta_v|_{\Zp^{\times}}$. Let us choose  a $K$-invariant $\OO$-lattice in $\sigma(\tau_v)$ and  let 
$\overline{\sigma(\tau_v)}$ be the semi-simplification of its reduction modulo $\varpi$.
Since $c\ge 2$ it follows from 
\cite[Lem.A.3]{pa10} that every irreducible $k$-representation $\sigma$ of $K$ with central character  congruent to $\theta_v|_{\Zp^{\times}}$ modulo $\varpi$ occurs 
as a subquotient of $\overline{\sigma(\tau_v)}$. In particular, the Serre weights corresponding to $\rhobar_v$ will also occur as subquotients
of $\overline{\sigma(\tau_v)}$.

Let $\widetilde{R}_v^{\square, \psi}$ be the maximal reduced, $p$-torsion free quotient of $R_v^{\square, \psi}$ 
all of whose $\overline{L}$-points give rise to representations  of $G_{F_v}$, which are potentially semistable
 of type $((0,1), \tau_v, \psi)$. It follows from the geometric Breuil--M\'ezard conjecture proved in \cite{eg, ht, pa12} that 
 $\widetilde{R}_v^{\square, \psi}\neq 0$ and for all potentially semi-sta\-ble deformation rings $\bar{R}_v^{\square, \psi}$ the underlying topological space of $\Spec \bar{R}_v^{\square, \psi}/(\varpi)$ is contained in the underlying topological 
 space of $\Spec \widetilde{R}_v^{\square, \psi}/(\varpi)$. We define the ring $\widetilde{R}^{\psi}_{F, S}$ in the same way as the ring $\bar{R}^{\psi}_{F, S}$ with the same deformation conditions
 at $v\nmid p$, but setting
 $\wt_v=(0,1)$ and $\tau_v=\theta_{v,1}\oplus \theta_{v,2}$ as above for every $v\mid p$. 
  
 If $A$ is a commutative ring, $B$ is an $A$-algebra   and $\mathfrak a_1$, $\mathfrak a_2$ are ideals of $A$ such that $V(\mathfrak a_1)\subset V(\mathfrak a_2)$ 
 in $\Spec A$ then $V(\mathfrak a_1 B)\subset V(\mathfrak a_2 B)$ in $\Spec B$ as $V(\mathfrak a B)$ is the preimage of $V(\mathfrak a)$ 
  in $\Spec B$. Combining this observation with the local results explained above we obtain an inclusion of the underlying topological spaces 
  $\Spec \bar{R}_{F, S}^{\psi}/(\varpi) \subset \Spec \widetilde{R}^{\psi}_{F, S}/(\varpi)$ in $\Spec R^{\psi}_{F, S}$. Since  $\widetilde{R}^{\psi}_{F, S}$ is a finitely generated $\OO$-module by Lemma 4.4.3(3) of \cite{gee_kisin}, where we take both subsets $\Sigma$ and $\Sigma'$ in 
  \textit{loc.~cit.} to be empty, the underlying topological space of  $\Spec \widetilde{R}^{\psi}_{F, S}/(\varpi)$ consists
  only of the maximal ideal of $R^{\psi}_{F, S}$. This implies that $\bar{R}_{F, S}^{\psi}/(\varpi)$ is a finite dimensional $k$-vector space 
  and Nakayama's lemma  for compact $\OO$-modules implies that $\bar{R}_{F, S}^{\psi}$ is a finitely generated $\OO$-module.
\end{proof} 

\begin{prop}\label{reduced} The ring  $\bar{R}^{\psi}_{F, S}[1/p]$ is reduced. 
\end{prop} 
\begin{proof} We may assume that $\bar{R}^{\psi}_{F, S}[1/p]$ is non-zero. If $\wt_v=(0,1)$ for all $v\mid p$ then the result 
follows from \cite[Lem.4.4.3(3)]{gee_kisin}. We will explain how to modify the proof of \textit{loc.~cit.} so that the result holds in general.

We will first reduce the proof to the case when $[F:\Q]$ is even. Let $F'$ be a totally real quadratic extension of $F$ linearly disjoint from 
$\overline{F}^{\ker \rhobar}(\zeta_p)$ over $F$ such that $p$ splits completely 
in $F'$. Let $S'$ be the set of places of $F'$ above $S$ and let $\psi'$ be the restriction of $\psi$ to $G_{F'}$. The composition
$G_{F'}\rightarrow G_F\twoheadrightarrow G_{F, S}$ factors through $G_{F', S'}$. Let $N$ be the image of $G_{F', S'}$ in $G_{F, S}$. 
Then $N$ is an open  normal subgroup of $G_{F, S}$ of index $1$ or $2$. In particular, the index is not divisible by $p$. The inflation-restriction 
exact sequence implies that restriction induces an injection 
$$ H^1(G_{F, S}, \ad^0 \rhobar)\hookrightarrow H^1(N, \ad^0\rhobar).$$
Since the action of $G_{F', S'}$ on $\ad^0\rhobar$  factors through the action of $N$ we have 
$$B^1(N, \ad^0\rhobar)=B^1(G_{F', S'}, \ad^0\rhobar), \quad Z^1(N, \ad^0\rhobar)\subset Z^1(G_{F',S'}, \ad^0 \rhobar).$$ 
Hence, the group homomorphism $G_{F', S'}\twoheadrightarrow N$ induces an  injection 
$$H^1(N, \ad^0\rhobar)\hookrightarrow  H^1(G_{F', S'}, \ad^0\rhobar).$$
 We conclude that the group homomorphism $G_{F', S'}\rightarrow G_{F, S}$ induces an injection 
$$ H^1(G_{F, S}, \ad^0 \rhobar)\hookrightarrow H^1(G_{F', S'}, \ad^0\rhobar).$$
The restriction to $G_{F'}$ induces a homomorphism of local $\OO$-algebras $R^{\psi'}_{F', S'}\rightarrow R^{\psi}_{F, S}$. The injection 
on the $H^1$-groups above implies that the map induces a surjection on tangent spaces, and hence the map is surjective. For each 
$v\in \Sigma$ and each $w\mid v$ we let $\tau_w:= \tau_v|_{F'_w}$. Moreover, for each $v\in \Sigma$ such that $v\mid p$ and each 
$w\mid v$ we let $\wt_w=\wt_v$. The surjection $R^{\psi'}_{F', S'}\twoheadrightarrow R^{\psi}_{F, S}\twoheadrightarrow \bar{R}^{\psi}_{F,S}$
factors through the map $\bar{R}^{\psi'}_{F', S'}\rightarrow \bar{R}^{\psi}_{F, S}$. We conclude that this map and 
hence the map $\bar{R}^{\psi'}_{F', S'}[1/p]\rightarrow \bar{R}^{\psi}_{F, S}[1/p]$ are surjective. Since both rings are zero dimensional 
by Proposition \ref{finiteO}, if $\bar{R}^{\psi'}_{F', S'}[1/p]$ is reduced, so is $\bar{R}^{\psi}_{F,S}[1/p]$.  The degree $[F':\Q]$ is even as 
$F'$ is a quadratic extension of $F$.

Let us assume that $[F:\Q]$ is even and let $D$ be the quaternion algebra with centre $F$, which is  split at all finite and is ramified at all infinite places. We fix a maximal order $\OO_D$ of $D$, 
and an isomorphism $(\OO_D)_v\cong  M_2(\OO_{F_v} )$ at each finite place $v$. Let $\mathbb{A}_F^{\infty}$ be the ring of finite adeles over $F$.  Let $U=\prod_v U_v$ be a compact open subgroup of 
$(D\otimes_F \mathbb{A}_F^{\infty})^{\times}$ such  that $U_v=(\OO_D)^{\times}_v$ for all $v\in \Sigma$ and all $v\not \in S$, except for one carefully chosen place $v_1\not\in S$  in \cite[\S 4.3.2]{gee_kisin}, where
 $U_{v_1}$ is the subgroup of $(\OO_{D})_{v_1}^{\times}\cong \GL_2(\OO_{F_{v_1}})$ consisting of elements which are upper triangular unipotent modulo $v_1$. If $v\in S\setminus \Sigma$ is  a finite place  then we choose $U_v$ to be an open compact subgroup of $(\OO_D)_v$, which is small enough in a sense to be made precise below. 
 
If $v$ is a finite place and $\tau_v: I_{F_v}\rightarrow \GL_2(L)$ is a representation with an open kernel then we let  $\sigma(\tau_v)$ be a smooth irreducible 
 $L$-representation of $\GL_2(\OO_{F_v})$ defined as in \cite[\S 3.2]{shotton} following \cite{he}. If $\tau_v$ is scalar, so that $\tau_v\cong \theta_v \oplus \theta_v$, then we let 
 $\sigma(\tau_v)$ be the inflation of the Steinberg representation of $\GL_2(k_v)$ twisted by $\theta_v\circ\det$. 
 
 If $v\in S\setminus \Sigma$ is a finite place then there are only finitely many $\tau_v: I_{F_v}\rightarrow \GL_2(\overline{L})$ such that 
 there is a lift $\rho: G_{F_v}\rightarrow \GL_2(\overline{L})$ of $\rhobar|_{G_{F_v}}$ such that the semi-simplification of 
 $\rho|_{I_{F_v}}$ is isomorphic to $\tau_v$, \cite[Lem.3.7, 3.8]{shotton}. In this case we choose $U_v$ to be an open subgroup
 of $(\OO_D)_v^{\times}$, which is contained in the kernel of $\sigma(\tau_v)$ for each $\tau_v$ as above.

 For every $v\mid p$ we let 
 $\sigma(\wt_v, \tau_v)= \Sym^{b_v-a_v-1} L^2 \otimes \det^{a_v} \otimes \sigma(\tau_v)$. 
 We fix a $U_v$-invariant $\OO$-lattice 
 $\mathcal{L}_v$ in $\sigma(\tau_v)$ if $v\nmid p$ and in $\sigma(\wt_v, \tau_v)$ if $v \mid p$. 
 Then $\mathcal{L}:=\otimes_{v\in \Sigma} \mathcal {L}_v$ is a free $\OO$-module of finite rank with a continuous action of $U$, 
 where each factor $U_v$ acts trivially if $v\not\in \Sigma$ and on the
 factor $\mathcal{L}_v$ if $v\in \Sigma$. The central character of $\mathcal L$ is equal to the restriction of $\psi$ to $U\cap (\mathbb{A}_F^{\infty})^{\times}$.
 
 We let $S_{\tau}(U, \OO)$ be the set of continuous functions $f: D^{\times}\backslash (D\otimes_F \mathbb A_F^{\infty})^{\times}\rightarrow 
 \mathcal L$ such that for $g\in (D\otimes_F \mathbb A_F^{\infty})^{\times}$ we have $f(gu)= u^{-1} f(g)$ for all $u\in U$    and $f(gz)=\psi^{-1}(z) f(g)$ for all $z\in (\mathbb A_F^{\infty})^{\times}$. Then $S_{\tau}(U, \OO)$ is a free $\OO$-module of finite rank. Let 
 $\mathbb{T}_{S, \OO}^{\univ}=\OO[T_v, S_v]_{v\not\in S, v\neq v_1}$ be a commutative polynomial ring in the indicated formal variables. 
Then $S_{\tau}(U, \OO)$ is naturally a $\mathbb{T}_{S, \OO}^{\univ}$-module with $T_v$ acting as a Hecke operator corresponding 
to the double coset $U_v\bigl ( \begin{smallmatrix} \varpi_v & 0 \\ 0 & 1 \end{smallmatrix}\bigr) U_v$ and $S_v$ acting as a 
Hecke operator corresponding to the double coset $ U_v\bigl ( \begin{smallmatrix} \varpi_v & 0 \\ 0 & \varpi_v \end{smallmatrix}\bigr) U_v$.
Let $\mathbb{T}$ be the image of $\mathbb{T}_{S, \OO}^{\univ}$ in $\End_{\OO}(S_{\tau}(U, \OO))$. It follows from \cite[Lem.1.3(1)]{taylor_deg2}  that the action of $\mathbb T[1/p]$ on $S_{\tau}(U, \OO)\otimes_{\OO} L$ is semi-simple. 

There is a continuous representation $\rho^{\mathrm{mod}}: G_{F, S}\rightarrow \GL_2(\mathbb T\otimes_{\OO} \overline{L})$, such that for each $v\not\in S$ 
we have $\tr \rho^{\mathrm{mod}}(\Frob_v)= T_v$ and $\det \rho^{\mathrm{mod}}(\Frob_v)=S_v \mathbf Nv$, where $\mathbf Nv$ denotes the 
number of elements in the residue field of $v$ and $\Frob_v$ is the arithmetic Frobenius, see \cite[\S 1]{taylor_deg2} or \cite[\S 5.2]{notes_gee} 
for a nice exposition. 
Let $\mathfrak{m}$ be a maximal ideal of $\mathbb{T}_{S, \OO}^{\univ}$ generated  by $\varpi$ and all elements, which reduce modulo 
$\varpi$ to  $T_v- \tr \rhobar(\Frob_v), S_v \mathbf Nv- \det \rhobar(\Frob_v)$ for all $v\not\in S$. 
Let $M:=S_{\tau}(U, \OO)_{\mathfrak{m}}$ be the localization of $S_{\tau}(U, \OO)$ at $\mathfrak{m}$. We will see in the next paragraph that 
$M$ is non-zero. Hence the localization $\mathbb T_{\mathfrak m}$ is non-zero and from $\rho^{\mathrm{mod}}$ we obtain a continuous 
representation $\rho^{\mathrm{mod}}_{\mathfrak m}: G_{F, S}\rightarrow \GL_2(\mathbb T_{\mathfrak m})$ which reduces to $\rhobar$ modulo 
$\mathfrak m$. The representations obtained by specializing $\rho^{\mathrm{mod}}_{\mathfrak m}$ at the maximal ideals of 
$\mathbb T_{\mathfrak m}[1/p]$ have determinant $\psi \chi_{\cyc}$ and satisfy the local conditions at $v\in \Sigma$ by the local-global 
compatibility of Langlands correspondence. Since $\mathbb T_{\mathfrak m}$ is reduced and $\OO$-torsion free we deduce that 
$\rho^{\mathrm{mod}}_{\mathfrak m}$ satisfies the local conditions at $v\in \Sigma$ and has determinant $\psi\chi_{\cyc}$. The universal property of $\bar{R}_{F, S}^{\psi}$ gives us a natural map 
$\bar{R}_{F, S}^{\psi}\rightarrow \mathbb T_{\mathfrak m}$. This map is surjective since the images of $T_v$ and $S_v \mathbf Nv$ for
 $v\not \in S$  will be equal to the image of the trace and the determinant, respectively, of $\rho^{\univ}(\Frob_v)$, where $\rho^{\univ}$ is the 
 tautological deformation of $\rhobar$ to $\bar{R}_{F, S}^{\psi}$. Thus   $M$ is  naturally an $\bar{R}^{\psi}_{F, S}$-module. 
 
 A representation $\rho_x:G_{F, S}\rightarrow \GL_2(\kappa(x))$ corresponding to a maximal ideal $x\in \mSpec \bar{R}_{F, S}^{\psi}[1/p]$ 
 is modular by \cite{ki09} and \cite[Thm.6.3]{ht}, where the assumption $\rhobar|_{G_{F_v}}
 \not\sim ( \begin{smallmatrix} \chi\omega & \ast\\ 0 & \chi\end{smallmatrix} \bigr )$ for all $v\mid p$ is removed, if $p\ge 5$. It follows from 
 Jacquet--Langlands and the compatibility of local and global Langlands correspondences that $x$ lies in the support of $M$.
 In particular, $M$ is non-zero. We use \eqref{present5} to get a map $\bar{R}_{\Sigma}^{\square,\psi}\rightarrow \bar{R}^{\psi}_{F, S}$, which 
 makes $M$ into an $\bar{R}_{\Sigma}^{\square,\psi}$-module.
 It follows from Propositions \ref{abstract_prop} and \ref{finiteO}  that every irreducible component of 
 $\Spec \bar{R}_{\Sigma}^{\square,\psi}$ meets the support of $M$. Moreover, the argument with purity in the proof of \cite[Lem.4.4.3]{gee_kisin} shows that every $x\in \mSpec \bar{R}_{\Sigma}^{\square, \psi}[1/p]$ in the
 support of $M$ lies in the smooth locus of  $\Spec \bar{R}_{\Sigma}^{\square, \psi}[1/p]$.
 
 The patching argument carried out in \cite[\S 4.4]{gee_kisin} gives a finitely generated Cohen-Macaulay module $M_{\infty}$ of 
 $R_{\infty}:=\bar{R}_{\Sigma}^{\square, \psi}[[x_1, \ldots, x_g]]$, an $M_{\infty}$-regular sequence of elements $y_1, \ldots, y_h\in R_{\infty}$ such that 
 $\dim R_{\infty}= h+1$, a surjection of $\bar{R}_{\Sigma}^{\square, \psi}$-algebras
 \begin{equation}\label{patched_ring}
  R_{\infty}/(y_1, \ldots, y_h)\twoheadrightarrow \bar{R}_{F, S}^{\psi},
  \end{equation}
 and an isomorphism of $R_{\infty}$-modules $M_{\infty}/(y_1, \ldots, y_h)M_{\infty} \cong M$. Every irreducible component of $\Spec R_{\infty}$ will 
 contain a smooth point in the support of $M$, since we have already established this for $\Spec \bar{R}_{\Sigma}^{\square, \psi}$ and the irreducible 
 components of $\Spec R_{\infty}$ are of the form $V(\pp[[x_1, \ldots, x_g]])$, where $\pp$ is a minimal prime ideal of $\bar{R}_{\Sigma}^{\square, \psi}$.
 Since $R_{\infty}$ is reduced and $\dim M_{\infty}=h+1=\dim R_{\infty}$ we conclude that $R_{\infty}$ acts faithfully on $M_{\infty}$. It follows from Lemma \ref{com_alg}
 below applied with $A=R_{\infty}$, $M=M_{\infty}$, $x_1=p$, $x_i= y_{i-1}$ for $2\le i\le h+1$ that $\bar{R}_{F, S}^{\psi}[1/p]$ is equal to 
 $R_{\infty} /(y_1, \ldots, y_h)[1/p]$, which is reduced. 
 \end{proof}

\begin{lem}\label{com_alg} Let $A$  be a local noetherian ring. Let $M$ be a finitely generated Cohen-Macaulay faithful $A$-module. 
Let $x_1,\ldots, x_d$ be a system of parameters of $A$. Assume 
\begin{itemize}
\item[(1)]  $(M/(x_2,\ldots, x_d) M) [1/x_1]$ is a semi-simple $A[1/x_1]$-module;
\item[(2)] the localization $A_{\nn}$ is regular, for every maximal ideal $\nn$ of $A[1/x_1]$ in the support of $(M/(x_2,\ldots, x_d) M) [1/x_1]$.
\end{itemize}
Then $A/(x_2,\ldots, x_d)[1/x_1]$ is  reduced  and it acts on $(M/(x_2,\ldots, x_d) M) [1/x_1]$ faithfully.
\end{lem}
\begin{proof} For every $\nn$ in (2) $M_{\nn}$ is a finitely generated Cohen--Macaulay 
module over a regular ring $A_\nn$.
 Auslander--Buchsbaum implies that $M_{\nn}$ is free of finite rank over $A_\nn$. This yields 
that $(M/(x_2, \ldots, x_d)M)_\nn$ is finitely generated and free over  $(A/(x_2,\ldots, x_n))_\nn$. It follows from (1) that   $(A/(x_2,\ldots, x_n))_\nn$ is reduced. 
Since $A/(x_2,\ldots, x_d)[1/x_1]$ is artinian, we have an injection 
$$A/(x_2,\ldots, x_d)[1/x_1] \hookrightarrow \prod_{\nn} A/(x_2,\ldots, x_d)_{\nn}.$$
 So $A/(x_2, \ldots, x_d)[1/x_1]$ is reduced. Since $M$ is a faithful $A$-module, the support of 
$(M/(x_2,\ldots, x_d) M) [1/x_1]$ is equal to $\Spec A/(x_2, \ldots, x_d)[1/x_1]$. Since the ring is reduced this implies that the action is faithful.
\end{proof} 
\begin{rem} In this subsection we have followed the notation of \cite{gee_kisin}  and \cite{ki09}, so that  $R_{\infty}$ is by definition a formal power series ring over
$R_{\mathrm{loc}}$, which is a completed tensor product of local deformation rings. In particular, it is not clear a priori that \eqref{patched_ring} is an isomorphism. In the next subsection we will follow the notation of \cite{kisin_serre_2, kw2, blocksp2}, and $R_\infty$ will denote the projective limit of  the finite rings that we patch. In particular, the analog of \eqref{patched_ring}, see \eqref{mod_y}, is an isomorphism. 
\end{rem} 

\begin{thm} In addition to the hypotheses made in the beginning of the subsection, assume that the following hold: 
\begin{itemize}
\item[(1)] $\bar{R}^{\square, \psi}_v\neq 0$ for each $v\in \Sigma$;
\item[(2)] if $v\mid p$ then $\tau_v\cong \theta_{1,v}\oplus \theta_{2, v}$, for distinct characters $\theta_{1,v}, \theta_{2,v}: I_v\rightarrow L^{\times}$ with open kernel, 
which extend to $W_{F_v}$;
\item[(3)] for each $v\mid p$ either $\End_{G_{F_v}}(\rhobar)=k$ or $\rhobar|_{G_{F_v}}\sim \bigl ( \begin{smallmatrix} \delta_1 & 0 \\ 0 & \delta_2\end{smallmatrix} \bigr)$, such that $\delta_1 \delta_2^{-1}\neq 1, \omega^{\pm 1}$ and $\theta_{1,v}$ is not congruent modulo $\varpi$ to any of the four characters 
$\delta_1\omega^{-a_v}$, $\delta_1\omega^{-b_v}$, $\delta_2\omega^{-a_v}$, $\delta_2\omega^{-b_v}$.
\end{itemize}
Then $\bar{R}^{\psi}_{F, S}$ is reduced and is a finite free $\OO$-module of rank at least $1$. 
\end{thm}

\begin{proof} Theorems \ref{theorem-main}, \ref{gen_split} and \cite[\S 5.5]{shotton} show that the hypothesis of Proposition  \ref{abstract_prop} are satisfied in this case and hence $\bar{R}^{ \psi}_{F, S}$ is a finite free $\OO$-module of rank at least $1$. Hence $\bar{R}_{F, S}^{\psi}$ injects 
into $\bar{R}_{F, S}^{\psi}[1/p]$, which is reduced by Proposition \ref{reduced}.
\end{proof}

\subsection{Even primes} 

In this subsection we assume that $p=2$, so that $L$ is a finite extension of $\Q_2$. We will prove a similar result as in the case when $p$ is odd by replacing the patching argument of \cite{gee_kisin} with \cite{blocksp2}, which is based on \cite{kisin_serre_2} and \cite{kw2}. We assume throughout that
$\rhobar: G_{F, S}\rightarrow \GL_2(k)$ is modular and the image of $\rhobar$ is non-solvable. We also assume that the character $\psi: G_{F, S}\rightarrow 
\OO^{\times}$, which we fixed at the beginning of \S \ref{global_app} is totally even, which means that its restriction to  $G_{F_v}$ is trivial for all archimedean places $v$. 
 The argument 
is slightly different to the previous subsection: we will first prove the result, when the ramification of $\rhobar$ is minimal by appealing to the results of \cite{blocksp2}, and then 
deduce the general case from it. The following proposition will allow us  to verify that $\bar{R}^{\psi}_{F, S}$ is a finitely generated $\OO$-module after replacing $F$ by a finite extension. 

\begin{prop}\label{base_change_fg} Let $F'$ be a finite Galois extension of of $F$ and let $S'$ be the set of places of $F'$ above $S$. Assume that $\End_{G_{F'}}(\rhobar)=k$. 
Then the natural map $R^{\psi}_{F', S'}\rightarrow R^{\psi}_{F,S}$, induced by restricting deformations of $\rhobar$ to $G_{F'}$, makes $R^{\psi}_{F,S}$
into a finitely generated $R^{\psi}_{F', S'}$-module.
\end{prop}
\begin{proof} Let $\mm'$ be the maximal ideal of $R^{\psi}_{F', S'}$. It follows from Lemma 3.6 of \cite{kw_annals} that it is enough to show that 
the image of $G_{F, S}\rightarrow \GL_2(R^{\psi}_{F, S})\rightarrow \GL_2(R^{\psi}_{F,S}/\mm' R^{\psi}_{F,S})$ is finite. Since $G_{F', S'}$  is of finite index in $G_{F, S}$ and it gets mapped to the finite subgroup $\rhobar(G_{F', S'})$, we are done. 
\end{proof}

\begin{lem}\label{v_not_2} Let $v$ be a finite place not dividing $2$. Assume that $\rhobar|_{G_{F_v}}$ is unramified. Let $\bar{R}^{\square, \psi}_v$ 
(resp. $R^{\square, \psi, \mathrm{ur}}_v$) be 
the maximal reduced and $2$-torsion free quotient of $R_v^{\square, \psi}$ all of whose $\overline{L}$-points give rise to representations $\rho$ of $G_{F_v}$, such that $\rho|_{I_{F_v}}\sim \bigl (\begin{smallmatrix} 1 & \ast \\ 0 & 1 \end{smallmatrix}\bigr )$  (resp.  $\sim \bigl (\begin{smallmatrix} 1 & 0 \\ 0 & 1 \end{smallmatrix}\bigr )$).  Let $V(\pp)$ be an irreducible component 
of $\Spec \bar{R}^{\square, \psi}_v$. 
Then one of the following holds:
\begin{itemize}
\item[(i)] $\bar{R}^{\square, \psi}_v/\pp= R^{\square, \psi, \mathrm{ur}}_v$ is formally smooth over $\OO$ of relative dimension $3$;
\item[(ii)] the eigenvalues of $\rhobar(\Frob_v)$ are equal and $\bar{R}^{\square, \psi}_v/\pp= R^{\square, \psi}_v(\gamma)$ is of relative dimension $3$ over $\OO$, where
 $\gamma: G_{F_v}\rightarrow \OO^{\times}$ is an unramified character such that
$\gamma^2= \psi$ and $R^{\square, \psi}_v(\gamma)$ is  
the maximal reduced and $2$-torsion free quotient of $R_v^{\square, \psi}$ all of whose $\overline{L}$-points give rise to representations $\rho$ of $G_{F_v}$, such that $\rho\sim \bigl (\begin{smallmatrix} \gamma \chi_{\cyc} & \ast \\ 0 & \gamma \end{smallmatrix}\bigr )$. 
\end{itemize}
Moreover, if the eigenvalues of $\rhobar(\Frob_v)$ are distinct then $\bar{R}^{\square, \psi}_v=R^{\square, \psi, \mathrm{ur}}_v$.
\end{lem}
\begin{proof} Let $x\in \mSpec \bar{R}^{\square, \psi}_v[1/2]$ and let $\rho: G_{F_v}\rightarrow \GL_2(\kappa(x))$ be the corresponding Galois representation. 
If the image of $I_{F_v}$ is trivial then $x$ lies in $\Spec R^{\square, \psi, \mathrm{ur}}_v$. Assume that the image of $I_{F_v}$ is non-trivial. Since $\rhobar$
is unramified at $v$, the action of $G_{F_v}$ on any deformation of $\rhobar$ factors through the quotient isomorphic to 
$\mathbb{Z}_2\rtimes \widehat{\mathbb{Z}}$, 
where $\mathbb{Z}_2$ is identified with the maximal pro-$2$ quotient of the tame inertia, and if we further quotient out by it, then $\widehat{\mathbb{Z}}$ is identified with the 
absolute Galois group of the residue field of $F_v$. If we let $\sigma$ be a topological generator of $\mathbb{Z}_2$ and $\phi$ be the topological generator 
of $\widehat{\mathbb{Z}}$ which maps to $\Frob_v$ in $G_{F_v}/I_{F_v}$, then $\phi \sigma \phi^{-1}= \sigma^{q_v}$, where $q_v$ is the number of  elements 
in the residue field of $F_v$. Since $\tau_v$ is trivial and $\rho$ is ramified after conjugation we may assume that $\rho(\sigma)=
\bigl (\begin{smallmatrix} 1 & 1\\ 0  & 1 \end{smallmatrix}\bigr)$. The relation $\rho(\phi)\rho(\sigma)=\rho(\sigma)^{q_v} \rho(\phi)$ can be written as
$$ \begin{pmatrix} a &  b\\ c & d\end{pmatrix} \begin{pmatrix} 1 & 1 \\ 0 & 1 \end{pmatrix}= \begin{pmatrix} 1 & q_v \\ 0 & 1 \end{pmatrix}\begin{pmatrix} a &  b\\ c & d\end{pmatrix},$$
and is equivalent to  $c=0$ and $ a=q_vd$. Thus $\rho \sim  \bigl (\begin{smallmatrix} \gamma \chi_{\cyc} & \ast \\ 0 & \gamma \end{smallmatrix}\bigr )$, 
where $\gamma$ is an unramified character with $\gamma(\Frob_v)=d$. Since $\det \rho=\psi \chi_{\cyc}$ we conclude that $\gamma^2=\psi$. Hence, 
$x$ lies in $\Spec R^{\square, \psi}_v(\gamma)$.

The characteristic polynomial of $\rhobar(\Frob_v)$ is congruent modulo $\varpi$ to the characteristic polynomial of 
$\rho(\phi)$, which is equal to $x^2 - d(1+q_v) x + d^2 q_v$. Since $q_v$ is odd we deduce that $\rhobar(\Frob_v)$ has equal eigenvalues. Hence if the 
eigenvalues of $\rhobar(\Frob_v)$ are distinct we conclude that  $\bar{R}^{\square, \psi}_v=R^{\square, \psi, \mathrm{ur}}_v$.

It is shown in \cite[Prop.2.5.3]{kisin_serre_2} that $R^{\square, \psi, \mathrm{ur}}_v$ is formally smooth over $\OO$ of relative dimension $3$. 
Hence, if the eigenvalues of $\rhobar(\Frob_v)$ are distinct then $\pp=0$ and we are in part (i) of the Lemma. Let us assume that the eigenvalues of 
$\rhobar(\Frob_v)$ are equal. 
Since $\chi_{\cyc}$ is trivial modulo $2$, it follows from \cite[Prop.2.5.2]{kisin_serre_2} that 
$R^{\square, \psi}_v(\gamma)$ is an integral domain of relative dimension $3$ over $\OO$. Let $\pp_{\mathrm{ur}}$ and $\pp_{\gamma}$ be the kernels
of $\bar{R}^{\square, \psi}_v\twoheadrightarrow R^{\square, \psi, \mathrm{ur}}_v$ and 
$\bar{R}^{\square, \psi}_v\twoheadrightarrow R^{\square, \psi}_v(\gamma)$, respectively. 
There are exactly two unramified characters $\gamma$ 
such that $\gamma^2=\psi$, which we denote by $\gamma_1$ and $\gamma_2$. We have shown that $\mSpec \bar{R}^{\square, \psi}_v[1/2]$ is contained 
in the union of  $V(\pp_{\mathrm{ur}})$, $V(\pp_{\gamma_1})$ and  $V(\pp_{\gamma_2})$. Since $\bar{R}^{\square, \psi}_v$ is $\OO$-torsion free and 
$\bar{R}^{\square, \psi}_v[1/2]$ is Jacobson, we deduce that $\bar{R}^{\square, \psi}_v$ is equidimensional of relative dimension $3$ over $\OO$, 
and either (i) or (ii) of the lemma hold.
\end{proof}

Let $\Sigma$ be a subset of $S$ containing all the places of $F$ above $2$ and all infinite places. 

\begin{prop}\label{components_meet} Assume that $\rhobar$ is unramified outside the places dividing $2$ and $\tau_v$ is trivial for all  $v\in \Sigma$, such that  $v\nmid 2 \infty$, and
for all $v\mid 2$ we have $\rhobar|_{G_{F_v}}\not\sim \bigl (\begin{smallmatrix} \chi & \ast \\ 0 & \chi \end{smallmatrix}\bigr )$ for any character 
$\chi: G_{F_v}\rightarrow k^{\times}$.  
Assume that the eigenvalues of $\rhobar(\Frob_v)$ are distinct for all $v\in S\setminus \Sigma$. Let 
$$\bar{R}_S^{\square,\psi}:=(\wtimes_{v\in \Sigma} \bar{R}^{\square, \psi}_v)\wtimes_{\OO} (\wtimes_{v\in S\setminus \Sigma} R^{\square, \psi}_v).$$
Then every irreducible component of $\Spec \bar{R}_S^{\square, \psi}$ meets $\mSpec \bar{R}^{\psi}_{F, S}[1/2]$.
\end{prop}
\begin{proof} Since $\bar{R}_v^{\square, \psi}$ and $R^{\square, \psi}_v$ are equidimensional for all $v\in S$, either by Lemma \ref{v_not_2}, \cite[Prop.2.5.6, 2.5.4]{kisin_serre_2} or  \cite[Thm.3.3.4]{ki08}, Lemma \ref{equidim} implies that $\bar{R}_S^{\square,\psi}$ and 
$R^{\square, \psi}_{S\setminus \Sigma}:=\wtimes_{v\in S\setminus \Sigma} R^{\square, \psi}_v$ are  equidimensional. Let $\qq$ be a minimal prime of $\bar{R}_S^{\square,\psi}$ and for each $v\in \Sigma$ let  $\pp_v$ be the kernel of 
the natural map $\bar{R}_v^{\square, \psi} \rightarrow \bar{R}_S^{\square,\psi}/\qq$.
Then $\qq$ is also a minimal prime of $(\wtimes_{v\in \Sigma} \bar{R}^{\square, \psi}_v /\pp_v) \wtimes R^{\square, \psi}_{S\setminus \Sigma}$. Lemma \ref{equidim} implies that for all $v\in \Sigma$,  $\pp_v$ is a minimal prime of $\bar{R}^{\square, \psi}_v$. Let $\Sigma_1$ (resp. $\Sigma_2$) be the subset of $\Sigma$ consisting of finite places not above $2$ such that part (i) (resp. part (ii)) of Lemma \ref{v_not_2} applied with $\pp=\pp_v$ holds.
Then it is enough to prove that every irreducible component of the ring 
$$(\wtimes_{v\in \Sigma_1} R^{\square, \psi, \mathrm{ur}}_v)\wtimes_{\OO}(\wtimes_{v\in \Sigma_2}  R^{\square, \psi}_v(\gamma_v)) \wtimes_{\OO}
(\wtimes_{v\mid 2\infty} \bar{R}_v^{\square, \psi})\wtimes_{\OO} R^{\square, \psi}_{S\setminus \Sigma}$$
contains a point of $\mSpec \bar{R}_{S, F}^{\psi}[1/2]$. Since $R^{\square, \psi, \mathrm{ur}}_v$ is formally smooth over $\OO$, it is enough to prove that 
every irreducible component of 
\begin{equation}\label{ring}
(\wtimes_{v\in \Sigma_2}  R^{\square, \psi}_v(\gamma_v)) \wtimes_{\OO}
(\wtimes_{v\mid 2\infty} \bar{R}_v^{\square, \psi})\wtimes_{\OO} R^{\square, \psi}_{S\setminus \Sigma}
\end{equation}
contains a point of $\mSpec \bar{R}_{S\setminus \Sigma_1, F}^{\psi}[1/2]$.  So without loss of generality we may assume that $\Sigma_1=\emptyset$.

We note that the same argument allows us to enlarge $S$. If  we can prove that every irreducible component of $\Spec \bar{R}_{S\cup\{v_1\}}^{\square, \psi}$ meets $\mSpec \bar{R}^{\psi}_{F, S\cup \{v_1\}}[1/2]$, then since $\Spec R^{\square, \psi,  \mathrm{ur}}_{v_1}$ is an irreducible component of 
$\Spec R^{\square, \psi}_{v_1}$ by \cite[Prop.2.5.4]{kisin_serre_2}, the claim for $S\cup \{v_1\}$ implies the claim for $S$. Thus we may assume that 
$S\setminus \Sigma_2$ contains a finite place not dividing $2$. It follows from Proposition \ref{abstract_prop} applied with $\Sigma=\Sigma_2\cup \{ v\mid 2\infty\}$ and 
$\bar{R}_v^{\square, \psi}=R^{\square, \psi}_v(\gamma_v)$ for $v\in \Sigma_2$ 
that it is enough to prove that $\bar{R}^{\psi}_{F, S}$ is a finitely generated $\OO$-module.

Let $F'$ be a finite totally real solvable extension of $F$ in which $2$ splits completely, let $\Sigma_2'$ and $S'$ be the set of places of $F'$ above the places of $\Sigma_2$ and $S$ respectively. For each $w\in \Sigma_2'$ let $\bar{R}^{\square, \psi}_w= R^{\square, \psi}_w(\gamma_w)$, where $\gamma_w$ is the restriction of 
$\gamma_v$ to $G_{F'_w}$, where $v\in \Sigma_2$ such that $w\mid v$. Since the image of $\rhobar$ is non-solvable, the restriction of $\rhobar$ to $G_{F'}$ remains
irreducible. The composition $R^{\psi}_{F', S'}\rightarrow  R^{\psi}_{F, S}\twoheadrightarrow \bar{R}^{\psi}_{F, S}$ factors through 
$\bar{R}^{\psi}_{F', S'}\rightarrow \bar{R}^{\psi}_{F, S}$. It follows from Proposition \ref{base_change_fg} that $\bar{R}^{\psi}_{F, S}$ is a finitely generated $\bar{R}^{\psi}_{F', S'}$-module.
If we can show that $\bar{R}^{\psi}_{F', S'}$ is a finitely generated $\OO$-module, then $\bar{R}^{\psi}_{F, S}$ is also a finitely generated $\OO$-module.

By standard base change and level raising arguments, as used in the proof of \cite[Thm.3.30]{blocksp2}, we can get ourselves in the setup of \cite[\S 3B]{blocksp2}: 
we may find $F'$ as above such that both $[F':\Q]$ and $|\Sigma_2'|$
are even, and $\rhobar|_{G_{F'}}$ is congruent to a Galois representation associated to a Hilbert eigenform $g$ over $F'$ of parallel weight $2$, which  is special of 
conductor $1$ at $v\in \Sigma_2$ and is unramified otherwise. Let $D$ be the quaternion algebra with centre $F'$ ramified precisely at $\Sigma_2'$ and all the infinite 
places. Then $g$ corresponds to an automorphic form on $D$ by the Jacquet--Langlands correspondence. It is explained at the end of the proof of 
\cite[Thm.3.30]{blocksp2} that this implies that the hypothesis 
$S_{\sigma, \psi}(U, \OO)_\mm\neq 0$ made in \cite[Cor.3.27]{blocksp2} is satisfied.  Then the ring in \eqref{ring} 
(with $F$, $\Sigma_2$ and $S$ replaced by $F'$, $\Sigma_2'$ and $S'$ respectively) is the quotient  of the ring denoted by $R^{\psi, \square}_S(\sigma)$ in \cite{blocksp2} (with $S$ replaced by $S'$). Corollary 3.27 of \cite{blocksp2} and part (d) of \cite[Prop.3.17]{blocksp2}
imply that $\bar{R}^{\psi}_{F', S'}$ is a finitely generated $\OO$-module.
\end{proof} 

\begin{prop}\label{almost_done} Assume that $[F:\Q]$ is even, $\rhobar$ is unramified outside the places dividing $2$ and for all $v\mid 2$ we have $\rhobar|_{G_{F_v}}\not\sim \bigl (\begin{smallmatrix} \chi & \ast \\ 0 & \chi \end{smallmatrix}\bigr )$ for any character 
$\chi: G_{F_v}\rightarrow k^{\times}$. Assume that the eigenvalues of $\rhobar(\Frob_v)$ are distinct for all $v\in S\setminus \Sigma$.  Assume that $\rhobar $ is congruent to a Galois representation associated to a Hilbert eigenform $f$ over $F$, which is unramified outside $S$ and either special of conductor $1$ or unramified 
principal series at $v\in \Sigma$, $v\nmid 2\infty$.

 Let $\bar{R}^{\square, \psi}_v$ 
be defined as in \S \ref{global_pst} with $\tau_v$  trivial for all  $v\in \Sigma$, such that $v\nmid 2 \infty$. If $\bar{R}_v^{\square, \psi}\neq 0$ for all $v\in \Sigma$ then $\bar{R}^{\psi}_{F, S}$ is a finitely generated $\OO$-module and $\bar{R}^{\psi}_{F, S}[1/2]$ is reduced. 
\end{prop}
\begin{proof} We will deduce the statement by modifying the patching argument in \cite{blocksp2}. We need to carry out the modification, because in \cite{blocksp2} at places $v\nmid 2\infty$ we did not work with $\bar{R}_v^{\square, \psi}$ but with certain quotients of it, corresponding to either
 unramified or genuinely semi-stable components, see \cite[\S 3B1]{blocksp2}. This ensured that the patched module had the same constant multiplicity on all 
 the irreducible components, which did not depend on the $p$-adic Hodge type at place above $p$, see \cite[Lem.3.10]{blocksp2}. It would have been better 
 to avoid this problem by  using types at $v\nmid 2\infty$, as is done in \cite{gee_kisin}. For the proof of the Proposition the question of multiplicities is irrelevant 
 once we use the results of \cite{blocksp2} as an input.

Let $D$ be the quaternion algebra with the centre $F$ ramified at all infinite places and unramified at all finite places. If $S'$ is a set of places of $F$ containing $S$, 
then the natural map $\bar{R}^{\psi}_{F, S'}\rightarrow \bar{R}^{\psi}_{F, S}$ is surjective, and so if $\bar{R}^{\psi}_{F, S'}$ is a finitely generated $\OO$-module such that 
$\bar{R}^{\psi}_{F, S'}[1/2]$ is reduced then the same holds for $\bar{R}^{\psi}_{F, S}$. Hence, we may assume that there is $v_1\in S\setminus \Sigma$ such that 
$\rhobar(\Frob_{v_1})$ has distinct eigenvalues and $v_1$ satisfies the conditions of \cite[Lem.3.2]{blocksp2}. 

We fix a maximal order $\OO_D$ in $D$ and identify
$D\otimes_F \A^\infty_F$ with $M_2(\A^\infty_F)$ so that $(\OO_D)_v$ gets identified with $M_2(\OO_{F_v})$ for all finite $v$. 
 Let $U=\prod_v U_v$ be a compact open subgroup of $\GL_2(\A^\infty_F)$ such  that 
 $U_v=\{g\in \GL_2(\OO_{F_v}): g \equiv \bigl( \begin{smallmatrix} \ast & \ast \\ 0 & \ast\end{smallmatrix}\bigr)\pmod{\varpi_v}\} $ for all $v\in \Sigma\setminus \{w\mid 2\infty\}$, 
  $U_v=\{g\in \GL_2(\OO_{F_v}): g \equiv \bigl( \begin{smallmatrix} 1& \ast \\ 0 & 1\end{smallmatrix}\bigr)\pmod{\varpi_v}\} $ for all $v\in S\setminus \Sigma$, 
  and $U_v=\GL_2(\OO_{F_v})$ if $v\not \in S$ or $v\mid 2$.
 
 For every $v\mid 2$ we let 
 $\sigma(\wt_v, \tau_v)= \Sym^{b_v-a_v-1} L^2 \otimes \det^{a_v} \otimes \sigma(\tau_v)$ and fix a $U_v$-invariant $\OO$-lattice 
 $\mathcal{L}_v$ in $\sigma(\wt_v, \tau_v)$. 
 Then $\mathcal{L}:=\otimes_{v\mid 2} \mathcal {L}_v$ is a free $\OO$-module of finite rank with a continuous action of $U$, which we denote by $\sigma$, 
 where each factor $U_v$ acts trivially if $v\nmid 2$ and naturally on the
 factor $\mathcal{L}_v$ if $v\mid 2$. The central character of $\sigma$ is equal to the restriction of $\psi$ to $U\cap (\mathbb{A}_F^{\infty})^{\times}$.  
Let $S_{\sigma, \psi}(U,\OO)$ denote the set of continuous functions 
$$f : D^{\times} \backslash( D \otimes_F  \AfF )^{\times} \rightarrow\mathcal L$$
such that for $g\in(D\otimes_F \AfF)^{\times}$ we have $f(gu)=\sigma (u)^{-1} f(g)$ for $u\in U$  and $f(gz)=
\psi^{-1}(z)f(g)$ for $z\in (\AfF )^{\times}$. As in the proof of Proposition \ref{reduced} the polynomial ring 
$\mathbb{T}_{S, \OO}^{\univ}=\OO[T_v, S_v]_{v\not\in S}$ acts on $S_{\sigma, \psi}(U,\OO)$ by Hecke operators and we 
let $\mathbb{T}$ be the image of $\mathbb{T}_{S, \OO}^{\univ}$ in $\End_{\OO}(S_{\sigma, \psi}(U, \OO))$. It follows from \cite[Lemma 1.3 (1)]{taylor_deg2}  
that the action of $\mathbb T[1/2]$ on $S_{\sigma, \psi}(U, \OO)\otimes_{\OO} L$ is semi-simple. 

Let $\mm$ be the maximal ideal of $\mathbb{T}_{S, \OO}^{\univ}$ corresponding to $\rhobar$. Since $D$ is split at all finite places there is an automorphic form 
on $D$ corresponding to $f$ via the Jacquet--Langlands correspondence.  The assumption that $\rhobar(\Frob_v)$ has distinct eigenvalues at all 
$v\in S\setminus \Sigma$ implies that  $f$ is tamely ramified principal series at $v\in S\setminus \Sigma$. 
Hence, $f$ gives rise to an eigenvector for $\mathbb{T}_{S, \OO}^{\univ}$ in $S_{\sigma', \psi'}(U, \OO)_{\mm} [1/2]$, for some $\sigma'$ and $\psi'$. 
This implies that $S_{\sigma', \psi'}(U, \OO)_{\mm}$, and hence $M:=S_{\sigma, \psi}(U, \OO)_{\mm}$,  are non-zero, see
Lemma 3.29 and the fourth paragraph in the proof of Theorem 3.30 in \cite{blocksp2}. As in the proof of Proposition \ref{reduced} we have a surjection 
$\bar{R}^{\psi}_{F, S}\twoheadrightarrow \mathbb{T}_{\mm}$, which makes $M$ into an $\bar{R}^{\psi}_{F, S}$-module.

Now we will redo the patching argument carried out in \cite{blocksp2}. Note that we are almost in the setting of \cite[\S 3B]{blocksp2}, except there $D$ was ramified at all $v\in \Sigma$, $v\nmid 2\infty$ and $U_v=(\OO_D)_v^{\times}$
for all such $v$, so that the ring denoted by $R^{\psi}_{F, S}(\sigma)$ is the quotient of our ring  $\bar{R}^{\psi}_{F, S}$ corresponding to components
described in (ii) of Lemma \ref{v_not_2}. The patching argument is based on \cite{kw2}, \cite{kisin_serre_2},  the details are sketched in the proof of 
\cite[Prop.3.22]{blocksp2}. The output is three complete local noetherian $\OO$-algebras $R'_{\infty}$, $R^{\inv}_\infty$, $R_{\infty}$ satisfying properties (P1)--(P5) of 
\cite[\S3C]{blocksp2}, where in (P1) one has to replace the ring denoted by $R^{\psi, \square}_S(\sigma)$ in \cite{blocksp2} with $\bar{R}^{\square, \psi}_S$ defined in Proposition \ref{components_meet}, a finitely generated Cohen--Macaulay $R_{\infty}$-module $M_{\infty}$ and elements $y_1, \ldots, y_{h+j}$ in the maximal ideal of $R_{\infty}$, such that 
$\varpi, y_1, \ldots, y_{h+j}$ is a regular sequence  for $M_{\infty}$, 
\begin{equation}\label{mod_y}
R_{\infty}/(y_1, \ldots, y_{h+j})\cong \bar{R}^{\psi}_{F, S}, \quad M_{\infty}/ (y_1, \ldots, y_{h+j}) M_{\infty}\cong M,
\end{equation}
and the action of $\bar{R}^{\psi}_{F, S}$ on $M$ given by these isomorphisms coincides with the action defined above.

 If we can show that $R_{\infty}$ acts faithfully on 
$M_{\infty}$, then  $\varpi, y_1, \ldots, y_{h+j}$ is a system of parameters for $R_{\infty}$ and the above isomorphism implies that $\bar{R}^{\psi}_{F, S}$ is a finitely 
generated $\OO$-module. Moreover, if we additionally can show that the support of $M[1/2]$ lies in the regular locus of $\Spec R_{\infty}[1/2]$ then Lemma \ref{com_alg}
implies that $\bar{R}^{\psi}_{F, S}[1/2]$ is reduced. 

This is proved by repeating the arguments of \cite{blocksp2} with $R^{\psi, \square}_S(\sigma)$, $R'_{\infty}(\sigma)$, $R^{\inv}_{\infty}(\sigma)$, $R_{\infty}(\sigma)$
and $M_{\infty}(\sigma)$ replaced with $\bar{R}^{\square, \psi}_S$, $R'_{\infty}$, $R^{\inv}_\infty$, $R_{\infty}$ and $M_{\infty}$, respectively: the proofs of Lemmas 3.3 to 3.8 
in \cite{blocksp2} go through unchanged. This last lemma 
says that if $\nn$ is a maximal ideal of $R_{\infty}[1/2]$ in the support of $M$ then the localization is a regular ring. Proposition \ref{components_meet} implies that 
every irreducible component  of $\Spec \bar{R}^{\square, \psi}_S[[x_1, \ldots, x_g]]$ meets $\Spec \bar{R}^{\psi}_{F, S}[1/2]$ via the map:
$$\bar{R}^{\square, \psi}_S[[x_1, \ldots, x_g]]\twoheadrightarrow R'_{\infty}\twoheadrightarrow R_{\infty}\twoheadrightarrow \bar{R}^{\psi}_{F, S},$$
where the first two arrows are given by patching in (P1) and the last is \eqref{mod_y}. This means that the assumption analogous to the assumption of 
\cite[Lem.3.14]{blocksp2} is satisfied, and since the proof of parts (i) to (iv) of that Lemma use only Lemmas 3.3 to 3.8 of \cite{blocksp2}, the conclusion continues to hold in our setting. In particular, $R_{\infty}$ is $\OO$-torsion free, reduced, equidimensional and the support of $M_{\infty}$ meets every irreducible component of $\Spec R_{\infty}$. Moreover, 
the intersection of the support of $M_\infty$ with each irreducible component of $R_\infty$ contains a regular point of that component.
Since the support of $M_\infty$ is a union of irreducible components of $\Spec R_{\infty}$ by \cite[Lem.3.5]{blocksp2}, we conclude that $R_{\infty}$ acts faithfully on $M_{\infty}$.
\end{proof}

\begin{lem}\label{reduced_fibre} Let $F'$ be a finite Galois extension of $F$ and let $S'$ be the set of places of $F'$ above the places in $S$. Assume that 
$\End_{G_{F', S'}}(\rhobar)=k$,  let $\nn$ be a maximal ideal of $R^{\psi}_{F, S}[1/2]$ and let $\nn'$ be the image of $\nn$ in 
$\Spec R^{\psi}_{F', S'}$. Then the ring $\kappa(\nn')\otimes_{R^{\psi}_{F', S'}} R^{\psi}_{F, S}$
is zero dimensional and reduced. 
\end{lem}
\begin{proof} Proposition \ref{base_change_fg} implies that  $\nn'$ is a maximal ideal 
of $R^{\psi}_{F', S'}[1/2]$ and the ring $\kappa(\nn')\otimes_{R^{\psi}_{F', S'}} R^{\psi}_{F, S}$ is zero dimensional. Let $\rho_{\nn}$ be the Galois representation 
corresponding to $\nn$. The $\nn$-adic completion of $R^{\psi}_{F, S}[1/2]$, which we denote by $A$, can be identified with the universal deformation 
ring of $\rho_{\nn}$ to local artinian $\kappa(\nn)$-algebras with determinant $\psi\chi_{\cyc}$. The quotient $\kappa(\nn')\otimes_{R^{\psi}_{F', S'}} A$ parameterizes deformations of $\rho_{\nn}$, which become trivial deformations after restriction to $G_{F', S'}$.  The tangent space of this closed subfunctor 
is the $\kappa(\nn)$-vector space dual of  the kernel of the natural map 
$$H^1(G_{F, S}, \ad^0\rho_{\nn})\rightarrow H^1(G_{F', S'}, \ad^0 \rho_{\nn}),$$
which can be calculated as $H^1(\Gal(F'/F), (\ad^0 \rho_{\nn})^{G_{F', S'}})$, see the proof of Proposition \ref{reduced}. This group vanishes as
$\Gal(F'/F)$ is a finite group and we work over a field of characteristic $0$. Hence, $\kappa(\nn')\otimes_{R^{\psi}_{F', S'}} A=\kappa(\nn)$, and the 
same holds for any maximal ideal of $R^{\psi}_{F, S}[1/2]$ above $\nn'$. Hence, the localization of $\kappa(\nn')\otimes_{R^{\psi}_{F', S'}} R^{\psi}_{F, S}$
at any maximal ideal is a field, and so the algebra is reduced.
\end{proof}

\begin{prop}\label{reduced_p2} Assume that for all $v\mid 2$ we have $\rhobar|_{G_{F_v}}\not\sim \bigl (\begin{smallmatrix} \chi & \ast \\ 0 & \chi \end{smallmatrix}\bigr )$ 
for any character $\chi: G_{F_v}\rightarrow k^{\times}$. Assume that $\Sigma\neq S$.  If $\bar{R}^{\square, \psi}_v\neq 0$ for all $v\in \Sigma$ then 
$\bar{R}^{\psi}_{F, S}$ is a finitely generated $\OO$-module of rank at least $1$  and $\bar{R}^{\psi}_{F, S}[1/2]$ is reduced.
\end{prop}
\begin{proof} Using Lemma 2.2 of \cite{taylor_ico2} we may find a finite totally real soluble Galois extension $F'$ of $F$ of even degree such that $2$ splits completely in $F'$ and $\rhobar|_{G_{F'}}$ is unramified outside $2$. After enlarging $F'$ again we may assume that  if $v\in S$, $v\nmid 2\infty$ and  $w\mid v$ then 
the image of $I_{F'_w}$ in $\GL_2(R^{\square, \psi}_v)$ is either trivial or is isomorphic to $\Z_p$. This implies that if $f$ is a Hilbert eigenform over $F$, such that 
$\rhobar\cong \rho_f \mod{\varpi}$ then the base change of $f$ to $F'$ will be unramified outside $S'$, the set of places of $F'$ above $S$, 
and either unramified or special of conductor $1$ at $v\in S'$, $v\nmid 2\infty$. 
 Let $\bar{R}^{\square, \psi}_w$ be the quotient of
$R^{\square, \psi}_w$  corresponding to the trivial inertial type. Since $R^{\square, \psi}_v$ is reduced and $\OO$-torsion free by
\cite[Thm.2.5]{shotton2} or \cite[Cor.8.3]{helm}, the natural map $R^{\square, \psi}_w\rightarrow R^{\square, \psi}_v$, induced by restriction of deformations to $G_{F'_w}$,
 factors through $\bar{R}^{\square, \psi}_w\rightarrow R^{\square, \psi}_v$. This implies that if we let $\bar{R}^{\psi}_{F', S'}$ be the ring defined in 
 \S\ref{global_pst} by taking $\Sigma=S'$ and $\tau_w$ trivial for all finite $w\in S'$ not dividing $2$ then the natural map 
 $R^{\psi}_{F', S'}\rightarrow R^{\psi}_{F, S}\twoheadrightarrow \bar{R}^{\psi}_{F, S}$ factors through $\bar{R}^{\psi}_{F', S'}\rightarrow \bar{R}^{\psi}_{F, S}$. 
 Propositions \ref{base_change_fg} and  \ref{almost_done} imply that $\bar{R}^{\psi}_{F, S}$ is a finitely generated $\OO$-module. Proposition \ref{abstract_prop} implies that its rank is at least $1$.
 
Let $A$ be the image of $\bar{R}^{\psi}_{F', S'}[1/2]$ in $\bar{R}^{\psi}_{F, S}[1/2]$.  It follows from Proposition \ref{almost_done} that $A$ is a product of fields. Thus it is enough to show that $\kappa(\nn)\otimes_A\bar{R}^{\psi}_{F, S}[1/2]$ is reduced for all maximal ideals $\nn$ of $A$. This algebra is a quotient 
 of $\kappa(\nn)\otimes_{R^{\psi}_{F', S'}} R^{\psi}_{F, S}$, which is a product of fields by Lemma \ref{reduced_fibre}. Hence, $\bar{R}^{\psi}_{F, S}[1/2]$ is reduced.
 \end{proof}
\begin{rem} If $\Sigma=S$ then we can still conclude that $\bar{R}^{\psi}_{F, S}$ is finitely generated $\OO$-module and $\bar{R}^{\psi}_{F, S}[1/2]$ is reduced, 
but we cannot rule out that $\bar{R}^{\psi}_{F, S}[1/2]=0$.
\end{rem}

\begin{thm} In addition to the hypotheses made in the beginning of the subsection, assume that $ \Sigma\neq S$ and that the following hold: 
\begin{itemize}
\item[(1)] $\bar{R}^{\square, \psi}_v\neq 0$ for each $v\in \Sigma$;
\item[(2)] if $v\mid 2$ then $\tau_v\cong \theta_{1,v}\oplus \theta_{2, v}$, for distinct characters $\theta_{1,v}, \theta_{2,v}: I_v\rightarrow L^{\times}$ with open kernel, which extend to $W_{F_v}$;
\item[(3)] for each $v\mid 2$,  $\End_{G_{F_v}}(\rhobar)=k$.
\end{itemize}
Then $\bar{R}^{\psi}_{F, S}$ is reduced and is a finite free $\OO$-module of rank at least $1$. 
\end{thm}

\begin{proof} Theorem \ref{theorem-main}, Theorem \ref{thm:CM},  proved by Jack Shotton in the appendix below, and Proposition \ref{reduced_p2} show that the hypothesis of Proposition  \ref{abstract_prop} are satisfied in this case and hence $\bar{R}^{\square, \psi}_{F, S}$ is a finite free $\OO$-module of rank at least $1$. Hence $\bar{R}_{F, S}^{\psi}$ injects 
into $\bar{R}_{F, S}^{\psi}[1/2]$, which is reduced by Proposition \ref{reduced_p2}.
\end{proof}

\appendix \normalsize

\section{Equidimensionality of completed tensor products}\label{equip}

Let $\mathcal C(\OO)$ be the category of local noetherian $\OO$-algebras with residue field $k$ which are $\OO$-torsion free. If $A, B\in \mathcal C(\OO)$ 
then the completed tensor product $A\wtimes_{\OO} B$ also lies in $\mathcal C(\OO)$. The purpose of the appendix is to prove the following result, the proof 
of which we were not able to find in the literature. 

\begin{lem}\label{equidim} If $A$ and $B$ in $\mathcal C(\OO)$ are both equidimensional then  $A\wtimes_{\OO} B$ is equidimensional of
dimension $\dim A+\dim B-1$.
\end{lem}  

Recall that a ring of finite Krull dimension is called \textit{equidimensional} if $\dim A=\dim A/\mathfrak p$ for all minimal primes $\pp$ of $A$. Recall that $A\in \mathcal C(\OO)$ is called \textit{geometrically integral} if  the algebra $A\otimes_{\OO} \OO'$ is an integral domain for all finite extensions $L'/L$, where $\OO'$ denotes the ring of integers in $L'$. 
We note that even if both $A$ and $B$ are integral domains the completed tensor product need not be an integral domain. For example, if $\OO=\Zp$, 
$A=B=\Zp[x]/(x^2-p)$ then $A\wtimes_{\OO} B\cong \Zp[x, y]/( x^2-p, y^2-p)$ and $(x-y)(x+y)=0$ in $A\wtimes_{\OO} B$. However, if both $A$ and $B$ are geometrically integral then $A \wtimes_{\OO} B$ is also geometrically integral by \cite[Lemma 3.3 (3)]{blght}. Moreover, Lemma \ref{equidim} holds if $A/\pp$ and $B/\qq$ are geometrically integral 
for all   the minimal primes $\pp$ of $A$ and all the minimal primes $\qq$ of $B$ by \cite[Lemma 3.3 (2), (5)]{blght}. We will prove the lemma by reducing to this case. 

\begin{lem}\label{base_equidim} If $A\in \mathcal C(\OO)$ then the minimal primes of $A':=A\otimes_{\OO} \OO'$ are precisely the prime ideals
lying above the minimal primes of $A$. Moreover, $A$ is equidimensional if and only if $A'$ is equidimensional.
\end{lem}
\begin{proof} The algebra $A'$ is finite and flat over $A$. In particular, $\dim A'=\dim A$ and $\dim A/\pp=\dim A'/\mathfrak P$ for all  
primes $\mathfrak P$ of $A'$ with $\pp= A\cap \mathfrak P$.

If $\pp$ is a minimal prime of $A$ then it follows from \cite[Theorem 9.3 (i), (ii)]{mat} that there is a prime $\mathfrak P$ of $A'$ lying above 
$\pp$  and any such $\mathfrak P$ is minimal. If $A'$ is equidimensional then $\dim A/\pp =\dim A'/\mathfrak P= \dim A'=\dim A$, 
 and thus $A$ is equidimensional.

If $\mathfrak P$ is a minimal prime of $A'$ then Going down theorem for flat extensions implies that 
$\mathfrak \pp:=\mathfrak P\cap A$ is a minimal prime of $A$. If $A$ is equidimensional 
then $\dim A'=\dim A= \dim A/\pp=\dim A'/\mathfrak P$ and thus $A'$ is equidimensional. 
\end{proof}

\begin{lem}\label{appendix_bound} If  $A\in \mathcal C(\OO)$ is an integral  domain then there is a number $n(A)$ such that for all finite extensions $L'/L$,  
$A\otimes_{\OO} \OO'$  has at most $n(A)$ minimal primes. 
\end{lem}
\begin{proof}  Since $A$ is $\OO$-torsion free, by Cohen's structure theorem for complete local rings there is a subring $B\subset A$ such that $A$ is finite over $B$ and $B\cong \OO[[x_1, \ldots, x_d]]$, see \cite[Theorem 29.4 (iii)]{mat} and the Remark following it.  
Let $A'= A\otimes_\OO \OO'$ and let $B'=B\otimes_{\OO} \OO'$. 
Then $B'\subset A'$, $A'$ is finite over $B'$ 
and $B'\cong \OO'[[x_1, \ldots, x_d]]$. In particular, $B'$ is an integral  domain. If $\mathfrak P$ is a minimal prime of $A'$ then 
$\mathfrak P\cap A$ is a minimal prime of $A$ by Lemma \ref{base_equidim} and hence $\mathfrak P\cap A=0$, as $A$ is an integral domain. 
Thus $\mathfrak P\cap B=0$. Since $(\mathfrak P\cap B')\cap B=0$ and $B'$ is an integral domain, Lemma \ref{base_equidim}
implies that $\mathfrak P\cap B'=0$. Hence $K(B')\otimes_{B'} A'/\mathfrak P$ is non-zero, where $K(B')$ denotes the fraction field of $B'$. 

Let $\mathfrak P_1, \ldots, \mathfrak P_n$ be the minimal primes of $A'$. 
Since every 
minimal prime ideal is also an associated prime by \cite[Theorem 6.5 (iii)]{mat} for each $i$ there is an injection of $A'$-modules $A'/\mathfrak P_i\hookrightarrow A'$.
The kernel $K$ of the map $\oplus_{i=1}^n A'/\mathfrak P_i \rightarrow A'$ is not supported on any $\mathfrak P_i$. Hence, $\dim K< \dim B'$ and so 
$K(B')\otimes_{B'} K=0$. Hence the map  $\oplus_{i=1}^n K(B')\otimes_{B'} A'/\mathfrak P_i \rightarrow K(B')\otimes_{B'} A'$ is injective and 
we conclude that $A'$ can have at most 
$\dim_{K(B')} K(B')\otimes_{B'} A'$ minimal ideals. The dimension of  $K(B')\otimes_{B'} A'$ as $K(B')$-vector space can be bounded by the number of generators 
of $A'$ as a $B'$-module, which can be bounded by the number of generators of $A$ as a $B$-module. This number does not depend on the 
extension $L'/L$.
\end{proof}

If $A\in \mathcal C(\OO)$ is a domain with fraction field $K(A)$ 
then we have injections:
$$ A\otimes_{\OO} \OO'\hookrightarrow K(A)\otimes_{\OO} \OO'\cong K(A) \otimes_L L'.$$
Since $L'/L$ is a finite separable extension $K(A)\otimes_L L'$ is a finite product of fields and the map induces a bijection between the minimal primes of $A\otimes_{\OO} \OO'$ and the minimal primes of $K(A)\otimes_L L'$. In particular, we obtain:
 
\begin{lem}\label{domain_scalar} If $A\in \mathcal C(\OO)$ is an integral domain such that $A\otimes_{\OO} \OO'$ has only one minimal prime then 
$A\otimes_{\OO} \OO'$ is an integral domain. 
\end{lem}

\begin{lem}\label{base_change} Let $A\in \mathcal C (\OO)$ and let $L'$ be a finite extension of $L$ with the ring of integers $\OO'$ such that 
the number of minimal primes of $A':=A\otimes_{\OO} \OO'$ is maximal. Then $A'/\mathfrak P$ is geometrically integral for all minimal primes $\mathfrak P$ of $A'$.
\end{lem}
\begin{proof} Since $A$ is noetherian, it has only finitely many minimal primes. The minimal primes of $A\otimes_{\OO} \OO'$ are precisely 
the primes lying over the minimal primes of $A$. It follows from Lemma \ref{appendix_bound} that 
the set of numbers of minimal primes of $A\otimes_{\OO} \OO'$ as $L'$ ranges over all finite extensions of $L$ is bounded from above.
Thus we may pick an extension $L'$ such that this number is maximal. If $\mathfrak P$ is a minimal prime of $A'$ and $L''$ is a finite extension 
of $L'$ with the ring of integers $\OO''$ such that  $(A'/\mathfrak P)\otimes_{\OO'} \OO''$ has more than one minimal prime then
we would conclude that $A\otimes_{\OO} \OO''$ has strictly more minimal primes than $A\otimes_{\OO} \OO'$ contradicting 
the choice of $L'$. Hence,  $(A'/\mathfrak P)\otimes_{\OO'} \OO''$ has only one minimal prime and Lemma \ref{domain_scalar} implies that 
it is an integral domain. 
\end{proof}

\begin{proof}[Proof of Lemma \ref{equidim}]  Let $L'$ be a finite extension of $L$ with the ring of integers $\OO'$ such that 
$A'/\mathfrak P$ and $B'/\mathfrak Q$ are geometrically integral 
for all the minimal primes $\mathfrak P$ of $A'$ and all the minimal 
primes $\mathfrak Q$ of $B'$, where $A':= A\otimes_{\OO} \OO'$ and $B':= B\otimes_{\OO} \OO'$. The existence of such extension is
granted by Lemma \ref{base_change}. Since $A'$ and $B'$ are equidimensional by Lemma \ref{base_equidim}, it follows from 
\cite[Lemma 3.3 (2), (5)]{blght} that $A'\wtimes_{\OO'} B'$ is equidimensional. Since 
$$(A\wtimes_{\OO} B)\otimes_{\OO} \OO'\cong A'\wtimes_{\OO'} B'$$
we use Lemma \ref{base_equidim} again to conclude that $A\wtimes_{\OO} B$ is equidimensional. 

Since $B$ is $\OO$-flat, 
$A\wtimes_{\OO} B$ is a flat $A$-algebra, thus $$\dim A\wtimes_{\OO} B=\dim A + \dim k\otimes_A ( A\wtimes_{\OO} B)$$ by \cite[Thm.15.1]{mat}.
The fibre ring $k\otimes_A ( A\wtimes_{\OO} B)$ is isomorphic to $B/\varpi B$, which has dimension equal to $\dim B-1$, as $\varpi$ is $B$-regular.
\end{proof}

\section{Local deformation rings for $2$-adic representations of $G_{\QQ_l}$, $l \neq 2$. By \textsc{Jack Shotton}.}\label{appendix_jack}
Let $l$ and $p$ be distinct primes. Let $L/\QQ_p$ be a finite extension with ring of integers $\Oc$,
uniformiser $\varpi$ and residue field $k$.  Let $F/\QQ_l$ be a finite extension with absolute
Galois group $G_F$, inertia group $I_F$, and wild inertia group $P_F$.  Let $\tilde{P}_F$ be the
kernel of the maximal pro-$l$ quotient of $I_F$.  Let $q$ be the order of the residue field of $F$.
We assume that $L$ contains all $(q^2-1)th$ roots of unity.  Choose a pro-generator $\sigma$ of
$I_F/\tilde{P}_F$ and $\phi \in G_F/\tilde{P}_F$ lifting the arithmetic Frobenius element of
$G_F/I_F$.  Then we have the relation
\begin{equation}\label{eq:relation}\phi \sigma \phi^{-1} = \sigma^q.\end{equation}

If $\rhobar : G_F \rarrow \GL_2(k)$ is a continuous homomorphism, let $R^{\square}_{\rhobar}$ be the universal
framed deformation ring for $\rhobar$ parametrising lifts with coefficients in $\Oc$-algebras.  By
\cite{shotton2} Theorem~2.5, $R^{\square}_{\rhobar}$ is a reduced, $\Oc$-flat complete intersection ring of
relative dimension $4$ over $\Oc$.

If $\tau : I_F \rarrow \GL_2(L)$ is a continuous semisimple representation that extends to $G_F$, let
$R^{\square}_{\rhobar}(\tau)$ be the maximal reduced, $p$-torsion free quotient of
$R^{\square}_{\rhobar}$ such that, for every $\Oc$-algebra homomorphism
$x : R^{\square}_{\rhobar} \rarrow \bar{L}$, the corresponding representation
$\rho_x : G_F \rarrow \GL_2(\bar{L})$ satisfies $(\rho_x\!\mid_{I_F})^{ss} \cong \tau$.

The goal of this appendix is to prove:
\begin{thm} \label{thm:CM}
  For any $\rhobar$ and $\tau$ as above, the ring $R^{\square}_{\rhobar}(\tau)$ is either
  Cohen--Macaulay or zero.
\end{thm}

If $p > 2$, then this is the content of section~5.5 of \cite{shotton}.  If $p = 2$ and
$\rhobar|_{\tilde{P}_F}$ is non-scalar, then the proof of proposition~5.1 of \cite{shotton} shows
that $R^{\square}_{\rhobar}$ is a completed tensor product of deformation rings of characters, all
of whose irreducible components are formally smooth, and that $R^{\square}_{\rhobar}(\tau)$ is an
irreducible component of $R^{\square}_{\rhobar}$; thus $R^{\square}_{\rhobar}(\tau)$ is formally
smooth in this case.  From now on, then, we assume that $p = 2$ and that $\rhobar|_{\tilde{P}_F}$ is
scalar; by twisting, we may and do assume that $\rhobar|_{\tilde{P}_F}$ is trivial.  In this case,
we may list the semisimple inertial types $\tau$ for which $R^{\square}_{\rhobar}(\tau)$ may be non-zero.  They are
determined by the eigenvalues of $\tau(\sigma)$, which must be of $2$-power order and either fixed
or interchanged by raising to the power $q$.  Writing $a = v_2(q-1)$ and $b = v_2(q^2 - 1)$, if
$R^{\square}_{\rhobar}(\tau)$ is non-zero then either
\begin{itemize}
\item $\tau = \tau_\zeta$ is the inertial type in which the eigenvalues of $\tau(\sigma)$ are both
  equal to a $2^a$th root of unity, $\zeta$;
\item $\tau = \tau_{\zeta_1, \zeta_2}$ is the inertial type in which the eigenvalues of
  $\tau(\sigma)$ are equal to distinct $2^a$th roots of unity $\zeta_1$ and $\zeta_2$;
\item $\tau = \tau_{\xi}$ is the inertial type in which the eigenvalues of $\tau(\sigma)$ are equal
  to $\xi$ and $\xi^q$ for $\xi$ a $2^b$th root of unity with $\xi \neq \xi^q$ (equivalently, with
  $\xi$ not a $2^a$th root of unity).
\end{itemize}

We also give a version with fixed determinant:

\begin{cor}
  If $\psi$ is any lift of $\det \rhobar$ to $\Oc^{\times}$ such
  that $\psi|_{I_F} = \det \tau$, let $R^{\square, \psi}_{\rhobar}(\tau)$ be the universal framed
  deformation ring with determinant $\psi$ and type $\tau$.  Then $R^{\square, \psi}_{\rhobar}(\tau)$ is
  Cohen--Macaulay or zero.
\end{cor}

\begin{proof}
  By Theorem~\ref{thm:CM}, $R^{\square}_{\rhobar}(\tau)$ is Cohen--Macaulay.  If we impose a single
   additional equation $\det \rho(\phi) = \psi(\phi)$, then the ring will still be Cohen--Macaulay
  provided that $\det \rho(\phi) - \psi(\phi)$ is a non-zerodivisor --- in other words, that it
  doesn't vanish on any irreducible components of $\Spec R^\square_{\rhobar}(\tau)$.  This is the
  case, since the action of $\GG_m^{\wedge}$ on $\Spec R^\square_{\rhobar}(\tau)$ given by making
  unramified twists preserves irreducible components but varies the determinant.
\end{proof}
Let $\Xc$ be the affine $\Oc$-scheme whose $R$ points, for an $\Oc$-algebra $R$, are pairs
\[\{(\Sigma, \Phi) \in \GL_2(R) \times \GL_2(R) : \Phi \Sigma = \Sigma^q \Phi\}.\] Then $\Xc$ is a
reduced, $\Oc$-flat complete intersection of relative dimension 4 over $\Spec \Oc$ by the proof of
Theorem~2.5 of \cite{shotton2}.  Let $\Ac$ be the coordinate ring of $\Xc$.  We write
$\Sigma = \twomat{1+A}{B}{C}{1+D}$ and $\Phi = \twomat{P}{Q}{R}{T-P}$, so that $\Ac$ is a quotient
of
\[\Sc = \Oc[A,B,C,D,P,Q,R,T][(\det \Sigma)^{-1}, (\det \Phi)^{-1}].\]
For any continuous $\rhobar : G_F \rarrow \GL_2(k)$, the pair of matrices $\rhobar(\sigma)$ and
$\rhobar(\phi)$ give rise to a closed point of $\Xc$, and so a maximal ideal $\mf$ of $\Ac$.  Then
$R^{\square}_{\rhobar} = \Ac^\wedge_{\mf}$.  If $\Cc$ is a conjugacy class in $\GL_2(\bar{L})$, then
there is a unique irreducible component of $\Xc$ such that, for a dense set of geometric points of
that component, the corresponding matrix $\Sigma$ has conjugacy class $\Cc$. This provides a
bijection between the irreducible components of $\Xc$ and the conjugacy classes of $\GL_2(\bar{L})$
that are preserved under the $q$-power map (by \cite{shotton2} Proposition~2.6).
If $\tau$ is one of the above inertial types then we write $\Xc(\tau)$ for the union of those
irreducible components corresponding to conjugacy classes with the same characteristic polynomial as
$\tau(\sigma)$, with the reduced subscheme structure, and $\Ac(\tau)$ for its coordinate ring.
Note that, since $\Xc$ is $\Oc$-flat and $\Xc(\tau)$ is an irreducible component of $\Xc$,
$\Xc(\tau)$ is also $\Oc$-flat, so that $\Ac(\tau)$ is $\varpi$-torsion free.

\begin{lem} \label{lem:type-compl}
  If $\tau = \tau_\zeta, \tau_{\zeta_1, \zeta_2},$ or $\tau_{\xi}$, then
  $\Ac(\tau)_{\mf}^\wedge = R^{\square}_{\rhobar}(\tau)$.
\end{lem}

\begin{proof}
  Since $\Ac$ is $\Oc$-flat and $\Ac(\tau)$ is the quotient of $\Ac$ by an intersection of minimal
  prime ideals, it is also $\Oc$-flat.  Thus $\Ac(\tau)_{\mf}^\wedge$ is also $\Oc$-flat, by
  flatness of localisation and completion. Since $\Ac(\tau)$ is of finite type over a DVR it is
  Nagata by \cite[Tag 0335]{stacks-project}.  Since $\Ac(\tau)$ is reduced, the completion
  $\Ac(\tau)_{\mf}^\wedge$ is also reduced by \cite[Tag 07NZ]{stacks-project}.  The composite map
  $\Ac \rarrow R^{\square}_{\rhobar} \onto R^{\square}_{\rhobar}(\tau)$ factors through a map
  $\Ac(\tau) \rarrow R^{\square}_{\rhobar}(\tau)$, since any function in $\Ac$ that vanishes on all
  $\bar{L}$-points of type $\tau$ must vanish in $R^{\square}_{\rhobar}(\tau)$ by definition. Thus
  we get a surjection
  $\Ac(\tau)_{\mf}^{\wedge} =\Ac(\tau) \otimes_{\Ac} R^{\square}_{\rhobar} \onto
  R^{\square}_{\rhobar}(\tau)$.  However, since $\Ac(\tau)_{\mf}^{\wedge}$ is reduced and
  $\Oc$-torsion free, and has the property that every $\bar{L}$-point gives a Galois representation
  of type $\tau$, this map is an isomorphism by the definition of $R^{\square}_{\rhobar}(\tau)$.
\end{proof}

Let $\bar{\Sc} = \Sc \otimes_{\Oc} k$, $\bar{\Ac} = \Ac \otimes_{\Oc} k$, and
$\bar{\Xc} = \Spec \bar{\Ac}$.  Then the irreducible components of $\bar{\Xc}$ are in bijection with
the conjugacy classes of $\GL_2(\bar{k})$ that are stable under the $q$-power map (again by
\cite{shotton2} Proposition~2.6).  Let $\bar{\Xc}_1$ be the irreducible component corresponding to
the trivial conjugacy class --- this is just the locus where $\Sigma = 1$ --- and let $\bar{\Xc}_N$
be that corresponding to the non-trivial unipotent conjugacy class (we give the irreducible
components the reduced subscheme structure).  Let $I_1$ and $I_N$ be the prime ideals of $\bar{\Sc}$
cutting out $\bar{\Xc}_1$ and $\bar{\Xc}_N$; these correspond to minimal primes of $\bar{\Ac}$.  If
$\tau$ is one of the above inertial types, then we write $I(\tau)$ for the ideal of $\bar{\Sc}$
cutting out $\Ac(\tau) \otimes_\Oc k$.

\begin{lem}\label{lem:primes} The ideals $I_1$ and $I_N$ have generators
  \begin{align*}
    I_1 &= (A,B,C,D) \\
    I_N &= (A^2 + BC, CQ + BR, T, A + D).
  \end{align*}
\end{lem}
\begin{proof} The presentation for $I_1$ is obvious.  For $I_N$, the condition that $\Sigma$ is
  unipotent gives $A + D \in I_N$ and $A^2 + BC \in I_N$.  If $N = \Sigma - 1$, then the relation
  $\Phi \Sigma = \Sigma^q \Phi$ becomes $\Phi N = q N \Phi = N\Phi$ (since we are working mod $2$),
  which implies that $CQ + BR = 0$.  At any closed point of $\bar{\Xc}_N$ where $N \neq 0$, the
  eigenvalues of $\Phi$ must be in the ratio $1:q = 1:1$, and so $T = 0$.  As such closed points are
  dense on $\bar{\Xc}_N$, we see that $T \in I_N$.  Therefore
  \[( A^2 + BC, CQ + BR, T, A + D) \subset I_N.\] The ideal $I = (A^2 + BC, CQ + BR, T, A+D)$ is
  prime of dimension 4; indeed, $\Sc/I$ is isomorphic to a localisation of
  \[\frac{k[A,B,C]}{(A^2 + BC)}[P,Q,R]/(CQ + BR)\] which is easily seen to be a 4-dimensional
  domain.  Thus $I \subset I_N$ are prime ideals of $\bar{\Sc}$ of the same dimension, and so must
  be equal.
\end{proof}

\begin{prop} \label{prop:induced}
 Let $\tau = \tau_\xi$.   Then $I(\tau) = I_N$.  
\end{prop}

\begin{proof} Write $\eta = \xi + \xi^{q} - 2$.  The condition that $\Sigma$ has characteristic
  polynomial $(X - \xi)(X - \xi^{q})$ shows that, on $\Xc(\tau)$, we have the equations
  \begin{align*}
    A + D &= \eta \\
    A(A - \eta) + BC &= \eta.
  \end{align*}
  Using the first of these, we replace $D$ by $\eta - A$ everywhere.  Now, if $x$ is an
  $\bar{L}$-point of $\Xc(\tau)$ corresponding to a pair of matrices $(\Sigma_x, \Phi_x)$, then
  $\Phi_x$ exchanges the $\xi$ and $\xi^q$ eigenspaces of $\Sigma_x$ and so must have trace zero.
  Therefore on $\Xc(\tau)$ we have the equation \[T = 0.\] Lastly, by the Cayley--Hamilton theorem,
  and the fact that \[X^q \equiv \xi + \xi^q - X \mod (X - \xi)(X - \xi^{q}),\] we see that
  $\Sigma^q = \twomat{1 + \eta - A}{-B}{-C}{1 + A}$ on $\Xc(\tau)$. Equating matrix entries in the
  relation $\Phi \Sigma = \Sigma^q \Phi$, and noting that $T = 0$, we obtain one new equation
  \[(2A - \eta)P + BR + CQ = 0.\]
  Thus, letting
  \[ J = (A+D, T, A(A - \eta) + BC - \eta, (2A - \eta)P + BR + CQ)\]
  we obtain a surjection $\Sc/J \onto \Ac(\tau)$, and therefore a surjection
  \[\bar{\Sc}/J \onto \bar{\Ac}(\tau).\]
  As $\eta$ is divisible by $\varpi$, we see that $J + (\varpi) = I_N$, and so we have a surjection
  $\bar{\Sc}/I_N \onto \bar{\Ac}(\tau)$.  This must be an isomorphism since $\bar{\Sc}/I_N$ is a
4-dimensional domain and $\bar{\Ac}(\tau)$ is a non-zero 4-dimensional ring.  Therefore $I_N = I(\tau)$ as
required.
\end{proof}

For the remaining types the following lemma will be useful.  If $R$ is a noetherian ring, $\pf$ is
a minimal prime of $R$, and $M$ is a finitely-generated $R$-module, let $e_R(M, \pf) =
l_{R_{\pf}}(M_{\pf})$ (this is a special case of the Hilbert--Samuel multiplicity).

\begin{lem} \label{lem:mult-completion} Let $f : R \rarrow S$ be a surjection of equidimensional
  rings of the same dimension, and suppose that $R$ is S1 and Nagata.  Let
  $\pf_1, \ldots, \pf_n$ be the minimal primes of $R$.  Suppose that, for $i = 1, \ldots, n$, there
  is a maximal ideal $\mf_i$ of $S$ such that $\pf_i \subset \mf_i$
  but $\pf_j \not \subset \mf_i$ for $i \neq j$.  If, for each $i$, we have
  \[e_R(R, \pf_i) \leq e_{S^{\wedge}_{\mf_i}}(S^{\wedge}_{\mf_i}, \qf_i)\]
  for some minimal prime $\qf_i$ of $S^{\wedge}_{\mf_i}$, then $f$ is an isomorphism.
\end{lem}

\begin{rem}
  For those primes $\pf_i$ such that $e_R(R, \pf_i) = 1$ --- which is all of them if $R$ is reduced
  --- the required inequality is implied simply by the existence of the $\mf_i$.
\end{rem}

\begin{proof} Since $R$ is S1, every associated prime of $\ker f$ is minimal and so, by \cite[Tag
  0311]{stacks-project}, it is enough to show that $f$ induces an isomorphism
  $f_{\pf_i} : R_{\pf_i}\rarrow S_{\pf_i}$ for each $i$.  Since $f$ is surjective and $R_{\pf_i}$ is
  artinian, it is enough to show that $e_R(R,\pf_i) \leq e_R(S, \pf_i)$.  Let $i \in \{1, \ldots,
  n\}$.  Choose $\mf_i$ and $\qf_i$ as in the
  hypotheses of the lemma.
  It is enough to show that for each $i$,
  \[e_R(S, \pf_i) = e_{S^\wedge_{\mf_i}}(S^{\wedge}_{\mf_i}, \qf_i).\] Since $\mf_i$ contains a
  unique minimal prime of $R$, after localising at $\mf_i$ we may assume that $R\rarrow S$ is a
  local map of local rings, and that $\pf_i$ is the unique minimal prime of $R$, and drop $i$ from
  the notation.  The hypothesis that $R$ and $S$ are equidimensional of the same dimension implies
  that $\pf S$ is the unique minimal prime of $S$, which we also denote by $\pf$. We have
  $e_R(S,\pf) = e_S(S, \pf)$ since both are just the length of $S_{\pf}$.  Since
  $S \rarrow S^\wedge$ is flat and $S^\wedge/\pf = (S/\pf)^\wedge$ is reduced because $R$ (and hence
  $S$) is Nagata, \cite[Tag 02M1]{stacks-project} implies that
  $e_S(S, \pf) = e_{S^\wedge}(S^{\wedge}, \qf)$.  So
  \[e_R(S, \pf) = e_S(S, \pf) = e_{S^\wedge}(S^{\wedge}, \qf) \geq
    e_{R}(R, \pf)\] as required.
\end{proof}

The S1 condition holds, in particular, if $R$ is reduced or Cohen--Macaulay, while the Nagata
condition holds if $R$ is of finite type over a field or DVR.

\begin{prop} \label{prop:unip}
  Let $\tau = \tau_\zeta$.   Then
  $$  I(\tau) = I_N \cap I_1 
            = (A+D, AT, BT, CT, A^2 + BC, BR + CQ).$$
\end{prop}

\begin{proof} For simplicity, we twist so that $\zeta = 1$.    Write $N = \Sigma - 1 = \twomat{A}{B}{C}{D}$.  On $\Ac(\tau)$, $\Sigma$ has characteristic
  polynomial $(X - 1)^2$, and so the equations
  \begin{align*}
    A + D &= 0 \\
    A^2 + BC &= 0
  \end{align*}
  hold on $\Ac(\tau)$.  Moreover, since $(\Sigma - 1)^2 = 0$ on $\Ac(\tau)$, by the Cayley--Hamilton theorem we
  have that $\Sigma^q = 1 + q(\Sigma - 1) = 1  + qN$ on $\Ac(\tau)$.  The equation $\Phi \Sigma = \Sigma^q
  \Phi$ becomes $\Phi N = q N \Phi$, and comparing matrix entries we get equations
  \begin{align*}
    qBR - CQ + (q - 1)AP &= 0 \\
    (q+1)QA + B(qT - (q+1)P) &= 0 \\
    (q+1)RA + C(T - (q+1)P)) &= 0 \\
    qCQ - BR + (q-1)A(P-T) &= 0.
  \end{align*}
  Summing the first and fourth of these gives $(q-1)(BR + CQ + A(2P - T)) = 0$; since $\Ac(\tau)$ is
  $(q-1)$-torsion free, we deduce that \[BR + CQ + A(2P - T) = 0\] in $\Ac(\tau)$ and
  can replace the fourth of the above equations by this.

  The ideal cutting out $\Ac(\tau)$ therefore contains the ideal
  \begin{multline*}J = (A+D, A^2 + BC, qBR - CQ + (q - 1)AP,  (q+1)QA + B(qT - (q+1)P), \\ (q+1)RA +
    C(T - (q+1)P), CQ + BR + A(2P-T)).
    \end{multline*}
    Now, the image of $J$ in $\bar{\Sc}$ is 
    \[(A+D, A^2+ BC, BR + CQ, BT, CT, BR + CQ + AT) \]
    which is equal to $(A + D, A^2 + BC, BR + CQ) + I_1 \cap (T) = I_N \cap I_1$.  Therefore there is a surjection
    \[f:\bar{\Sc}/(I_N \cap I_1) \onto \bar{\Ac}(\tau).\] Write
    $\tilde{R} = \bar{\Sc}/(I_N \cap I_1)$.  Then $\tilde{R}$ is reduced with two minimal primes,
    which we also call $I_N$ and $I_1$.  Let $\rho_1 : G_F \rarrow \GL_2(\Oc)$ be diagonal unramified
    with distinct eigenvalues of Frobenius, and let $\rho_N : G_F \rarrow \GL_2(\Oc)$ send
    $\sigma \mapsto \twomat{1}{1}{0}{1}$ and $\phi \mapsto \twomat{q}{0}{0}{1}$.  Let $\mf_1$ and
    $\mf_N$ be the corresponding maximal ideals of $\bar{\Ac}(\tau)$.  Then $I_1 \subset \mf_1$,
    $I_N \not \subset \mf_1$, $I_1 \not \subset \mf_N$ and $I_N \subset \mf_N$, so $f$ is an
    isomorphism by the remark following lemma~\ref{lem:mult-completion}.
  \end{proof}

\begin{prop} \label{prop:PS}
  Let $\tau = \tau_{\zeta_1, \zeta_2}$.   Then
  \[I(\tau) = (A+D, BT, CT, CQ+BR, A^2+BC).\]
\end{prop}

\begin{proof} Write $\mu = \zeta_1 + \zeta_2  - 2$.  The condition that $\Sigma$ has characteristic
  polynomial $(X - \zeta_1)(X - \zeta_2)$ is equivalent to the equations
  \begin{align*}
    A + D &= \mu \\
    A(A - \mu) + BC &= \mu.
  \end{align*}
  As $X^q \equiv X \mod (X - \zeta_1)(X - \zeta_2)$, we have by the Cayley--Hamilton theorem that
  $\Sigma^q = \Sigma$ on $\Ac(\tau)$.  The equation $\Phi \Sigma = \Sigma^q \Phi$ therefore becomes
  $\Phi \Sigma = \Sigma \Phi$, and comparing matrix entries we get three equations (the fourth being
  redundant):
  \begin{align*}
    BR - CQ &= 0 \\
    Q(2A - \mu) &= B(2P - T) \\
    R(2A - \mu) &= C(2P - T).
  \end{align*}
  Let 
  \begin{equation}
  \begin{split}
  J = (A+D - \mu, & A(A-\mu) + BC - \mu, BR-CQ, \\ & Q(2A - \mu) - B(2P-T), R(2A - \mu) - C(2P - T)).
  \end{split}
  \end{equation}
  Let $I$ be the image of $J$ in $\bar{\Sc}$, so that
  \[I = \left(A+D, BT, CT, CQ+BR, A^2 + BC\right).\] We have shown that there is a surjection
  $\Sc/J \onto \Ac(\tau)$, and therefore there is a surjection
  $f: \bar{\Sc}/I \rarrow \bar{\Ac}(\tau)$.  We have to show that $f$ is an isomorphism.  Write
  $\tilde{R} = \bar{\Sc}/I$.

  Then (see the proof of corollary~\ref{cor:CM-rings} below) $\bar{\Sc}/I$ is Cohen--Macaulay, with
  minimal primes $I_1$ and $I_N$, and it is easy to see that $e_{\tilde{R}}(\tilde{R},I_N) = 1$
  while $e_{\tilde{R}}(\tilde{R},I_1) = 2$.

  Let $\rho_1 : G_F \rarrow \GL_2(\Oc)$ be diagonal such that the eigenvalues of $\rho_1(\sigma)$ are
  $\zeta_1$ and $\zeta_2$, and the eigenvalues of $\rho_1(\phi)$ are distinct modulo $\varpi$.  Let
  $\rho_N : G_F \rarrow \GL_2(\Oc)$ send $\sigma \mapsto \twomat{\zeta_1}{1}{0}{\zeta_2}$ and
  $\phi \mapsto \twomat{1}{0}{0}{1}$.  Let $\mf_1$ and $\mf_N$ be the corresponding maximal ideals
  of $\bar{\Ac}(\tau)$.  Then $I_1 \subset \mf_1$, $I_N \not \subset \mf_1$,
  $I_1 \not \subset \mf_N$ and $I_N \subset \mf_N$.  By \cite{shotton}~Proposition 5.3, which
  remains valid when $p = 2$, $R^{\square}_{\rhobar_1}(\tau)$ is formally smooth
  over \[\frac{\Oc[[X-1]]}{(X - \zeta_1)(X - \zeta_2)}.\]  Therefore
  $R^{\square}_{\rhobar_1}(\tau)\otimes k$
    has a unique minimal prime $\qf$ and its multiplicity is 2.  By
    lemmas~\ref{lem:type-compl} and~\ref{lem:mult-completion}, $f$ is an isomorphism.
 \end{proof}

\begin{cor} (of propositions \ref{prop:induced}, \ref{prop:unip} and \ref{prop:PS}) \label{cor:CM-rings} For
  $\tau = \tau_\xi$, $\tau_\zeta$, or $\tau_{\zeta_1, \zeta_2}$, $\Ac(\tau)$ is Cohen--Macaulay.
\end{cor}

\begin{proof} Since $\varpi$ is a regular element of $\Ac(\tau)$, it suffices to prove that
  $\bar{\Ac}(\tau)$ is Cohen--Macaulay.  This can easily be checked in magma; we sketch an
  alternative proof by hand.  If $\tau = \tau_\xi$, then by proposition~\ref{prop:induced},
  $I(\tau) = I_N$.  But $\bar{\Sc}/I_N$ is a complete intersection ring of dimension 4, and
  therefore is Cohen--Macaulay.  If $\tau = \tau_\zeta$, then by proposition~\ref{prop:unip},
  $I(\tau_\xi) = I_1 \cap I_N$.  Now, $\bar{\Sc}/I_1$ and $\bar{\Sc}/I_N$ are Cohen--Macaulay of
  dimension 4 (the latter by the previous case), while $\bar{\Sc}/(I_1 + I_N)$ is regular, and so
  Cohen--Macaulay, of dimension 3.  By exercise~18.13 of
  \cite{eis}, $\bar{\Sc}/(I_1 \cap I_N)$ is also Cohen--Macaulay.
  Finally, if $\tau = \tau_{\zeta_1,\zeta_2}$ then by proposition~\ref{prop:PS}, 
  $I(\tau) = (A + D, A^2 + BC, BR + CQ, BT, CT)$.  Let $I = I(\tau)$.  Since
  $I + (AT) = I_1 \cap I_N$ and $AT\cdot I_1 = 0$, there is an exact sequence of $\bar{\Sc}/I$-modules
  \[\bar{\Sc}/I_1 \overset{AT}{\longrightarrow} \bar{\Sc}/I \longrightarrow \bar{\Sc}/(I_1 \cap I_N)\rarrow 0.\] The first
  map must be injective, since $I_1$ is prime and
  $e_{\bar{\Sc}/I}(\bar{\Sc}/I,I_1) = 2 > 1 = e_{\bar{\Sc}/I}(\bar{\Sc}/(I_1 \cap I_N), I_1)$. Since
  we have shown that $\bar{\Sc}/I_1$ and $\bar{\Sc}/(I_1 \cap I_N)$ are maximal Cohen--Macaulay
  modules over $\bar{\Sc}/I$, so is $\bar{\Sc}/I$ (by \cite{yos} Proposition~1.3).
  \end{proof}

Since $R^{\square}_{\rhobar}(\tau)$ is a completion of $\Ac(\tau)$ by lemma~\ref{lem:type-compl},
and a completion of a Cohen--Macaulay ring is Cohen--Macaulay (by \cite[Tag 07NX]{stacks-project}), we
obtain Theorem~\ref{thm:CM}.
\medskip
 
\textbf{Acknowledgements.} YH was partially supported by National Natural Science Foundation of China Grants
 11688101;  China's Recruitement Program of Global Experts, National Center for  Mathematics and Interdisciplinary Sciences and Hua Loo-Keng Center for Mathematical Sciences of Chinese Academy of Sciences.  VP was partially supported by SFB/TR45 of DFG. The project started when YH visited VP in 2013 supported by SFB/TR45 and he would like to thank the University Duisburg-Essen for the invitation and the hospitality. The authors  would like to thank Jack Shotton for  the appendix to the paper, as well as Toby Gee, James Newton, Shu Sasaki and Jack Thorne for their comments.   We also thank the anonymous referee for their careful reading of the paper and pertinent comments.



\bigskip

\vspace{\baselineskip}

\noindent Morningside Center of Mathematics, Academy of Mathematics and Systems Science,
 Chinese Academy of Sciences, University of the Chinese Academy of Sciences
Beijing, 100190,
China\\
{\it E-mail:} {\ttfamily yhu@amss.ac.cn}\\

\noindent Fakult\"at f\"ur Mathematik, Universit\"at Duisburg-Essen, Thea-Leymann-Str. 9, 45127 Essen, Germany\\
{\it E-mail:} {\ttfamily paskunas@uni-due.de} \\

\noindent Department of Mathematics, 
University of Chicago, 
5734 S University Avenue, 
Chicago, 
IL 60615, 
USA\\
{\it E-mail:} {\ttfamily jshotton@uchicago.edu}
\end{document}